\documentclass[reqno, 12pt]{amsart}
\usepackage{amsmath}
\usepackage{amssymb}
\usepackage{amsfonts}
\usepackage[parfill]{parskip}
\usepackage{graphicx}
\usepackage{epstopdf}
\usepackage{subfigure}

\newtheorem{theorem}{Theorem}
\theoremstyle{plain}
\newtheorem{acknowledgement}{Acknowledgement}

\newtheorem{corollary}{Corollary}

\newtheorem{definition}{Definition}

\newtheorem{lemma}{Lemma}

\newtheorem{proposition}{Proposition}
\newtheorem{remark}{Remark}

\numberwithin{equation}{section}
 \numberwithin{theorem}{section}
 \numberwithin{proposition}{section}
 \numberwithin{remark}{section}
 \numberwithin{definition}{section}
 \numberwithin{lemma}{section}
 \numberwithin{corollary}{section}
 \numberwithin{example}{section}
 \numberwithin{claim}{section}
\input{tcilatex}

\begin{document}
\title[Convergence rates in deterministic homogenization]{Approximation of
homogenized coefficients in deterministic homogenization and convergence
rates in the asymptotic almost periodic setting}
\author{Willi J\"{a}ger}
\curraddr{W. J\"{a}ger, Interdisciplinary Center for Scientific Computing
(IWR), University of Heidelberg, Im Neuenheimer Feld 205, 69120 Heidelberg,
Germany.}
\email{wjaeger@iwr.uni-heidelberg.de}
\author{Antoine Tambue}
\curraddr{A. Tambue, a) Department of Computing Mathematics and Physics,
Western Norway University of Applied Sciences, Inndalsveien 28, Bergen 5063,
Norway; b) Center for Research in Computational and Applied Mechanics
(CERECAM), and Department of Mathematics and Applied Mathematics, University
of Cape Town, Rondebosch 7701, South Africa; c) The African Institute for
Mathematical Sciences (AIMS) and Stellenbosch University, 6--8 Melrose Road,
Muizenberg 7945, South Africa}
\email{antonio@aims.ac.za}
\author{Jean Louis Woukeng}
\curraddr{J. L. Woukeng, Interdisciplinary Center for Scientific Computing
(IWR), University of Heidelberg, Im Neuenheimer Feld 205, 69120 Heidelberg,
Germany.}
\email{jwoukeng@yahoo.fr}
\date{2022}
\dedicatory{Dedicated to the memory of V.V. Zhikov, 1940-2017.}

\begin{abstract}
For a homogenization problem associated to a linear elliptic operator, we
prove the existence of a distributional corrector and we find an
approximation scheme for the homogenized coefficients. We also study the
convergence rates in the asymptotic almost periodic setting, and we show
that the rates of convergence for the zero order approximation, are near
optimal. The results obtained constitute a step towards the numerical
implementation of results from the deterministic homogenization theory
beyond the periodic setting. To illustrate this, numerical simulations based
on finite volume method are provided to sustain our theoretical results.
\end{abstract}

\maketitle

\section{Introduction}

The purpose of this work is to establish the existence of a distributional
corrector in the deterministic homogenization theory for a family of second
order elliptic equations in divergence form with rapidly oscillating
coefficients, and find an approximation scheme for the homogenized
coefficients, without smoothness assumption on the coefficients. Under
additional condition, we also study the convergence rates in the asymptotic
almost periodic setting. We start with the statement of the problem (\ref%
{1.1}).

Let $\mathcal{A}$ be an algebra with mean value on $\mathbb{R}^{d}$, that
is, a closed subalgebra of the $\mathcal{C}^{\ast }$-algebra of bounded
uniformly continuous real-valued functions on $\mathbb{R}^{d}$, $\mathrm{BUC}%
(\mathbb{R}^{d})$, which contains the constants, is translation invariant
and is such that any of its elements possesses a mean value in the following
sense: for every $u\in \mathcal{A}$, the sequence $(u^{\varepsilon
})_{\varepsilon >0}$ ($u^{\varepsilon }(x)=u(x/\varepsilon )$) weakly$\ast $%
-converges in $L^{\infty }(\mathbb{R}^{d})$ to some real number $M(u)$
(called the mean value of $u$) as $\varepsilon \rightarrow 0$. The mean
value expresses as 
\begin{equation}
M(u)=\lim_{R\rightarrow \infty }%
\mathchoice {{\setbox0=\hbox{$\displaystyle{\textstyle -}{\int}$ } \vcenter{\hbox{$\textstyle -$
}}\kern-.6\wd0}}{{\setbox0=\hbox{$\textstyle{\scriptstyle -}{\int}$ } \vcenter{\hbox{$\scriptstyle -$
}}\kern-.6\wd0}}{{\setbox0=\hbox{$\scriptstyle{\scriptscriptstyle -}{\int}$
} \vcenter{\hbox{$\scriptscriptstyle -$
}}\kern-.6\wd0}}{{\setbox0=\hbox{$\scriptscriptstyle{\scriptscriptstyle
-}{\int}$ } \vcenter{\hbox{$\scriptscriptstyle -$ }}\kern-.6\wd0}}%
\!\int_{B_{R}}u(y)dy\text{ for }u\in \mathcal{A}  \label{0.1}
\end{equation}%
where we have set $%
\mathchoice {{\setbox0=\hbox{$\displaystyle{\textstyle
-}{\int}$ } \vcenter{\hbox{$\textstyle -$
}}\kern-.6\wd0}}{{\setbox0=\hbox{$\textstyle{\scriptstyle -}{\int}$ }
\vcenter{\hbox{$\scriptstyle -$
}}\kern-.6\wd0}}{{\setbox0=\hbox{$\scriptstyle{\scriptscriptstyle -}{\int}$
} \vcenter{\hbox{$\scriptscriptstyle -$
}}\kern-.6\wd0}}{{\setbox0=\hbox{$\scriptscriptstyle{\scriptscriptstyle
-}{\int}$ } \vcenter{\hbox{$\scriptscriptstyle -$ }}\kern-.6\wd0}}%
\!\int_{B_{R}}=\frac{1}{\left\vert B_{R}\right\vert }\int_{B_{R}}$.

For $1\leq p<\infty $, we define the Marcinkiewicz space $\mathfrak{M}^{p}(%
\mathbb{R}^{d})$ to be the set of functions $u\in L_{loc}^{p}(\mathbb{R}%
^{d}) $ such that 
\begin{equation*}
\underset{R\rightarrow \infty }{\lim \sup }%
\mathchoice {{\setbox0=\hbox{$\displaystyle{\textstyle -}{\int}$ } \vcenter{\hbox{$\textstyle -$
}}\kern-.6\wd0}}{{\setbox0=\hbox{$\textstyle{\scriptstyle -}{\int}$ } \vcenter{\hbox{$\scriptstyle -$
}}\kern-.6\wd0}}{{\setbox0=\hbox{$\scriptstyle{\scriptscriptstyle -}{\int}$
} \vcenter{\hbox{$\scriptscriptstyle -$
}}\kern-.6\wd0}}{{\setbox0=\hbox{$\scriptscriptstyle{\scriptscriptstyle
-}{\int}$ } \vcenter{\hbox{$\scriptscriptstyle -$ }}\kern-.6\wd0}}%
\!\int_{B_{R}}\left\vert u(y)\right\vert ^{p}dy<\infty .
\end{equation*}%
Then $\mathfrak{M}^{p}(\mathbb{R}^{d})$ is a complete seminormed space
endowed with the seminorm 
\begin{equation*}
\left\Vert u\right\Vert _{p}=\left( \underset{R\rightarrow \infty }{\lim
\sup }%
\mathchoice {{\setbox0=\hbox{$\displaystyle{\textstyle -}{\int}$ } \vcenter{\hbox{$\textstyle -$
}}\kern-.6\wd0}}{{\setbox0=\hbox{$\textstyle{\scriptstyle -}{\int}$ } \vcenter{\hbox{$\scriptstyle -$
}}\kern-.6\wd0}}{{\setbox0=\hbox{$\scriptstyle{\scriptscriptstyle -}{\int}$
} \vcenter{\hbox{$\scriptscriptstyle -$
}}\kern-.6\wd0}}{{\setbox0=\hbox{$\scriptscriptstyle{\scriptscriptstyle
-}{\int}$ } \vcenter{\hbox{$\scriptscriptstyle -$ }}\kern-.6\wd0}}%
\!\int_{B_{R}}\left\vert u(y)\right\vert ^{p}dy\right) ^{1/p}.
\end{equation*}%
We denote by $B_{\mathcal{A}}^{p}(\mathbb{R}^{d})$ ($1\leq p<\infty $) the
closure of $\mathcal{A}$ in $\mathfrak{M}^{p}(\mathbb{R}^{d})$. Then for any 
$u\in B_{\mathcal{A}}^{p}(\mathbb{R}^{d})$ we have that 
\begin{equation}
\left\Vert u\right\Vert _{p}=\left( \lim_{R\rightarrow \infty }%
\mathchoice {{\setbox0=\hbox{$\displaystyle{\textstyle -}{\int}$ } \vcenter{\hbox{$\textstyle -$
}}\kern-.6\wd0}}{{\setbox0=\hbox{$\textstyle{\scriptstyle -}{\int}$ } \vcenter{\hbox{$\scriptstyle -$
}}\kern-.6\wd0}}{{\setbox0=\hbox{$\scriptstyle{\scriptscriptstyle -}{\int}$
} \vcenter{\hbox{$\scriptscriptstyle -$
}}\kern-.6\wd0}}{{\setbox0=\hbox{$\scriptscriptstyle{\scriptscriptstyle
-}{\int}$ } \vcenter{\hbox{$\scriptscriptstyle -$ }}\kern-.6\wd0}}%
\!\int_{B_{R}}\left\vert u(y)\right\vert ^{p}dy\right) ^{\frac{1}{p}%
}=(M(\left\vert u\right\vert ^{p}))^{\frac{1}{p}}.  \label{0.2}
\end{equation}%
Consider the space $B_{\mathcal{A}}^{1,p}(\mathbb{R}^{d})=\{u\in B_{\mathcal{%
A}}^{p}(\mathbb{R}^{d}):\nabla _{y}u\in (B_{\mathcal{A}}^{p}(\mathbb{R}%
^{d}))^{d}\}$ which is a complete seminorned space with respect to the
seminorm 
\begin{equation*}
\left\Vert u\right\Vert _{1,p}=\left( \left\Vert u\right\Vert
_{p}^{p}+\left\Vert \nabla _{y}u\right\Vert _{p}^{p}\right) ^{\frac{1}{p}}.
\end{equation*}%
The Banach counterpart of the previous spaces are defined as follows. We set 
$\mathcal{B}_{\mathcal{A}}^{p}(\mathbb{R}^{d})=B_{\mathcal{A}}^{p}(\mathbb{R}%
^{d})/\mathcal{N}$ where $\mathcal{N}=\{u\in B_{\mathcal{A}}^{p}(\mathbb{R}%
^{d}):\left\Vert u\right\Vert _{p}=0\}$. We define $\mathcal{B}_{\mathcal{A}%
}^{1,p}(\mathbb{R}^{d})$ mutatis mutandis: replace $B_{\mathcal{A}}^{p}(%
\mathbb{R}^{d})$ by $\mathcal{B}_{\mathcal{A}}^{p}(\mathbb{R}^{d})$ and $%
\partial /\partial y_{i}$ by $\overline{\partial }/\partial y_{i}$, where $%
\overline{\partial }/\partial y_{i}$ is defined by 
\begin{equation}
\frac{\overline{\partial }}{\partial y_{i}}(u+\mathcal{N}):=\frac{\partial u%
}{\partial y_{i}}+\mathcal{N}\text{ for }u\in B_{\mathcal{A}}^{1,p}(\mathbb{R%
}^{d}).  \label{0.3}
\end{equation}%
It is important to note that $\overline{\partial }/\partial y_{i}$ is also
defined as the infinitesimal generator in the $i$th direction coordinate of
the strongly continuous group $\mathcal{T}(y):\mathcal{B}_{\mathcal{A}}^{p}(%
\mathbb{R}^{d})\rightarrow \mathcal{B}_{\mathcal{A}}^{p}(\mathbb{R}^{d});\ 
\mathcal{T}(y)(u+\mathcal{N})=u(\cdot +y)+\mathcal{N}$. Let us denote by $%
\varrho :B_{\mathcal{A}}^{p}(\mathbb{R}^{d})\rightarrow \mathcal{B}_{%
\mathcal{A}}^{p}(\mathbb{R}^{d})=B_{\mathcal{A}}^{p}(\mathbb{R}^{d})/%
\mathcal{N}$, $\varrho (u)=u+\mathcal{N}$, the canonical surjection. Remark: 
$u\in B_{\mathcal{A}}^{1,p}(\mathbb{R}^{d})$ implies $\varrho (u)\in 
\mathcal{B}_{\mathcal{A}}^{1,p}(\mathbb{R}^{d})$ and observing (\ref{0.3}), $%
\frac{\overline{\partial }\varrho (u)}{\partial y_{i}}=\varrho \left( \frac{%
\partial u}{\partial y_{i}}\right) $.

We assume in the sequel that the algebra $\mathcal{A}$ is ergodic, that is,
any $u\in \mathcal{B}_{\mathcal{A}}^{p}(\mathbb{R}^{d})$ that is invariant
under $(\mathcal{T}(y))_{y\in \mathbb{R}^{d}}$ is a constant in $\mathcal{B}%
_{\mathcal{A}}^{p}(\mathbb{R}^{d})$, i.e., if $\left\Vert \mathcal{T}%
(y)u-u\right\Vert _{p}=0$ for every $y\in \mathbb{R}^{d}$, then $\left\Vert
u-c\right\Vert _{p}=0$, $c$ a constant. Let us also recall the following
property \cite{CMP, NA}:

\begin{itemize}
\item[(\textbf{1)}] The mean value $M$ viewed as defined on $\mathcal{A}$,
extends by continuity to a non negative continuous linear form (still
denoted by $M$) on $B_{\mathcal{A}}^{p}(\mathbb{R}^{d})$. For each $u\in B_{%
\mathcal{A}}^{p}(\mathbb{R}^{d})$ and all $a\in \mathbb{R}^{d}$, we have $%
M(u(\cdot +a))=M(u)$, and $\left\Vert u\right\Vert _{p}=\left( M(\left\vert
u\right\vert ^{p})\right) ^{1/p}$.
\end{itemize}

To the space $B_{\mathcal{A}}^{p}(\mathbb{R}^{d})$ we also attach the
following \textit{corrector} space 
\begin{equation*}
B_{\#\mathcal{A}}^{1,p}(\mathbb{R}^{d})=\{u\in W_{loc}^{1,p}(\mathbb{R}%
^{d}):\nabla u\in B_{\mathcal{A}}^{p}(\mathbb{R}^{d})^{d}\text{ and }%
M(\nabla u)=0\}\text{.}
\end{equation*}%
In $B_{\#\mathcal{A}}^{1,p}(\mathbb{R}^{d})$ we identify two elements by
their gradients: $u=v$ in $B_{\#\mathcal{A}}^{1,p}(\mathbb{R}^{d})$ iff $%
\nabla (u-v)=0$, i.e. $\left\Vert \nabla (u-v)\right\Vert _{p}=0$. We equip $%
B_{\#\mathcal{A}}^{1,p}(\mathbb{R}^{d})$ with the gradient norm $\left\Vert
u\right\Vert _{\#,p}=\left\Vert \nabla u\right\Vert _{p}$ and obtain a
Banach space \cite[Theorem 3.12]{Casado} containing $B_{\mathcal{A}}^{1,p}(%
\mathbb{R}^{d})$.

We recall the $\Sigma $-convergence. A sequence $(u_{\varepsilon
})_{\varepsilon >0}\subset L^{p}(\Omega )$ ($1\leq p<\infty $) is said to:

\begin{itemize}
\item[(i)] \emph{weakly }$\Sigma $\emph{-converge} in $L^{p}(\Omega )$ to $%
u_{0}\in L^{p}(\Omega ;\mathcal{B}_{\mathcal{A}}^{p}(\mathbb{R}^{d}))$ if,
as $\varepsilon \rightarrow 0$, 
\begin{equation}
\int_{\Omega }u_{\varepsilon }(x)f\left( x,\frac{x}{\varepsilon }\right)
dx\rightarrow \int_{\Omega }M(u_{0}(x,\cdot )f(x,\cdot ))dx  \label{3.1}
\end{equation}%
for any $f\in L^{p^{\prime }}(\Omega ;\mathcal{A})$ ($p^{\prime }=p/(p-1)$)$%
; $

\item[(ii)] \emph{strongly }$\Sigma $\emph{-converge} in $L^{p}(\Omega )$ to 
$u_{0}\in L^{p}(\Omega ;\mathcal{B}_{\mathcal{A}}^{p}(\mathbb{R}^{d}))$ if (%
\ref{3.1}) holds and further $\left\Vert u_{\varepsilon }\right\Vert
_{L^{p}(\Omega )}\rightarrow \left\Vert u_{0}\right\Vert _{L^{p}(\Omega ;%
\mathcal{B}_{\mathcal{A}}^{p}(\mathbb{R}^{d}))}$.
\end{itemize}

We denote (i) by "$u_{\varepsilon }\rightarrow u_{0}$ in $L^{p}(\Omega )$%
-weak $\Sigma $", and (ii) by "$u_{\varepsilon }\rightarrow u_{0}$ in $%
L^{p}(\Omega )$-strong $\Sigma $".

The main properties of the above concept are:

\begin{itemize}
\item Every bounded sequence in $L^{p}(\Omega )$ ($1<p<\infty $) possesses a
subsequence that weakly $\Sigma $-converges in $L^{p}(\Omega )$.

\item If $(u_{\varepsilon })_{\varepsilon \in E}$ is a bounded sequence in $%
W^{1,p}(\Omega )$, then there exist a subsequence $E^{\prime }$ of $E$ and a
couple $(u_{0},u_{1})\in W^{1,p}(\Omega )\times L^{p}(\Omega ;B_{\#\mathcal{A%
}}^{1,p}(\mathbb{R}^{d}))$ such that 
\begin{align*}
u_{\varepsilon }& \rightarrow u_{0}\text{ in }W^{1,p}(\Omega )\text{-weak} \\
\frac{\partial u_{\varepsilon }}{\partial x_{j}}& \rightarrow \frac{\partial
u_{0}}{\partial x_{j}}+\frac{\partial u_{1}}{\partial y_{j}}\text{ in }%
L^{p}(\Omega )\text{-weak }\Sigma \ \ (1\leq j\leq d)
\end{align*}

\item If $u_{\varepsilon }\rightarrow u_{0}$ in $L^{p}(\Omega )$-weak $%
\Sigma $ and $v_{\varepsilon }\rightarrow v_{0}$ in $L^{q}(\Omega )$-strong $%
\Sigma $, then $u_{\varepsilon }v_{\varepsilon }\rightarrow u_{0}v_{0}$ in $%
L^{r}(\Omega )$-weak $\Sigma $, where $1\leq p,q,r<\infty $ and $\frac{1}{p}+%
\frac{1}{q}=\frac{1}{r}$.
\end{itemize}

Our aim is to study the following problem 
\begin{equation}
-\nabla \cdot \left( A\left( x,\frac{x}{\varepsilon }\right) \nabla
u_{\varepsilon }\right) =f\text{\ in }\Omega \text{, }u_{\varepsilon }\in
H_{0}^{1}(\Omega )  \label{1.1}
\end{equation}%
where $\varepsilon >0$ is a small parameter, $f\in L^{2}(\Omega )$, $\Omega $
is an open bounded set of $\mathbb{R}^{d}$ (integer $d\geq 1$) with smooth
boundary $\partial \Omega $, and $A\in \mathcal{C}(\overline{\Omega }%
;L^{\infty }(\mathbb{R}^{d})^{d\times d})$ is a symmetric matrix satisfying 
\begin{equation}
\alpha \left\vert \lambda \right\vert ^{2}\leq A(x,y)\lambda \cdot \lambda
\leq \beta \left\vert \lambda \right\vert ^{2}\text{ for all }(x,\lambda
)\in \overline{\Omega }\times \mathbb{R}^{d}\text{ and a.e. }y\in \mathbb{R}%
^{d};  \label{1.2}
\end{equation}%
\begin{equation}
A(x,\cdot )\in (B_{\mathcal{A}}^{2}(\mathbb{R}^{d}))^{d\times d}\text{ for
all }x\in \overline{\Omega }  \label{1.3}
\end{equation}%
where $\alpha $ and $\beta $ are two positive real numbers.

It is well-known that under assumptions (\ref{1.2}), problem (\ref{1.1})
uniquely determines a function $u_{\varepsilon }\in H_{0}^{1}(\Omega )$.
Under the additional assumption (\ref{1.3}), the following result holds.

\begin{theorem}
\label{t1.1}There exists $u_{0}\in H_{0}^{1}(\Omega )$ such that $%
u_{\varepsilon }\rightarrow u_{0}$ weakly in $H_{0}^{1}(\Omega )$ and
strongly in $L^{2}(\Omega )$ (as $\varepsilon \rightarrow 0$) and $u_{0}$
solves uniquely the problem 
\begin{equation}
-\nabla \cdot (A^{\ast }(x)\nabla u_{0})=f\text{ in }\Omega ,  \label{1.4}
\end{equation}%
$A^{\ast }$ being the homogenized matrix defined by 
\begin{equation}
A^{\ast }(x)=M\left( A(x,\cdot )(I_{d}+\nabla _{y}\chi (x,\cdot ))\right)
\label{1.5}
\end{equation}%
where, $\chi =(\chi _{j})_{1\leq j\leq d}\in \mathcal{C}(\overline{\Omega }%
;B_{\#\mathcal{A}}^{1,2}(\mathbb{R}^{d})^{d})$ is such that, for any $x\in
\Omega $, $\chi _{j}(x,\cdot )$ is the unique solution (up to an additive
constant depending on $x$) of the problem 
\begin{equation}
\nabla _{y}\cdot \left( A(x,\cdot )(e_{j}+\nabla _{y}\chi _{j}(x,\cdot
))\right) =0\text{ in }\mathbb{R}^{d}.  \label{1.6}
\end{equation}%
If we set $u_{1}(x,y)=\nabla u_{0}(x)\chi (x,y)=\sum_{i=1}^{d}\frac{\partial
u_{0}}{\partial x_{i}}(x)\chi _{i}(x,y)$ and assume that $u_{1}\in
H^{1}(\Omega ;\mathcal{A}^{1})$ ($\mathcal{A}^{1}=\{v\in \mathcal{A}:\nabla
_{y}v\in (\mathcal{A})^{d}\}$), then, as $\varepsilon \rightarrow 0$, 
\begin{equation}
u_{\varepsilon }-u_{0}-\varepsilon u_{1}^{\varepsilon }\rightarrow 0\text{
in }H^{1}(\Omega )\text{ strongly}  \label{1.7}
\end{equation}%
where $u_{1}^{\varepsilon }(x)=u_{1}(x,x/\varepsilon )$ for a.e. $x\in
\Omega $.
\end{theorem}

\begin{remark}
\label{r1.1'}\emph{Problem (\ref{1.6}) is the }corrector problem\emph{. It
helps to obtain a first order approximation }$u_{\varepsilon }(x)\approx
u_{0}(x)+\varepsilon u_{1}(x,x/\varepsilon )$\emph{\ of }$u_{\varepsilon }$%
\emph{\ as seen in (\ref{1.7}). Its solvability is addressed in the
following result, which is the first main result of this work.}
\end{remark}

\begin{theorem}
\label{t4.1}Let $\xi \in \mathbb{R}^{d}$ and $x\in \overline{\Omega }$ be
fixed. There exists a unique (up to an additive function of $x$) function $%
v_{\xi }\in \mathcal{C}(\overline{\Omega };H_{loc}^{1}(\mathbb{R}^{d}))$
such that $\nabla _{y}v_{\xi }\in \mathcal{C}(\overline{\Omega };B_{\mathcal{%
A}}^{2}(\mathbb{R}^{d})^{d})$ and $M(\nabla _{y}v_{\xi }(x,\cdot ))=0$,
which solves the equation 
\begin{equation}
\nabla _{y}\cdot \left( A(x,\cdot )(\xi +\nabla _{y}v_{\xi }(x,\cdot
))\right) =0\text{ in }\mathbb{R}^{d}.  \label{4.1}
\end{equation}
\end{theorem}

The proof of Theorem \ref{t4.1} will be obtained as a consequence of Lemma %
\ref{l4.1} in Section 2 below. The progress compared to the previously known
results exists in the solution of the corrector problem: it is obtained by
approximation with distributional solutions of partial differential
equations in sufficiently large balls. Since the approximation can be
quantitatively controlled, this method also provides a basis for the
numerical calculation. Theorem \ref{t4.1} is well known in the random
stationary ergodic environment. However for the general deterministic
setting, we believe that a detailed proof must be provided since it also
covers the non ergodic algebras framework.

The next step consists in finding an approximation scheme for the
homogenized matrix $A^{\ast }$ (see (\ref{1.5})). This problem has been
solved (for (\ref{1.1})) in the periodic setting, since under the periodic
assumption, the corrector problem is posed on a bounded domain (namely the
periodic cell $Y=(0,1)^{d}$) since in that case, the solution $\chi _{j}$ is
periodic. A huge contrast between the periodic setting and the general
deterministic setting (as considered in this work) is that in the latter,
the corrector problem is posed on the whole space $\mathbb{R}^{d}$, and
cannot be reduced (as in the periodic framework) to a problem on a bounded
domain. As a result, the solution of the corrector problem (\ref{1.6}) (and
hence the homogenized matrix which depends on this solution) can not be
computed directly. Therefore, as in the random setting (see e.g. \cite%
{BP2004}), truncations of (\ref{1.6}) must be considered, particularly on
large domains $(-R,R)^{d}$ with appropriate boundary conditions, and the
homogenized coefficients will therefore be captured in the asymptotic
regime. This is done in Theorem \ref{t3.1} (see Section 3). We then find the
rate of convergence for the approximation scheme (see Theorem \ref{t3.2}).
It is natural to determine the convergence rates for the approximation (\ref%
{1.7}) setting in two cases:

\begin{itemize}
\item[1)] the asymptotic periodic one represented by the algebra $\mathcal{A}%
=\mathcal{C}_{0}(\mathbb{R}^{d})+\mathcal{C}_{per}(Y)$;

\item[2)] the asymptotic almost periodic one represented by the algebra $%
\mathcal{A}=\mathcal{C}_{0}(\mathbb{R}^{d})+AP(\mathbb{R}^{d})$.
\end{itemize}

In case 1), the corrector function $\chi _{j}(x,\cdot )$ (solution of (\ref%
{1.6})) belongs to the Sobolev-Besicovitch space $B_{\mathcal{A}}^{1,2}(%
\mathbb{R}^{d})$ associated to the algebra $\mathcal{A}$ and is bounded in $%
L^{\infty }(\mathbb{R}^{d})$. As a result, we proceed as in the well-known
periodic setting. In contrast with case 1), the corrector function in case
2) does not (in general) belong to the associated Sobolev-Besicovitch space $%
B_{\mathcal{A}}^{1,2}(\mathbb{R}^{d})$, but rather to $B_{\#\mathcal{A}%
}^{1,2}(\mathbb{R}^{d})$. So information is available mainly for the
gradient of the corrector. To address this issue, we use the approximate
corrector $\chi _{T,j}$, distributional solution to $-\nabla \cdot
A(e_{j}+\nabla \chi _{T,j})+T^{-2}\chi _{T,j}=0$ in $\mathbb{R}^{d}$, which
belongs to $B_{\mathcal{A}}^{1,2}(\mathbb{R}^{d})$ as shown in Section 2.
This leads to the following result, which is one of the main result of the
work.

\begin{theorem}
\label{t1.4}Let $\Omega $ be a $\mathcal{C}^{1,1}$ bounded domain in $%
\mathbb{R}^{d}$. Suppose that the matrix $A(x,y)\equiv A(y)$ and is
asymptotic almost periodic. Assume that $A$ satisfies \emph{(\ref{1.2})}.
For $f\in L^{2}(\Omega )$, let $u_{\varepsilon }$ and $u_{0}$ be the weak
solutions of Dirichlet problems \emph{(\ref{1.1})} and \emph{(\ref{1.4})}
respectively. Then there exists a function $\eta :(0,1]\rightarrow \lbrack
0,\infty )$ depending on $A$ with $\lim_{t\rightarrow 0}\eta (t)=0$ such
that 
\begin{equation}
\left\Vert u_{\varepsilon }-u_{0}-\varepsilon \chi _{T}^{\varepsilon }\nabla
u_{0}\right\Vert _{H^{1}(\Omega )}\leq C\eta (\varepsilon )\left\Vert
f\right\Vert _{L^{2}(\Omega )}  \label{Eq03}
\end{equation}%
and 
\begin{equation}
\left\Vert u_{\varepsilon }-u_{0}\right\Vert _{L^{2}(\Omega )}\leq C\left[
\eta (\varepsilon )\right] ^{2}\left\Vert f\right\Vert _{L^{2}(\Omega )}
\label{Eq02}
\end{equation}%
where $T=\varepsilon ^{-1}$ and $\chi _{T}$ is the approximate corrector
defined by \emph{(\ref{11.5})}, and $C=C(\Omega ,A,d)$.
\end{theorem}

The precise convergence rates in case 1) are presented in the following
result.

\begin{theorem}
\label{t5.1}Suppose that $A$ is asymptotic periodic and satisfies
ellipticity conditions \emph{(\ref{1.2})} and \emph{(\ref{2.2})}. Assume $%
\Omega $, $f$, $u_{\varepsilon }$ and $u_{0}$ are as in Theorem \emph{\ref%
{t1.4}}. Denoting by $\chi $ the corrector defined by \emph{(\ref{1.6})},
there exists $C=C(\Omega ,A,d)>0$ such that 
\begin{equation}
\left\Vert u_{\varepsilon }-u_{0}-\varepsilon \chi ^{\varepsilon }\nabla
u_{0}\right\Vert _{H^{1}(\Omega )}\leq C\varepsilon ^{\frac{1}{2}}\left\Vert
f\right\Vert _{L^{2}(\Omega )}  \label{5.8}
\end{equation}%
and 
\begin{equation}
\left\Vert u_{\varepsilon }-u_{0}\right\Vert _{L^{2}(\Omega )}\leq
C\varepsilon \left\Vert f\right\Vert _{L^{2}(\Omega )}.  \label{1.14}
\end{equation}
\end{theorem}

Theorem \ref{t5.1} can be obtained as a special case of Theorem \ref{t1.4}.
However we provide an independent proof since we do not need the approximate
corrector in this special situation. Estimate (\ref{1.14}) is optimal.

The above results generalize the well known ones in the periodic and the
uniformly almost periodic settings as considered in \cite{Shen}. In Theorem %
\ref{t5.1} we assume that the matrix $A$ has the form $A=A_{0}+A_{per}$
where $A_{0}$ has entries in $L^{2}(\Omega )$ and $A_{per}$ is periodic. In
Theorem \ref{t1.4}, we do not make any restriction on $A_{0}$ as above.
Also, the estimate (\ref{Eq02}) is near optimal. The assumptions will be
made precise in the latter sections.

The problem considered in Theorems \ref{t1.4} and \ref{t5.1} has been
firstly addressed in the periodic framework by Avellaneda and Lin \cite{AL87}
(see also \cite{Jikov}), and in the random setting (that is, for second
order linear elliptic equations with random coefficients) by Yurinskii \cite%
{Yu86}, Pozhidaev and Yurinskii \cite{Po-Yu89}, and Bourgeat and Piatnitski 
\cite{BP2004} (see also a recent series of works by Gloria and Otto \cite%
{Gloria, GNO14, GNO15}, and the recent monograph \cite{Armstrong1}).
Although it is shown in \cite{24'''} that deterministic homogenization
theory can be seen as a special case of random homogenization theory at
least as far as the qualitative study is concerned, we can not expect to use
this random formulation to address the issues of rate of convergence in the
deterministic setting. Indeed, in the random framework, the rate of
convergence relies systematically on the \emph{uniform mixing} property (see
e.g. \cite{BP2004, Po-Yu89, Yu86}) of the coefficients of the equation. As
proved by Bondarenko et al. \cite{BBMM05}, the almost periodic operators do
not satisfy the uniform mixing property. As a result, we can not use the
random framework to address the issue in the general deterministic setting.
We therefore need to elaborate a new framework for solving the underlying
problem. Beyond the periodic (but non-random) setting Kozlov \cite{Koz79}
determined the rates of convergence in almost periodic homogenization by
using almost periodic coefficients satisfying a \textit{frequency condition}
(see e.g. (\ref{FC})). In the same vein, Bondarenko et al. \cite{BBMM05}
derived the rates of convergence by considering a perturbation of periodic
coefficients (in dimension $d=1$). The very first works that use the general
almost periodicity assumption are a recent series of work by Shen et al. 
\cite{AS2016, Shen, Shen1} in which they treated second order linear
elliptic systems in divergence form. They used approximate correctors to
derive the rates of convergence. A reason to use approximate correctors is
the lack of sufficient knowledge on the corrector itself. Indeed in that
case it is known that the gradient of the corrector is almost periodic.
However it is not known in general whether the corrector itself is almost
periodic. Under certain conditions, it is shown in \cite{Armstrong, Shen1}
that the corrector is almost periodic. But the approximate corrector is in
general almost periodic together with its gradient.

It seems necessary to compare ours results in Theorems \ref{t1.4} and \ref%
{t5.1} with the existing ones in the literature. First of all, it is worth
noting that the algebra of continuous asymptotic almost periodic functions
is included in the Banach space of Weyl almost periodic functions; see e.g. 
\cite{Besicovitch}. Thus the results obtained in \cite{Shen1} can be seen as
generalizing those in Theorems \ref{t1.4} and \ref{t5.1}. However it is not
exactly the case. Indeed in \cite{Shen1}, the rates of convergence are found
in terms of the modulus of Weyl-almost periodicity of the matrix $A$, that
is, in terms of the function 
\begin{equation*}
\rho _{A}^{1}(R,L)=\sup_{y\in \mathbb{R}^{d}}\inf_{\left\vert z\right\vert
\leq R}\left( \sup_{x\in \mathbb{R}^{d}}%
\mathchoice {{\setbox0=\hbox{$\displaystyle{\textstyle
-}{\int}$ } \vcenter{\hbox{$\textstyle -$
}}\kern-.6\wd0}}{{\setbox0=\hbox{$\textstyle{\scriptstyle -}{\int}$ } \vcenter{\hbox{$\scriptstyle -$
}}\kern-.6\wd0}}{{\setbox0=\hbox{$\scriptstyle{\scriptscriptstyle -}{\int}$
} \vcenter{\hbox{$\scriptscriptstyle -$
}}\kern-.6\wd0}}{{\setbox0=\hbox{$\scriptscriptstyle{\scriptscriptstyle
-}{\int}$ } \vcenter{\hbox{$\scriptscriptstyle -$ }}\kern-.6\wd0}}%
\!\int_{B_{L}(x)}\left\vert A(t+y)-A(t+z)\right\vert ^{2}dt\right) ^{\frac{1%
}{2}}\text{ for }R,L>0
\end{equation*}%
where $B_{L}(x)$ stands for the open ball \ in $\mathbb{R}^{d}$ centered at $%
x$ and of radius $L>0$. In our work, we distinguish two cases: 1) the
asymptotic periodic case in which we show that the rate of convergence is
optimal, that $\left\Vert u_{\varepsilon }-u_{0}\right\Vert _{L^{2}(\Omega
)}=O(\varepsilon )$; 2) In the general continuous asymptotic almost periodic
setting, we show as in \cite{Shen1}, that the rate of convergence depends on
the modulus of asymptotic almost periodicity defined by 
\begin{equation*}
\rho _{A}(R,L)=\sup_{y\in \mathbb{R}^{d}}\inf_{\left\vert z\right\vert \leq
R}\left\Vert A(\cdot +y)-A(\cdot +z)\right\Vert _{L^{\infty }(B_{L}(0))}.
\end{equation*}%
As it is easily seen, the comparison between $\rho _{A}^{1}(R,L)$ and $\rho
_{A}(R,L)$ is not straightforward. So our result in Theorem \ref{t5.1} does
not follows directly from its counterpart Theorem 1.4 in \cite{Shen1}.

Our work combines the framework of \cite{Shen} with the general
deterministic homogenization theory introduced by Zhikov and Krivenko \cite%
{Zhikov4} and Nguetseng \cite{Hom1}. Furthermore, numerical simulations
based on finite volume method are provided to sustain our main theoretical
results.

The further investigation is organized as follows. Section 2 is devoted to
the proof of Theorems \ref{t1.1} and \ref{t4.1}. Section 3 deals with the
approximation of the homogenized coefficients. In Section 4, we prove
Theorems \ref{t1.4} while in Section 5 we prove Theorem \ref{t5.1}. In
Section 6, we provide some examples of concrete algebras and functions for
which the results, in particular those of Theorems \ref{t3.2}, \ref{t1.4}
and \ref{t5.1} apply. Finally, in Section 7 we present numerical results
illustrating the method and supporting the proposed procedure.

\section{Existence result for the corrector equation}

Let the matrix $A$ satisfy (\ref{1.2}) and (\ref{1.3}). Our aim is to solve
the corrector problem (\ref{1.6}). Let $B_{\mathcal{A}}^{2,\infty }(\mathbb{R%
}^{d})=B_{\mathcal{A}}^{2}(\mathbb{R}^{d})\cap L^{\infty }(\mathbb{R}^{d})$,
which is a Banach space under the $L^{\infty }(\mathbb{R}^{d})$-norm.

\begin{lemma}
\label{l4.1}Let $h\in \mathcal{C}(\overline{\Omega };B_{\mathcal{A}%
}^{2,\infty }(\mathbb{R}^{d}))$ and $H\in \mathcal{C}(\overline{\Omega };B_{%
\mathcal{A}}^{2,\infty }(\mathbb{R}^{d})^{d})$. For any $T>0$, there exists
a unique function $u\in \mathcal{C}(\overline{\Omega };B_{\mathcal{A}}^{1,2}(%
\mathbb{R}^{d}))$ such that 
\begin{equation}
-\nabla _{y}\cdot \left( A(x,\cdot )\nabla _{y}u(x,\cdot )\right)
+T^{-2}u(x,\cdot )=h(x,\cdot )+\nabla _{y}\cdot H(x,\cdot )\text{ in }%
\mathbb{R}^{d}  \label{4.2}
\end{equation}%
for any fixed $x\in \overline{\Omega }$. The solution $u$ satisfies further 
\begin{equation}
\sup_{z\in \mathbb{R}^{d}}%
\mathchoice {{\setbox0=\hbox{$\displaystyle{\textstyle -}{\int}$ } \vcenter{\hbox{$\textstyle -$
}}\kern-.6\wd0}}{{\setbox0=\hbox{$\textstyle{\scriptstyle -}{\int}$ } \vcenter{\hbox{$\scriptstyle -$
}}\kern-.6\wd0}}{{\setbox0=\hbox{$\scriptstyle{\scriptscriptstyle -}{\int}$
} \vcenter{\hbox{$\scriptscriptstyle -$
}}\kern-.6\wd0}}{{\setbox0=\hbox{$\scriptscriptstyle{\scriptscriptstyle
-}{\int}$ } \vcenter{\hbox{$\scriptscriptstyle -$ }}\kern-.6\wd0}}%
\!\int_{B_{R}(z)}\left( T^{-2}\left\vert u(x,y)\right\vert ^{2}+\left\vert
\nabla u(x,y)\right\vert ^{2}\right) dy\leq C\sup_{z\in \mathbb{R}^{d}}%
\mathchoice {{\setbox0=\hbox{$\displaystyle{\textstyle -}{\int}$ } \vcenter{\hbox{$\textstyle -$
}}\kern-.6\wd0}}{{\setbox0=\hbox{$\textstyle{\scriptstyle -}{\int}$ } \vcenter{\hbox{$\scriptstyle -$
}}\kern-.6\wd0}}{{\setbox0=\hbox{$\scriptstyle{\scriptscriptstyle -}{\int}$
} \vcenter{\hbox{$\scriptscriptstyle -$
}}\kern-.6\wd0}}{{\setbox0=\hbox{$\scriptscriptstyle{\scriptscriptstyle
-}{\int}$ } \vcenter{\hbox{$\scriptscriptstyle -$ }}\kern-.6\wd0}}%
\!\int_{B_{R}(z)}(\left\vert H(x,y)\right\vert ^{2}+T^{2}\left\vert
h(x,y)\right\vert ^{2})dy  \label{4.3}
\end{equation}%
for any $R\geq T$ and all $x\in \overline{\Omega }$, where the constant $C$
depends only on $d$, $\alpha $ and $\beta $.
\end{lemma}

\begin{proof}
Since the variable $x$ in (\ref{4.2}) behaves as a parameter, we drop it
throughout the proof of the existence and uniqueness. Thus, in what follows,
we keep using the symbol $\nabla $ instead of $\nabla _{y}$ to denote the
gradient with respect to $y$, if there is no danger of confusion.\medskip

1. \textit{Existence}. Fix $R>0$ and define $v_{T,R}\equiv v_{R}\in
H_{0}^{1}(B_{R})$ as the unique solution of 
\begin{equation*}
-\nabla \cdot A\nabla v_{R}+T^{-2}v_{R}=h+\nabla \cdot H\text{ in }B_{R}.
\end{equation*}%
Extending $v_{R}$ by $0$ off $B_{R}$, we obtain a sequence $(v_{R})_{R}$ in $%
H_{loc}^{1}(\mathbb{R}^{d})$. Let us show that the sequence $(v_{R})_{R}$ is
bounded in $H_{loc}^{1}(\mathbb{R}^{d})$. We proceed as in \cite{Gloria}
(see also \cite{Po-Yu89}). In the variational formulation of the above
equation, we choose as test function, the function $\eta _{z}^{2}v_{R}$,
where $\eta _{z}(y)=\exp (-c\left\vert y-z\right\vert )$ for a fixed $z\in 
\mathbb{R}^{d}$, $c>0$ to be chosen later. We get 
\begin{align*}
\int_{B_{R}}\eta _{z}^{2}A\nabla v_{R}\cdot \nabla
v_{R}+T^{-2}\int_{B_{R}}\eta _{z}^{2}v_{R}^{2}& =-2\int_{B_{R}}\eta
_{z}v_{R}A\nabla v_{R}\cdot \nabla \eta _{z}-2\int_{B_{R}}\eta
_{z}v_{R}H\cdot \nabla \eta _{z} \\
& -\int_{B_{R}}\eta _{z}^{2}H\cdot \nabla v_{R}+\int_{B_{R}}h\eta
_{z}^{2}v_{R} \\
& =I_{1}+I_{2}+I_{3}+I_{4}.
\end{align*}%
The left-hand side of the above equality is bounded from below by 
\begin{equation*}
\alpha \int_{B_{R}}\eta _{z}^{2}\left\vert \nabla v_{R}\right\vert
^{2}+T^{-2}\int_{B_{R}}\eta _{z}^{2}v_{R}^{2},
\end{equation*}%
while for the right-hand side, we have the following bounds (after using the
Young's inequality and the bounds on $A$): 
\begin{align*}
\left\vert I_{1}\right\vert & \leq \frac{\alpha \beta T^{-2}}{k}%
\int_{B_{R}}v_{R}^{2}\left\vert \nabla \eta _{z}\right\vert ^{2}+\frac{%
T^{2}\beta k}{\alpha }\int_{B_{R}}\eta _{z}^{2}\left\vert \nabla
v_{R}\right\vert ^{2}, \\
\left\vert I_{2}\right\vert & \leq \frac{\alpha \beta T^{-2}}{k}%
\int_{B_{R}}v_{R}^{2}\left\vert \nabla \eta _{z}\right\vert ^{2}+\frac{T^{2}k%
}{\alpha \beta }\int_{B_{R}}\eta _{z}^{2}\left\vert H\right\vert ^{2}, \\
\left\vert I_{3}\right\vert & \leq \frac{T^{2}\beta k}{\alpha }%
\int_{B_{R}}\eta _{z}^{2}\left\vert \nabla v_{R}\right\vert ^{2}+\frac{%
T^{-2}\alpha }{4k}\int_{B_{R}}\eta _{z}^{2}\left\vert H\right\vert ^{2}, \\
\left\vert I_{4}\right\vert & \leq \frac{\alpha \beta T^{-2}c^{2}}{k}%
\int_{B_{R}}v_{R}^{2}\eta _{z}^{2}+\frac{T^{2}k}{4\alpha \beta c^{2}}%
\int_{B_{R}}\eta _{z}^{2}\left\vert h\right\vert ^{2}
\end{align*}%
where $k>0$ is to be chosen later. Noticing that $\left\vert \nabla \eta
_{z}\right\vert =c\eta _{z}$, we readily get after using the series of
inequalities above, 
\begin{eqnarray*}
&&\int_{B_{R}}\eta _{z}^{2}\left( \alpha -2\frac{T^{2}\beta k}{\alpha }%
\right) \left\vert \nabla v_{R}\right\vert ^{2}+T^{-2}\int_{B_{R}}\eta
_{z}^{2}\left( 1-3\frac{\alpha \beta c^{2}}{k}\right) v_{R}^{2} \\
&\leq &\int_{B_{R}}\left[ \left( \frac{T^{2}k}{\alpha \beta }+\frac{%
T^{-2}\alpha }{4\beta k}\right) \left\vert H\right\vert ^{2}+\frac{kT^{2}}{%
4\alpha \beta c^{2}}\left\vert h\right\vert ^{2}\right] \eta _{z}^{2}.
\end{eqnarray*}%
Choosing therefore $k=\frac{\alpha ^{2}}{4\beta T^{2}}$ and $c=\frac{1}{%
2\beta T}\left( \frac{\alpha }{6}\right) ^{1/2}$, we obtain the estimate 
\begin{equation}
\alpha \int_{B_{R}}\eta _{z}^{2}\left\vert \nabla v_{R}\right\vert
^{2}+T^{-2}\int_{B_{R}}\eta _{z}^{2}v_{R}^{2}\leq \int_{B_{R}}\left[ \left( 
\frac{\alpha }{4\beta ^{2}}+\frac{1}{\alpha }\right) \left\vert H\right\vert
^{2}+\frac{3}{2}T^{2}\left\vert h\right\vert ^{2}\right] \eta _{z}^{2}.
\label{4.5}
\end{equation}%
The inequality (\ref{4.5}) above shows that the sequence $(v_{R})$ is
bounded in $H_{loc}^{1}(\mathbb{R}^{d})$; indeed, for any compact subset $K$
in $\mathbb{R}^{d}$, the left-hand side of (\ref{4.5}) is bounded from below
by $c_{K}(\alpha \int_{B_{R}}\left\vert \nabla v_{R}\right\vert
^{2}+T^{-2}\int_{B_{R}}v_{R}^{2})$ where $c_{K}=\min_{K}\eta _{z}^{2}>0$
while the right-hand side is bounded from above by $C\int_{\mathbb{R}%
^{d}}\eta _{z}^{2}$ where 
\begin{equation*}
C=\left( \frac{\alpha }{4\beta ^{2}}+\frac{1}{\alpha }\right) \left\Vert
H\right\Vert _{\mathcal{C}(\overline{\Omega };L^{\infty }(\mathbb{R}%
^{d}))}^{2}+\frac{3}{2}T^{2}\left\Vert h\right\Vert _{\mathcal{C}(\overline{%
\Omega };L^{\infty }(\mathbb{R}^{d}))}^{2}.
\end{equation*}%
Hence there exist a subsequence of $(v_{R})$ and a function $v\in
H_{loc}^{1}(\mathbb{R}^{d})$ such that the above mentioned subsequence
weakly converges in $H_{loc}^{1}(\mathbb{R}^{d})$ to $v$, and it is easy to
see that $v$ is a distributional solution of (\ref{4.2}) in $\mathbb{R}^{d}$%
. Taking the $\lim \inf_{R\rightarrow \infty }$ in (\ref{4.5}) yields 
\begin{equation}
\alpha \int_{\mathbb{R}^{d}}\eta _{z}^{2}\left\vert \nabla v_{R}\right\vert
^{2}+T^{-2}\int_{\mathbb{R}^{d}}\eta _{z}^{2}v_{R}^{2}\leq \int_{\mathbb{R}%
^{d}}\left[ \left( \frac{\alpha }{4\beta ^{2}}+\frac{1}{\alpha }\right)
\left\vert H\right\vert ^{2}+\frac{3}{2}T^{2}\left\vert h\right\vert ^{2}%
\right] \eta _{z}^{2}.  \label{4.6}
\end{equation}%
We infer from (\ref{4.6}) that 
\begin{equation}
\sup_{z\in \mathbb{R}^{d}}%
\mathchoice {{\setbox0=\hbox{$\displaystyle{\textstyle -}{\int}$ } \vcenter{\hbox{$\textstyle -$
}}\kern-.6\wd0}}{{\setbox0=\hbox{$\textstyle{\scriptstyle -}{\int}$ } \vcenter{\hbox{$\scriptstyle -$
}}\kern-.6\wd0}}{{\setbox0=\hbox{$\scriptstyle{\scriptscriptstyle -}{\int}$
} \vcenter{\hbox{$\scriptscriptstyle -$
}}\kern-.6\wd0}}{{\setbox0=\hbox{$\scriptscriptstyle{\scriptscriptstyle
-}{\int}$ } \vcenter{\hbox{$\scriptscriptstyle -$ }}\kern-.6\wd0}}%
\!\int_{B_{R}(z)}(\left\vert \nabla v\right\vert ^{2}+T^{-2}v^{2})\leq C
\label{e2.4}
\end{equation}%
where $C$ does not depend on $z$, but on $T$. Estimate (\ref{4.3}) (for $R=T$%
) follows from \cite{Po-Yu89} while the case $R>T$ is a consequence of
Caccioppoli's inequality; see \cite[Lemma 3.2]{Shen1}.

Let us show that $v\in B_{\mathcal{A}}^{1,2}(\mathbb{R}^{d})$. It suffices
to check that $v$ solves the equation 
\begin{equation}
M(A(\xi +\nabla v)\cdot \nabla \phi +T^{-2}v\phi )=M(h\phi -H\cdot \nabla
\phi )\text{, all }\phi \in B_{\mathcal{A}}^{1,2}(\mathbb{R}^{d}).
\label{4.7}
\end{equation}%
To this end, let $\varphi \in \mathcal{C}_{0}^{\infty }(\mathbb{R}^{d})$ and 
$\phi \in B_{\mathcal{A}}^{1,2}(\mathbb{R}^{d})$. Define (for fixed $%
\varepsilon >0$), $\psi (y)=\varphi (\varepsilon y)\phi (y)$. Choose $\psi $
as test function in the variational form of (\ref{4.2}) and get 
\begin{align*}
& \int_{\mathbb{R}^{d}}\left[ A\nabla u\cdot (\varepsilon \phi \nabla
\varphi (\varepsilon \cdot )+\varphi (\varepsilon \cdot )\nabla \phi
)+T^{-2}u\varphi (\varepsilon \cdot )\phi \right] dy \\
& =\int_{\mathbb{R}^{d}}\left[ h\varphi (\varepsilon \cdot )\phi -H\cdot
(\varepsilon \phi \nabla \varphi (\varepsilon \cdot )+\varphi (\varepsilon
\cdot )\nabla \phi )\right] dy.
\end{align*}%
The change of variables $t=\varepsilon y$ leads (after multiplication by $%
\varepsilon ^{d}$) to 
\begin{align*}
& \int_{\mathbb{R}^{d}}\left[ A^{\varepsilon }(\nabla _{y}u)^{\varepsilon
}\cdot (\varepsilon \phi ^{\varepsilon }\nabla \varphi +\varphi (\nabla
_{y}\phi )^{\varepsilon })+T^{-2}u^{\varepsilon }\varphi \phi ^{\varepsilon }%
\right] dt \\
& =\int_{\mathbb{R}^{d}}\left[ h^{\varepsilon }\phi ^{\varepsilon }\varphi
-H^{\varepsilon }\cdot (\varepsilon \phi ^{\varepsilon }\nabla \varphi
+\varphi (\nabla _{y}\phi )^{\varepsilon })\right] dt
\end{align*}%
where $w^{\varepsilon }(t)=w(t/\varepsilon )$ for a given $w$. Letting $%
\varepsilon \rightarrow 0$ above yields 
\begin{align*}
\int_{\mathbb{R}^{d}}M(A\nabla u\cdot \nabla \phi +T^{-2}u\phi )\varphi dt&
=\int_{\mathbb{R}^{d}}M(h\phi -H\cdot \nabla \phi )\varphi dt \\
\text{for all }\varphi & \in \mathcal{C}_{0}^{\infty }(\mathbb{R}^{d})\text{
and }\phi \in B_{\mathcal{A}}^{1,2}(\mathbb{R}^{d}).
\end{align*}%
which amounts to (\ref{4.7}). So, we have just shown that, if $v\in
H_{loc}^{1}(\mathbb{R}^{d})$ solves (\ref{4.2}) in the sense of
distributions in $\mathbb{R}^{d}$, then it satisfies (\ref{4.7}). Before we
proceed any further, let us first show that (\ref{4.7}) possesses a unique
solution in $B_{\mathcal{A}}^{1,2}(\mathbb{R}^{d})$ up to an additive
function $w\in B_{\mathcal{A}}^{1,2}(\mathbb{R}^{d})$ satisfying $%
M(\left\vert w\right\vert ^{2})=0$. First and foremost, we recall that the
space $\mathcal{B}_{\mathcal{A}}^{1,2}(\mathbb{R}^{d})=B_{\mathcal{A}}^{1,2}(%
\mathbb{R}^{d})/\mathcal{N}$ (where $\mathcal{N}=\{u\in B_{\mathcal{A}%
}^{1,2}(\mathbb{R}^{d}):\left\Vert u\right\Vert _{1,2}=0\}$) is a Hilbert
space with inner product 
\begin{equation*}
(u+\mathcal{N},v+\mathcal{N})_{1,2}=M(uv+\nabla u\cdot \nabla v)\text{ for }%
u,v\in B_{\mathcal{A}}^{1,2}(\mathbb{R}^{d}).
\end{equation*}%
If $w\in \mathcal{N}$ then $M(w)=0$, since $\left\vert M(w)\right\vert \leq
M(\left\vert w\right\vert )\leq (M(\left\vert w\right\vert
^{2}))^{1/2}=\left\Vert w\right\Vert _{2}=0$, so that $\left( ,\right)
_{1,2} $ is well defined. Now, (\ref{4.7}) is equivalent to $a(v,\phi )=\ell
(\phi ) $ for all $\phi \in B_{\mathcal{A}}^{1,2}(\mathbb{R}^{d})$ where 
\begin{equation*}
a(v,\phi )=M(T^{-2}v\phi +A\nabla v\cdot \nabla \phi ),\ \ell (\phi
)=M(h\phi -H\cdot \nabla \phi ).
\end{equation*}%
$a(\cdot ,\cdot )$ defines a continuous coercive bilinear form on $\mathcal{B%
}_{\mathcal{A}}^{1,2}(\mathbb{R}^{d})$; $\ell $ is a continuous linear form
on $\mathcal{B}_{\mathcal{A}}^{1,2}(\mathbb{R}^{d})$. Lax-Milgram theorem
implies that $v+\mathcal{N}$ is a unique solution of (\ref{4.7}). This
yields $v\in B_{\mathcal{A}}^{1,2}(\mathbb{R}^{d})$.\medskip

2. \textit{Uniqueness}. The uniqueness of the solution amounts to consider (%
\ref{4.2}) with $h=0$ and $H=0$. We derive from (\ref{4.6}) 
\begin{equation*}
\alpha \int_{\mathbb{R}^{d}}\eta _{z}^{2}\left\vert \nabla v\right\vert
^{2}+T^{-2}\int_{\mathbb{R}^{d}}\eta _{z}^{2}v^{2}=0,
\end{equation*}%
so that $v=0$ for the corresponding equation.\medskip

3. \textit{Continuity}. To investigate the continuity of $v$ with respect to 
$x$, we fix $x_{0}\in \overline{\Omega }$ \ and we let $w(x)=v(x,\cdot
)-v(x_{0},\cdot )$. Then $w(x)\in B_{\mathcal{A}}^{1,2}(\mathbb{R}^{d})$ and 
\begin{eqnarray*}
-\nabla \cdot A(x,\cdot )\nabla w(x)+T^{-2}w(x) &=&h(x,\cdot )-h(x_{0},\cdot
)+\nabla \cdot (H(x,\cdot )-H(x_{0},\cdot )) \\
&&+\nabla \cdot (A(x,\cdot )-A(x_{0},\cdot ))\nabla v(x_{0},\cdot ),
\end{eqnarray*}%
so that, using estimate (\ref{4.3}), we find (for any $R\geq T$) 
\begin{eqnarray*}
\sup_{z\in \mathbb{R}^{d}}%
\mathchoice {{\setbox0=\hbox{$\displaystyle{\textstyle -}{\int}$ } \vcenter{\hbox{$\textstyle -$
}}\kern-.6\wd0}}{{\setbox0=\hbox{$\textstyle{\scriptstyle -}{\int}$ } \vcenter{\hbox{$\scriptstyle -$
}}\kern-.6\wd0}}{{\setbox0=\hbox{$\scriptstyle{\scriptscriptstyle -}{\int}$
} \vcenter{\hbox{$\scriptscriptstyle -$
}}\kern-.6\wd0}}{{\setbox0=\hbox{$\scriptscriptstyle{\scriptscriptstyle
-}{\int}$ } \vcenter{\hbox{$\scriptscriptstyle -$ }}\kern-.6\wd0}}%
\!\int_{B_{R}(z)}\left( T^{-2}\left\vert w(x)\right\vert ^{2}+\left\vert
\nabla w(x)\right\vert ^{2}\right) dy &\leq &CT^{2}\sup_{z\in \mathbb{R}^{d}}%
\mathchoice {{\setbox0=\hbox{$\displaystyle{\textstyle -}{\int}$ } \vcenter{\hbox{$\textstyle -$
}}\kern-.6\wd0}}{{\setbox0=\hbox{$\textstyle{\scriptstyle -}{\int}$ } \vcenter{\hbox{$\scriptstyle -$
}}\kern-.6\wd0}}{{\setbox0=\hbox{$\scriptstyle{\scriptscriptstyle -}{\int}$
} \vcenter{\hbox{$\scriptscriptstyle -$
}}\kern-.6\wd0}}{{\setbox0=\hbox{$\scriptscriptstyle{\scriptscriptstyle
-}{\int}$ } \vcenter{\hbox{$\scriptscriptstyle -$ }}\kern-.6\wd0}}%
\!\int_{B_{R}(z)}\left\vert h(x,y)-h(x_{0},y)\right\vert ^{2}dy \\
&&+C\sup_{z\in \mathbb{R}^{d}}%
\mathchoice {{\setbox0=\hbox{$\displaystyle{\textstyle -}{\int}$ } \vcenter{\hbox{$\textstyle -$
}}\kern-.6\wd0}}{{\setbox0=\hbox{$\textstyle{\scriptstyle -}{\int}$ } \vcenter{\hbox{$\scriptstyle -$
}}\kern-.6\wd0}}{{\setbox0=\hbox{$\scriptstyle{\scriptscriptstyle -}{\int}$
} \vcenter{\hbox{$\scriptscriptstyle -$
}}\kern-.6\wd0}}{{\setbox0=\hbox{$\scriptscriptstyle{\scriptscriptstyle
-}{\int}$ } \vcenter{\hbox{$\scriptscriptstyle -$ }}\kern-.6\wd0}}%
\!\int_{B_{R}(z)}\left\vert H(x,y)-H(x_{0},y)\right\vert ^{2}dy \\
&&+C\sup_{z\in \mathbb{R}^{d}}%
\mathchoice {{\setbox0=\hbox{$\displaystyle{\textstyle -}{\int}$ } \vcenter{\hbox{$\textstyle -$
}}\kern-.6\wd0}}{{\setbox0=\hbox{$\textstyle{\scriptstyle -}{\int}$ } \vcenter{\hbox{$\scriptstyle -$
}}\kern-.6\wd0}}{{\setbox0=\hbox{$\scriptstyle{\scriptscriptstyle -}{\int}$
} \vcenter{\hbox{$\scriptscriptstyle -$
}}\kern-.6\wd0}}{{\setbox0=\hbox{$\scriptscriptstyle{\scriptscriptstyle
-}{\int}$ } \vcenter{\hbox{$\scriptscriptstyle -$ }}\kern-.6\wd0}}%
\!\int_{B_{R}(z)}\left\vert A(x,y)-A(x_{0},y)\right\vert ^{2}\left\vert
\nabla v(x_{0},y)\right\vert ^{2}dy \\
&\leq &CT^{2}\left\Vert h(x,\cdot )-h(x_{0},\cdot )\right\Vert _{L^{\infty }(%
\mathbb{R}^{d})}^{2} \\
&&+C\left\Vert H(x,\cdot )-H(x_{0},\cdot )\right\Vert _{L^{\infty }(\mathbb{R%
}^{d})}^{2} \\
&&+C\left\Vert A(x,\cdot )-A(x_{0},\cdot )\right\Vert _{L^{\infty }(\mathbb{R%
}^{d})}^{2}.
\end{eqnarray*}%
Continuity is a consequence of the following estimate 
\begin{eqnarray*}
&&T^{-2}\left\Vert v(x,\cdot )-v(x_{0},\cdot )\right\Vert
_{2}^{2}+\left\Vert \nabla v(x,\cdot )-\nabla v(x_{0},\cdot )\right\Vert
_{2}^{2} \\
&\equiv &\lim_{R\rightarrow \infty }%
\mathchoice {{\setbox0=\hbox{$\displaystyle{\textstyle -}{\int}$ } \vcenter{\hbox{$\textstyle -$
}}\kern-.6\wd0}}{{\setbox0=\hbox{$\textstyle{\scriptstyle -}{\int}$ } \vcenter{\hbox{$\scriptstyle -$
}}\kern-.6\wd0}}{{\setbox0=\hbox{$\scriptstyle{\scriptscriptstyle -}{\int}$
} \vcenter{\hbox{$\scriptscriptstyle -$
}}\kern-.6\wd0}}{{\setbox0=\hbox{$\scriptscriptstyle{\scriptscriptstyle
-}{\int}$ } \vcenter{\hbox{$\scriptscriptstyle -$ }}\kern-.6\wd0}}%
\!\int_{B_{R}(z)}T^{-2}\left\vert w(x)\right\vert ^{2}+\left\vert \nabla
w(x)\right\vert ^{2}dy \\
&\leq &CT^{2}\left\Vert h(x,\cdot )-h(x_{0},\cdot )\right\Vert _{L^{\infty }(%
\mathbb{R}^{d})}^{2}+C\left\Vert H(x,\cdot )-H(x_{0},\cdot )\right\Vert
_{L^{\infty }(\mathbb{R}^{d})}^{2} \\
&&+C\left\Vert A(x,\cdot )-A(x_{0},\cdot )\right\Vert _{L^{\infty }(\mathbb{R%
}^{d})}^{2}.
\end{eqnarray*}
\end{proof}

\begin{proof}[Proof of Theorem \protect\ref{t4.1}]
1. \textit{Existence and continuity}. Let us denote by $(\chi _{T,j}(x,\cdot
))_{T\geq 1}$ (for fixed $1\leq j\leq d$) the sequence constructed in Lemma %
\ref{l4.1} and corresponding to $h=0$ and $H=Ae_{j}$, $e_{j}$ being denoting
the $j$th vector of the canonical basis of $\mathbb{R}^{d}$. It satisfies (%
\ref{4.3}), so that by the weak compactness, the sequence $(\nabla \chi
_{T,j}(x,\cdot ))_{T\geq 1}$ weakly converges in $L_{loc}^{2}(\mathbb{R}%
^{d})^{d}$ (up to extraction of a subsequence) to some $V_{j}(x,\cdot )\in
L_{loc}^{2}(\mathbb{R}^{d})^{d}$. From the equality $\partial ^{2}\chi
_{T,j}(x,\cdot )/\partial y_{i}\partial y_{l}=\partial ^{2}\chi
_{T,j}(x,\cdot )/\partial y_{l}\partial y_{i}$, a limit passage in the
distributional sense yields $\partial V_{j,i}(x,\cdot )/\partial
y_{l}=\partial V_{j,l}(x,\cdot )/\partial y_{i}$, where $V_{j}=(V_{j,i})_{1%
\leq i\leq d}$. This implies $V_{j}(x,\cdot )=\nabla \chi _{j}(x,\cdot )$
for some $\chi _{j}(x,\cdot )\in H_{loc}^{1}(\mathbb{R}^{d})$. Using the
boundedness of $(T^{-1}\chi _{T,j}(x,\cdot ))_{T\geq 1}$ in $L_{loc}^{2}(%
\mathbb{R}^{d})$, we pass to the limit in the variational formulation of (%
\ref{4.2}) (as $T\rightarrow \infty $) to get that $\chi _{j}$ solves (\ref%
{4.1}). Arguing exactly as in the proof of (\ref{4.7}) (in Lemma \ref{l4.1}%
), we arrive at $V_{j}(x,\cdot )\in B_{\mathcal{A}}^{2}(\mathbb{R}^{d})^{d}$%
. Also, since $\chi _{T,j}(x,\cdot )\in B_{\mathcal{A}}^{1,2}(\mathbb{R}%
^{d}) $, we have $M(\nabla \chi _{T,j}(x,\cdot ))=0$, hence $M(\nabla \chi
_{j}(x,\cdot ))=0$. We repeat the proof of the Part 3. in the previous lemma
to find that $\nabla _{y}\chi _{j}\in \mathcal{C}(\overline{\Omega };B_{%
\mathcal{A}}^{2}(\mathbb{R}^{d})^{d})$.\medskip

2. \textit{Uniqueness} (of $\nabla _{y}\chi _{j}$). Fix $x\in \overline{%
\Omega }$ and assume that $\chi _{j}(x,\cdot )\in H_{loc}^{1}(\mathbb{R}%
^{d}) $ is such that $-\func{div}(A(x,\cdot )\nabla _{y}\chi _{j}(x,\cdot
))=0$ in $\mathbb{R}^{d}$ and $\nabla _{y}\chi _{j}(x,\cdot )\in B_{\mathcal{%
A}}^{2}(\mathbb{R}^{d})^{d}$. Then it follows from \cite[Property (3.10)]%
{Shen} that, given $0<\sigma <1$, there exists $C_{\sigma }>0$ independent
from $r$ and $R$\ such that 
\begin{equation}
\mathchoice {{\setbox0=\hbox{$\displaystyle{\textstyle
-}{\int}$ } \vcenter{\hbox{$\textstyle -$
}}\kern-.6\wd0}}{{\setbox0=\hbox{$\textstyle{\scriptstyle -}{\int}$ } \vcenter{\hbox{$\scriptstyle -$
}}\kern-.6\wd0}}{{\setbox0=\hbox{$\scriptstyle{\scriptscriptstyle -}{\int}$
} \vcenter{\hbox{$\scriptscriptstyle -$
}}\kern-.6\wd0}}{{\setbox0=\hbox{$\scriptscriptstyle{\scriptscriptstyle
-}{\int}$ } \vcenter{\hbox{$\scriptscriptstyle -$ }}\kern-.6\wd0}}%
\!\int_{B_{r}}\left\vert \nabla _{y}\chi _{j}(x,y)\right\vert ^{2}dy\leq
C_{\sigma }\left( \frac{r}{R}\right) ^{\sigma }%
\mathchoice {{\setbox0=\hbox{$\displaystyle{\textstyle
-}{\int}$ } \vcenter{\hbox{$\textstyle -$
}}\kern-.6\wd0}}{{\setbox0=\hbox{$\textstyle{\scriptstyle -}{\int}$ } \vcenter{\hbox{$\scriptstyle -$
}}\kern-.6\wd0}}{{\setbox0=\hbox{$\scriptstyle{\scriptscriptstyle -}{\int}$
} \vcenter{\hbox{$\scriptscriptstyle -$
}}\kern-.6\wd0}}{{\setbox0=\hbox{$\scriptscriptstyle{\scriptscriptstyle
-}{\int}$ } \vcenter{\hbox{$\scriptscriptstyle -$ }}\kern-.6\wd0}}%
\!\int_{B_{R}}\left\vert \nabla _{y}\chi _{j}(x,y)\right\vert ^{2}dy\text{
for all }0<r<R.  \label{04}
\end{equation}%
Next, since $-\func{div}(A(x,\cdot )\nabla _{y}\chi _{j}(x,\cdot ))=0$ in $%
\mathbb{R}^{d}$ and $\nabla _{y}\chi _{j}(x,\cdot )\in B_{\mathcal{A}}^{2}(%
\mathbb{R}^{d})^{d}$, we show as for (\ref{4.7}) that 
\begin{equation}
M(A(x,\cdot )\nabla _{y}\chi _{j}(x,\cdot )\cdot \nabla _{y}\phi )=0\text{
for all }\phi \in B_{\#\mathcal{A}}^{1,2}(\mathbb{R}^{d}).  \label{05}
\end{equation}%
Choosing $\phi =\chi _{j}(x,\cdot )$ in (\ref{05}), and using the
ellipticity of $A$, it emerges $M(\left\vert \nabla _{y}\chi _{j}(x,\cdot
)\right\vert ^{2})=0$, that is, $\lim_{R\rightarrow \infty }%
\mathchoice {{\setbox0=\hbox{$\displaystyle{\textstyle
-}{\int}$ } \vcenter{\hbox{$\textstyle -$
}}\kern-.6\wd0}}{{\setbox0=\hbox{$\textstyle{\scriptstyle -}{\int}$ } \vcenter{\hbox{$\scriptstyle -$
}}\kern-.6\wd0}}{{\setbox0=\hbox{$\scriptstyle{\scriptscriptstyle -}{\int}$
} \vcenter{\hbox{$\scriptscriptstyle -$
}}\kern-.6\wd0}}{{\setbox0=\hbox{$\scriptscriptstyle{\scriptscriptstyle
-}{\int}$ } \vcenter{\hbox{$\scriptscriptstyle -$ }}\kern-.6\wd0}}%
\!\int_{B_{R}}\left\vert \nabla _{y}\chi _{j}(x,y)\right\vert ^{2}dy=0$.
Coming back to (\ref{04}) and letting there $R\rightarrow \infty $, we are
led to $\int_{B_{r}}\left\vert \nabla _{y}\chi _{j}(x,y)\right\vert ^{2}dy=0$
for all $r>0$. This gives $\nabla _{y}\chi _{j}(x,\cdot )=0$.
\end{proof}

We can now prove Theorem \ref{t1.1}.

\begin{proof}[Proof of Theorem \protect\ref{t1.1}]
Let $\Phi _{\varepsilon }=\psi _{0}+\varepsilon \psi _{1}^{\varepsilon }$
with $\psi _{1}^{\varepsilon }(x)=\psi _{1}(x,x/\varepsilon )$ ($x\in \Omega 
$), where $\psi _{0}\in \mathcal{C}_{0}^{\infty }(\Omega )$ and $\psi
_{1}\in \mathcal{C}_{0}^{\infty }(\Omega )\otimes \mathcal{A}^{\infty }$, $%
\mathcal{A}^{\infty }=\{u\in \mathcal{A}:D^{\alpha }u\in \mathcal{A}$ for
all $\alpha \in \mathbb{N}^{d}\}$. Taking $\Phi _{\varepsilon }$ (wich
belongs to $\mathcal{C}_{0}^{\infty }(\Omega )$) as a test function in the
variational formulation of (\ref{1.1}) yields 
\begin{equation}
\int_{\Omega }A^{\varepsilon }\nabla u_{\varepsilon }\cdot \nabla \Phi
_{\varepsilon }dx=\int_{\Omega }f\Phi _{\varepsilon }dx.  \label{4.6'}
\end{equation}%
It is not difficult to see that the sequence $(u_{\varepsilon
})_{\varepsilon >0}$ is bounded in $H_{0}^{1}(\Omega )$, so that,
considering an ordinary sequence $E\subset \mathbb{R}_{+}^{\ast }$, there
exist a couple $(u_{0},u_{1})\in H_{0}^{1}(\Omega )\times L^{2}(\Omega ;B_{\#%
\mathcal{A}}^{1,2}(\mathbb{R}^{d}))$ and a subsequence $E^{\prime }$ of $E$
such that, as $E^{\prime }\ni \varepsilon \rightarrow 0$, 
\begin{equation*}
u_{\varepsilon }\rightarrow u_{0}\text{ in }H_{0}^{1}(\Omega )\text{-weak
and in }L^{2}(\Omega )\text{-strong}
\end{equation*}%
\begin{equation}
\nabla u_{\varepsilon }\rightarrow \nabla u_{0}+\nabla _{y}u_{1}\text{ in }%
L^{2}(\Omega )^{d}\text{-weak }\Sigma .\ \ \ \ \ \ \ \ \ \   \label{4.7'}
\end{equation}%
On the other hand 
\begin{equation}
\nabla \Phi _{\varepsilon }=\nabla \psi _{0}+(\nabla _{y}\psi
_{1})^{\varepsilon }+\varepsilon (\nabla \psi _{1})^{\varepsilon
}\rightarrow \nabla \psi _{0}+\nabla _{y}\psi _{1}\text{ in }L^{2}(\Omega
)^{d}\text{-strong }\Sigma .  \label{4.8'}
\end{equation}%
This yields in (\ref{4.6'}) the following limit problem 
\begin{equation}
\int_{\Omega }M\left( A(\nabla u_{0}+\nabla _{y}u_{1})\cdot (\nabla \psi
_{0}+\nabla _{y}\psi _{1})\right) dx=\int_{\Omega }f\psi _{0}dx\ \ \forall
(\psi _{0},\psi _{1})\in \mathcal{C}_{0}^{\infty }(\Omega )\times (\mathcal{C%
}_{0}^{\infty }(\Omega )\otimes \mathcal{A}^{\infty }).  \label{4.9'}
\end{equation}%
Problem (\ref{4.9'}) above is equivalent to the system 
\begin{equation}
\int_{\Omega }M\left( A(\nabla u_{0}+\nabla _{y}u_{1})\cdot \nabla \psi
_{0}\right) dx=\int_{\Omega }f\psi _{0}dx\ \ \forall \psi _{0}\in \mathcal{C}%
_{0}^{\infty }(\Omega )  \label{4.10'}
\end{equation}%
\begin{equation}
\int_{\Omega }M\left( A(\nabla u_{0}+\nabla _{y}u_{1})\cdot \nabla _{y}\psi
_{1}\right) dx=0\ \ \forall \psi _{1}\in \mathcal{C}_{0}^{\infty }(\Omega
)\otimes \mathcal{A}^{\infty }.  \label{4.11'}
\end{equation}%
Taking in (\ref{4.11'}) $\psi _{1}(x,y)=\varphi (x)v(y)$ with $\varphi \in 
\mathcal{C}_{0}^{\infty }(\Omega )$ and $v\in \mathcal{A}^{\infty }$, we get 
\begin{equation}
M\left( A(x,\cdot )(\nabla u_{0}+\nabla _{y}u_{1})\cdot \nabla _{y}v\right)
=0\ \ \forall v\in \mathcal{A}^{\infty },x\in \overline{\Omega },
\label{4.12'}
\end{equation}%
which is, thanks to the density of $\mathcal{A}^{\infty }$ in $B_{\mathcal{A}%
}^{1,2}(\mathbb{R}^{d})$, the weak form of 
\begin{equation}
\nabla _{y}\cdot \left( A(x,\cdot )(\nabla u_{0}+\nabla _{y}u_{1})\right) =0%
\text{ in }\mathbb{R}^{d}\text{ (for all fixed }x\in \overline{\Omega }\text{%
),}  \label{4.13'}
\end{equation}%
with respect to the duality defined by (\ref{4.12'}). So fix $\xi \in 
\mathbb{R}^{d}$ and consider the problem 
\begin{equation}
\nabla _{y}\cdot \left( A(x,\cdot )(\xi +\nabla _{y}v_{\xi }(x,\cdot
))\right) =0\text{ in }\mathbb{R}^{d};\ v_{\xi }(x,\cdot )\in B_{\#\mathcal{A%
}}^{1,2}(\mathbb{R}^{d}).  \label{4.14'}
\end{equation}%
Thanks to Theorem \ref{t4.1}, Eq. (\ref{4.14'}) possesses a unique solution $%
v_{\xi }$ (up to an additive constant depending on $x$) in $\mathcal{C}(%
\overline{\Omega };B_{\#\mathcal{A}}^{1,2}(\mathbb{R}^{d}))$. Choosing there 
$\xi =\nabla u_{0}(x)$, the uniqueness of the solution implies $%
u_{1}(x,y)=\chi (x,y)\cdot \nabla u_{0}(x)$ where $\chi =(\chi _{j})_{1\leq
j\leq d}$ with $\chi _{j}=v_{e_{j}}$, $e_{j}$ the $j$th vector of the
canonical basis of $\mathbb{R}^{d}$. Replacing in (\ref{4.10'}) $u_{1}$ by $%
\chi \cdot \nabla u_{0}$, we get 
\begin{equation*}
\int_{\Omega }(M(A(I+\nabla _{y}\chi )\nabla u_{0})\cdot \nabla \psi
_{0}dx=\int_{\Omega }f\psi _{0}dx\ \ \forall \psi _{0}\in \mathcal{C}%
_{0}^{\infty }(\Omega ),
\end{equation*}%
that is, $-\nabla \cdot A^{\ast }(x)\nabla u_{0}=f$ in $\Omega $.

It remains to verify (\ref{1.7}). Define $\Phi _{\varepsilon
}(x)=u_{0}(x)+\varepsilon u_{1}(x,x/\varepsilon )$. Then using (\ref{1.2})
we obtain 
\begin{align*}
\alpha \int_{\Omega }\left\vert \nabla u_{\varepsilon }-\nabla \Phi
_{\varepsilon }\right\vert ^{2}dx& \leq \int_{\Omega }A^{\varepsilon }\nabla
(u_{\varepsilon }-\Phi _{\varepsilon })\cdot \nabla (u_{\varepsilon }-\Phi
_{\varepsilon })dx \\
& =\int_{\Omega }f(u_{\varepsilon }-\Phi _{\varepsilon })dx-\int_{\Omega
}A^{\varepsilon }\nabla \Phi _{\varepsilon }\cdot \nabla (u_{\varepsilon
}-\Phi _{\varepsilon })dx.
\end{align*}%
Since $u_{1}\in L^{2}(\Omega ;\mathcal{A}^{1})$, we have that $\int_{\Omega
}f(u_{\varepsilon }-\Phi _{\varepsilon })dx\rightarrow 0$. Indeed $\Phi
_{\varepsilon }\rightarrow u_{0}$ in $L^{2}(\Omega )$ (and hence $%
u_{\varepsilon }-\Phi _{\varepsilon }\rightarrow 0$ in $L^{2}(\Omega )$).
Next observe that $\nabla \Phi _{\varepsilon }\rightarrow \nabla
u_{0}+\nabla _{y}u_{1}$ in $L^{2}(\Omega )$-strong $\Sigma $; in fact, $%
\nabla \Phi _{\varepsilon }=\nabla u_{0}+\varepsilon (\nabla
u_{1})^{\varepsilon }+(\nabla _{y}u_{1})^{\varepsilon }$, and since $\nabla
_{y}u_{1}\in L^{2}(\Omega ;\mathcal{A})$, we obtain $(\nabla
_{y}u_{1})^{\varepsilon }\rightarrow \nabla _{y}u_{1}$ in $L^{2}(\Omega )$%
-strong $\Sigma $. One gets readily $\nabla u_{\varepsilon }-\nabla \Phi
_{\varepsilon }\rightarrow 0$ in $L^{2}(\Omega )$-weak $\Sigma $. Using $A$
as a test function, $\int_{\Omega }A^{\varepsilon }\nabla \Phi _{\varepsilon
}\cdot \nabla (u_{\varepsilon }-\Phi _{\varepsilon })dx\rightarrow 0$. We
have just shown that $u_{\varepsilon }-u_{0}-\varepsilon u_{1}^{\varepsilon
}\rightarrow 0$ in $L^{2}(\Omega )$ and $\nabla (u_{\varepsilon
}-u_{0}-\varepsilon u_{1}^{\varepsilon })=\nabla u_{\varepsilon }-\nabla
\Phi _{\varepsilon }\rightarrow 0$ in $L^{2}(\Omega )$. This proves (\ref%
{1.7}) and completes the proof of Theorem \ref{t1.1}.
\end{proof}

We assume henceforth that the matrix $A$ does not depend on $x$, that is, $%
A(x,y)=A(y)$. Let $\chi _{T}=(\chi _{T,j})_{1\leq j\leq d}$ be defined by (%
\ref{e01}).

\begin{lemma}
\label{l11.1}Let $T\geq 1$ and $\sigma \in (0,1)$. Assume that $A\in (%
\mathcal{A})^{d\times d}$. There exist positive numbers $C=C(A,d)$ and $%
C_{\sigma }=C_{\sigma }(d,\sigma ,A)$ such that 
\begin{equation}
T^{-1}\left\Vert \chi _{T}\right\Vert _{L^{\infty }(\mathbb{R}^{d})}\leq C,
\label{e5.6}
\end{equation}%
\begin{equation}
\sup_{x\in \mathbb{R}^{d}}\left( 
\mathchoice {{\setbox0=\hbox{$\displaystyle{\textstyle -}{\int}$ } \vcenter{\hbox{$\textstyle -$
}}\kern-.6\wd0}}{{\setbox0=\hbox{$\textstyle{\scriptstyle -}{\int}$ } \vcenter{\hbox{$\scriptstyle -$
}}\kern-.6\wd0}}{{\setbox0=\hbox{$\scriptstyle{\scriptscriptstyle -}{\int}$
} \vcenter{\hbox{$\scriptscriptstyle -$
}}\kern-.6\wd0}}{{\setbox0=\hbox{$\scriptscriptstyle{\scriptscriptstyle
-}{\int}$ } \vcenter{\hbox{$\scriptscriptstyle -$ }}\kern-.6\wd0}}%
\!\int_{B_{r}(x)}\left\vert \nabla \chi _{T}\right\vert ^{2}dy\right) ^{%
\frac{1}{2}}\leq C_{\sigma }\left( \frac{T}{r}\right) ^{\sigma }\text{ for }%
0<r\leq T,  \label{e5.7}
\end{equation}%
\begin{equation}
\left\vert \chi _{T}(x)-\chi _{T}(y)\right\vert \leq C_{\sigma }T^{1-\sigma
}\left\vert x-y\right\vert ^{\sigma }\text{ for }\left\vert x-y\right\vert
\leq T.  \label{e5.8}
\end{equation}
\end{lemma}

\begin{proof}
Let us first check (\ref{e5.6}). From the inequality (\ref{4.3}), we deduce
that 
\begin{equation}
\sup_{z\in \mathbb{R}^{d},R\geq T}\left( 
\mathchoice {{\setbox0=\hbox{$\displaystyle{\textstyle -}{\int}$ } \vcenter{\hbox{$\textstyle -$
}}\kern-.6\wd0}}{{\setbox0=\hbox{$\textstyle{\scriptstyle -}{\int}$ } \vcenter{\hbox{$\scriptstyle -$
}}\kern-.6\wd0}}{{\setbox0=\hbox{$\scriptstyle{\scriptscriptstyle -}{\int}$
} \vcenter{\hbox{$\scriptscriptstyle -$
}}\kern-.6\wd0}}{{\setbox0=\hbox{$\scriptscriptstyle{\scriptscriptstyle
-}{\int}$ } \vcenter{\hbox{$\scriptscriptstyle -$ }}\kern-.6\wd0}}%
\!\int_{B_{R}(z)}\left\vert \chi _{T}\right\vert ^{2}\right) ^{\frac{1}{2}%
}\leq CT  \label{e5.9}
\end{equation}%
where $C$ depends only on $d$, $\alpha $ and $\beta $. Now fix $%
z=(z_{i})_{1\leq i\leq d}$ in $\mathbb{R}^{d}$ and define 
\begin{equation}
u(y)=\chi _{T,j}(y)+y_{j}-z_{j}\text{, }y\in \mathbb{R}^{d}.  \label{e5.11}
\end{equation}%
Then $u$ solves the equation 
\begin{equation}
\nabla \cdot (A\nabla u)=T^{-2}\chi _{T,j}\text{ in }\mathbb{R}^{d}.
\label{e5.12}
\end{equation}%
Using the De Giorgi-Nash estimates, we obtain 
\begin{eqnarray*}
\sup_{B_{T}(z)}\left\vert u\right\vert &\leq &C\left[ \left( 
\mathchoice {{\setbox0=\hbox{$\displaystyle{\textstyle -}{\int}$ } \vcenter{\hbox{$\textstyle -$
}}\kern-.6\wd0}}{{\setbox0=\hbox{$\textstyle{\scriptstyle -}{\int}$ } \vcenter{\hbox{$\scriptstyle -$
}}\kern-.6\wd0}}{{\setbox0=\hbox{$\scriptstyle{\scriptscriptstyle -}{\int}$
} \vcenter{\hbox{$\scriptscriptstyle -$
}}\kern-.6\wd0}}{{\setbox0=\hbox{$\scriptscriptstyle{\scriptscriptstyle
-}{\int}$ } \vcenter{\hbox{$\scriptscriptstyle -$ }}\kern-.6\wd0}}%
\!\int_{B_{2T}(z)}\left\vert u\right\vert ^{2}\right) ^{\frac{1}{2}%
}+T^{2}\left( 
\mathchoice {{\setbox0=\hbox{$\displaystyle{\textstyle -}{\int}$ } \vcenter{\hbox{$\textstyle -$
}}\kern-.6\wd0}}{{\setbox0=\hbox{$\textstyle{\scriptstyle -}{\int}$ } \vcenter{\hbox{$\scriptstyle -$
}}\kern-.6\wd0}}{{\setbox0=\hbox{$\scriptstyle{\scriptscriptstyle -}{\int}$
} \vcenter{\hbox{$\scriptscriptstyle -$
}}\kern-.6\wd0}}{{\setbox0=\hbox{$\scriptscriptstyle{\scriptscriptstyle
-}{\int}$ } \vcenter{\hbox{$\scriptscriptstyle -$ }}\kern-.6\wd0}}%
\!\int_{B_{2T}(z)}\left\vert T^{-2}\chi _{T,j}\right\vert ^{2}\right) ^{%
\frac{1}{2}}\right] \\
&\leq &CT+C\sup_{x\in \mathbb{R}^{d}}\left( 
\mathchoice {{\setbox0=\hbox{$\displaystyle{\textstyle -}{\int}$ } \vcenter{\hbox{$\textstyle -$
}}\kern-.6\wd0}}{{\setbox0=\hbox{$\textstyle{\scriptstyle -}{\int}$ } \vcenter{\hbox{$\scriptstyle -$
}}\kern-.6\wd0}}{{\setbox0=\hbox{$\scriptstyle{\scriptscriptstyle -}{\int}$
} \vcenter{\hbox{$\scriptscriptstyle -$
}}\kern-.6\wd0}}{{\setbox0=\hbox{$\scriptscriptstyle{\scriptscriptstyle
-}{\int}$ } \vcenter{\hbox{$\scriptscriptstyle -$ }}\kern-.6\wd0}}%
\!\int_{B_{2T}(x)}\left\vert \chi _{T,j}\right\vert ^{2}\right) ^{\frac{1}{2}%
}\leq CT
\end{eqnarray*}%
where $C=C(d,A)$. It follows that $\left\vert \chi _{T,j}(z)\right\vert \leq
CT$. Whence (\ref{e5.6}). Now, concerning (\ref{e5.8}), one uses Schauder
estimates: if $v\in H_{loc}^{1}(\mathbb{R}^{d})$ is a weak solution of $%
-\nabla \cdot (A\nabla v)=h+\nabla \cdot H$ in $B_{2R}(x_{0})$, then for
each $\sigma \in (0,1)$ and for all $x,y\in B_{R}(x_{0})$, 
\begin{eqnarray}
\left\vert v(x)-v(y)\right\vert &\leq &C\left\vert x-y\right\vert ^{\sigma } 
\left[ R^{-\sigma }\left( 
\mathchoice {{\setbox0=\hbox{$\displaystyle{\textstyle -}{\int}$ } \vcenter{\hbox{$\textstyle -$
}}\kern-.6\wd0}}{{\setbox0=\hbox{$\textstyle{\scriptstyle -}{\int}$ } \vcenter{\hbox{$\scriptstyle -$
}}\kern-.6\wd0}}{{\setbox0=\hbox{$\scriptstyle{\scriptscriptstyle -}{\int}$
} \vcenter{\hbox{$\scriptscriptstyle -$
}}\kern-.6\wd0}}{{\setbox0=\hbox{$\scriptscriptstyle{\scriptscriptstyle
-}{\int}$ } \vcenter{\hbox{$\scriptscriptstyle -$ }}\kern-.6\wd0}}%
\!\int_{B_{2R}(x_{0})}\left\vert v\right\vert ^{2}\right) ^{\frac{1}{2}%
}+\sup _{\substack{ z\in B_{R}(x_{0})  \\ 0<r<R}}r^{2-\sigma }\left( 
\mathchoice {{\setbox0=\hbox{$\displaystyle{\textstyle -}{\int}$ } \vcenter{\hbox{$\textstyle -$
}}\kern-.6\wd0}}{{\setbox0=\hbox{$\textstyle{\scriptstyle -}{\int}$ } \vcenter{\hbox{$\scriptstyle -$
}}\kern-.6\wd0}}{{\setbox0=\hbox{$\scriptstyle{\scriptscriptstyle -}{\int}$
} \vcenter{\hbox{$\scriptscriptstyle -$
}}\kern-.6\wd0}}{{\setbox0=\hbox{$\scriptscriptstyle{\scriptscriptstyle
-}{\int}$ } \vcenter{\hbox{$\scriptscriptstyle -$ }}\kern-.6\wd0}}%
\!\int_{B_{r}(z)}\left\vert h\right\vert ^{2}\right) ^{\frac{1}{2}}\right.
\label{e5.10} \\
&&\left. +\sup_{\substack{ z\in B_{R}(x_{0})  \\ 0<r<R}}r^{1-\sigma }\left( 
\mathchoice {{\setbox0=\hbox{$\displaystyle{\textstyle -}{\int}$ } \vcenter{\hbox{$\textstyle -$
}}\kern-.6\wd0}}{{\setbox0=\hbox{$\textstyle{\scriptstyle -}{\int}$ } \vcenter{\hbox{$\scriptstyle -$
}}\kern-.6\wd0}}{{\setbox0=\hbox{$\scriptstyle{\scriptscriptstyle -}{\int}$
} \vcenter{\hbox{$\scriptscriptstyle -$
}}\kern-.6\wd0}}{{\setbox0=\hbox{$\scriptscriptstyle{\scriptscriptstyle
-}{\int}$ } \vcenter{\hbox{$\scriptscriptstyle -$ }}\kern-.6\wd0}}%
\!\int_{B_{r}(z)}\left\vert H\right\vert ^{2}\right) ^{\frac{1}{2}}\right] 
\notag
\end{eqnarray}%
where $C=C(\sigma ,A)$ (see e.g. \cite{Giaquinta} or \cite[Theorem 3.4]{Shen}%
). Assume $x,y\in \mathbb{R}^{d}$ with $\left\vert x-y\right\vert \leq T$.
Applying (\ref{e5.10}) with $2R=T$, $h=T^{-2}\chi _{T,j}$, $H=Ae_{j}$, $%
v=\chi _{T,j}$ and $x_{0}=0$, 
\begin{eqnarray*}
\left\vert \chi _{T,j}(x)-\chi _{T,j}(y)\right\vert &\leq &C\left\vert
x-y\right\vert ^{\sigma }(T^{-\sigma }\left\Vert \chi _{T,j}\right\Vert
_{L^{\infty }}+T^{2-\sigma }\left\Vert T^{-2}\chi _{T,j}\right\Vert
_{L^{\infty }}+T^{1-\sigma }\left\Vert A\right\Vert _{L^{\infty }}) \\
&\leq &CT^{1-\sigma }\left\vert x-y\right\vert ^{\sigma },
\end{eqnarray*}%
where we have used (\ref{e5.6}) for the last inequality above. To obtain (%
\ref{e5.7}), we use Caccioppoli's inequality for $-\nabla \cdot (A\nabla
\chi _{T,j})+T^{-2}\chi _{T,j})=\nabla \cdot (Ae_{j})$ in $B_{2r}(x)$ and (%
\ref{e5.8}) to get 
\begin{eqnarray*}
\mathchoice {{\setbox0=\hbox{$\displaystyle{\textstyle -}{\int}$ } \vcenter{\hbox{$\textstyle -$
}}\kern-.6\wd0}}{{\setbox0=\hbox{$\textstyle{\scriptstyle -}{\int}$ } \vcenter{\hbox{$\scriptstyle -$
}}\kern-.6\wd0}}{{\setbox0=\hbox{$\scriptstyle{\scriptscriptstyle -}{\int}$
} \vcenter{\hbox{$\scriptscriptstyle -$
}}\kern-.6\wd0}}{{\setbox0=\hbox{$\scriptscriptstyle{\scriptscriptstyle
-}{\int}$ } \vcenter{\hbox{$\scriptscriptstyle -$ }}\kern-.6\wd0}}%
\!\int_{B_{r}(x)}\left\vert \nabla \chi _{T,j}(y)\right\vert ^{2}dy &\leq
&Cr^{-2}%
\mathchoice {{\setbox0=\hbox{$\displaystyle{\textstyle -}{\int}$ } \vcenter{\hbox{$\textstyle -$
}}\kern-.6\wd0}}{{\setbox0=\hbox{$\textstyle{\scriptstyle -}{\int}$ } \vcenter{\hbox{$\scriptstyle -$
}}\kern-.6\wd0}}{{\setbox0=\hbox{$\scriptstyle{\scriptscriptstyle -}{\int}$
} \vcenter{\hbox{$\scriptscriptstyle -$
}}\kern-.6\wd0}}{{\setbox0=\hbox{$\scriptscriptstyle{\scriptscriptstyle
-}{\int}$ } \vcenter{\hbox{$\scriptscriptstyle -$ }}\kern-.6\wd0}}%
\!\int_{B_{2r}(x)}\left\vert \chi _{T,j}(y)-\chi _{T,j}(x)\right\vert
^{2}dy+C%
\mathchoice {{\setbox0=\hbox{$\displaystyle{\textstyle -}{\int}$ } \vcenter{\hbox{$\textstyle -$
}}\kern-.6\wd0}}{{\setbox0=\hbox{$\textstyle{\scriptstyle -}{\int}$ } \vcenter{\hbox{$\scriptstyle -$
}}\kern-.6\wd0}}{{\setbox0=\hbox{$\scriptstyle{\scriptscriptstyle -}{\int}$
} \vcenter{\hbox{$\scriptscriptstyle -$
}}\kern-.6\wd0}}{{\setbox0=\hbox{$\scriptscriptstyle{\scriptscriptstyle
-}{\int}$ } \vcenter{\hbox{$\scriptscriptstyle -$ }}\kern-.6\wd0}}%
\!\int_{B_{2r}(x)}\left\vert A\right\vert ^{2}dy \\
&\leq &Cr^{-2}(T^{1-\sigma }r^{\sigma })^{2}+C\leq C\left( \frac{T^{1-\sigma
}}{r^{1-\sigma }}\right) ^{2}\text{ since }0<r\leq T\text{.}
\end{eqnarray*}%
(\ref{e5.7}) follows by replacing $\sigma $ by $1-\sigma $. This finishes
the proof.
\end{proof}

The next result will be used in the forthcoming sections. It involves
Green's function $G:\mathbb{R}^{d}\times \mathbb{R}^{d}\rightarrow \mathbb{R}
$ solution of 
\begin{equation}
-\nabla _{x}\cdot \left( A(x)\nabla _{x}G(x,y)\right) =\delta _{y}(x)\text{
in }\mathbb{R}^{d}.  \label{10.1}
\end{equation}%
The properties of the function $G$ require the definition of the weak-$L^{2}$
space denoted by $L^{2,\infty }(\mathbb{R}^{d})$ (see \cite[Chapter 1]{BL76}
for its definition) together with its topological dual denoted by $L^{2,1}(%
\mathbb{R}^{d})$ (see \cite{Tar07} for its definition).

\begin{proposition}
\label{p10.1}Assume the matrix $A\in L^{\infty }(\mathbb{R}^{d})^{d\times d}$
is uniformly elliptic (see \emph{(\ref{1.2})}) and symmetric. Then equation 
\emph{(\ref{10.1})} has a unique solution in $L^{\infty }(\mathbb{R}%
_{y}^{d};W_{loc}^{1,1}(\mathbb{R}_{x}^{d}))$ satisfying:

\begin{itemize}
\item[(i)] $G(\cdot ,y)\in W_{loc}^{1,2}(\mathbb{R}^{d}\backslash \{y\})$
for all $y\in \mathbb{R}^{d};$

\item[(ii)] There exists $C=C(d)>0$ such that%
\begin{equation}
\left\Vert \nabla _{y}G(x,\cdot )\right\Vert _{L^{2,\infty }(\mathbb{R}%
^{d})}\leq C,  \label{10.2}
\end{equation}%
\begin{equation}
\left\vert G(x,y)\right\vert \leq \left\{ 
\begin{array}{l}
C(1+\left\vert \log \left\vert x-y\right\vert \right\vert )\text{ if }d=2 \\ 
C\left\vert x-y\right\vert ^{2-d}\text{ if }d\geq 3%
\end{array}%
\right. \text{, all }x,y\in \mathbb{R}^{d}\text{ with }x\neq y,  \label{10.3}
\end{equation}%
\begin{equation}
\int_{B_{2R}(x)\backslash B_{R}(x)}\left\vert \nabla _{y}G(x,y)\right\vert
^{q}dy\leq \frac{C}{R^{N(q-1)-q}}\text{ for all }R>0\text{ and }1\leq q\leq
2.  \label{2.8}
\end{equation}%
\noindent If $A$ has H\"{o}lder continuous entries, then for $d\geq 3$ and
for all $x,y\in \mathbb{R}^{d}$ with $x\neq y,$%
\begin{equation}
\left\vert \nabla _{y}G(x,y)\right\vert \leq C\left\vert x-y\right\vert
^{1-d}.  \label{10.4}
\end{equation}
\end{itemize}
\end{proposition}

Properties (\ref{10.3}) and (\ref{10.4}) are classical; see e.g. \cite[%
Theorems 1.1 and 3.3]{Widman}. (\ref{2.8}) is proved in \cite[Lemma 4.2]%
{Lebris}.

\section{Approximation of homogenized coefficients: quantitative estimates}

To simplify the presentation of the results, we assume from now on that $%
A(x,y)=A(y)$. We henceforth denote the mean value by $\left\langle \cdot
\right\rangle $.

\subsection{Approximation by Dirichlet problem}

In the preceding section, we saw that the corrector problem is posed on the
whole of $\mathbb{R}^{d}$. However, if the coefficients of our problem are
periodic (say the function $y\mapsto A(y)$ is $Y$-periodic ($%
Y=(-1/2,1/2)^{d} $), then this problem reduces to another one posed on the
bounded subset $Y$ of $\mathbb{R}^{d}$, and this yields coefficients that
are computable. Contrasting with the periodic setting, the corrector problem
in the general deterministic framework cannot be reduced to a problem on a
bounded domain. Therefore, truncations must be considered, particularly on
large domains like $Q_{R}$ (the closed cube centered at the origin and of
side length $R$) with appropriate boundary conditions. We proceed exactly as
in the random setting (see \cite{BP2004}). We consider the equation 
\begin{equation}
-\nabla _{y}\cdot \left( A(e_{j}+\nabla _{y}\chi _{j,R})\right) =0\text{ in }%
Q_{R},\ \ \chi _{j,R}\in H_{0}^{1}(Q_{R}),  \label{3.3}
\end{equation}%
which possesses a unique solution satisfying 
\begin{equation}
\left( 
\mathchoice {{\setbox0=\hbox{$\displaystyle{\textstyle -}{\int}$ } \vcenter{\hbox{$\textstyle -$
}}\kern-.6\wd0}}{{\setbox0=\hbox{$\textstyle{\scriptstyle -}{\int}$ } \vcenter{\hbox{$\scriptstyle -$
}}\kern-.6\wd0}}{{\setbox0=\hbox{$\scriptstyle{\scriptscriptstyle -}{\int}$
} \vcenter{\hbox{$\scriptscriptstyle -$
}}\kern-.6\wd0}}{{\setbox0=\hbox{$\scriptscriptstyle{\scriptscriptstyle
-}{\int}$ } \vcenter{\hbox{$\scriptscriptstyle -$ }}\kern-.6\wd0}}%
\!\int_{Q_{R}}\left\vert \nabla _{y}\chi _{j,R}\right\vert ^{2}dy\right) ^{%
\frac{1}{2}}\leq C\text{ for any }R\geq 1  \label{i}
\end{equation}%
where $C$ is independent of $R$. Set $\chi _{R}=(\chi _{j,R})_{1\leq j\leq
d} $. We define the effective and approximate effective matrices $A^{\ast }$
and $A_{R}^{\ast }$ respectively, as follows 
\begin{equation}
A^{\ast }=\left\langle A(I+\nabla _{y}\chi )\right\rangle \text{ and }%
A_{R}^{\ast }=%
\mathchoice {{\setbox0=\hbox{$\displaystyle{\textstyle
-}{\int}$ } \vcenter{\hbox{$\textstyle -$
}}\kern-.6\wd0}}{{\setbox0=\hbox{$\textstyle{\scriptstyle -}{\int}$ } \vcenter{\hbox{$\scriptstyle -$
}}\kern-.6\wd0}}{{\setbox0=\hbox{$\scriptstyle{\scriptscriptstyle -}{\int}$
} \vcenter{\hbox{$\scriptscriptstyle -$
}}\kern-.6\wd0}}{{\setbox0=\hbox{$\scriptscriptstyle{\scriptscriptstyle
-}{\int}$ } \vcenter{\hbox{$\scriptscriptstyle -$ }}\kern-.6\wd0}}%
\!\int_{Q_{R}}A(y)(I+\nabla _{y}\chi _{R}(y))dy.  \label{eq5}
\end{equation}

\begin{theorem}
\label{t3.1}The generalized sequence of matrices $A_{R}^{\ast }$ converges,
as $R\rightarrow \infty $, to the homogenized matrix $A^{\ast }$.
\end{theorem}

\begin{proof}
We set, for $x\in Q_{1}$, $w_{j}^{R}(x)=\frac{1}{R}\chi _{j,R}(Rx)$, $%
A_{R}(x)=A(Rx)$ and consider the re-scaled version of (\ref{3.3}) whose $%
w_{j}^{R}$ is solution. It reads as 
\begin{equation}
-\nabla \cdot (A_{R}(e_{j}+\nabla w_{j}^{R}))=0\text{ in }Q_{1}\text{, \ }%
w_{j}^{R}=0\text{ on }\partial Q_{1}.  \label{3.6}
\end{equation}%
Then (\ref{3.6}) possesses a unique solution $w_{j}^{R}\in H_{0}^{1}(Q_{1})$
satisfying the estimate 
\begin{equation}
\left\Vert \nabla w_{j}^{R}\right\Vert _{L^{2}(Q_{1})}\leq C\ \ \ (1\leq
j\leq d)  \label{3.7}
\end{equation}%
where $C>0$ is independent of $R>0$. Proceeding as in the proof of Theorem %
\ref{t1.1}, we derive the existence of $w_{j}\in H_{0}^{1}(Q_{1})$ and $%
w_{j,1}\in L^{2}(Q_{1};B_{\#\mathcal{A}}^{1,2}(\mathbb{R}^{d}))$ such that,
up to a subsequence not relabeled, 
\begin{equation}
w_{j}^{R}\rightarrow w_{j}\text{ in }H_{0}^{1}(Q_{1})\text{-weak and }\nabla
w_{j}^{R}\rightarrow \nabla w_{j}+\nabla _{y}w_{j,1}\text{ in }%
L^{2}(Q_{1})^{d}\text{-weak }\Sigma  \label{4.00}
\end{equation}%
and the couple $(w_{j},w_{j,1})$ solves the equation 
\begin{equation}
\int_{Q_{1}}\left\langle A(e_{j}+\nabla w_{j}+\nabla _{y}w_{j,1})\cdot
(\nabla \psi _{0}+\nabla _{y}\psi _{1})\right\rangle dx=0\ \ \forall (\psi
_{0},\psi _{1})\in \mathcal{C}_{0}^{\infty }(Q_{1})\times (\mathcal{C}%
_{0}^{\infty }(Q_{1})\otimes \mathcal{A}^{\infty }),  \label{4.100}
\end{equation}%
which can be rewritten in the following equivalent form (\ref{4.101})-(\ref%
{4.102}) 
\begin{equation}
\int_{Q_{1}}\left\langle A(e_{j}+\nabla w_{j}+\nabla
_{y}w_{j,1})\right\rangle \cdot \nabla \psi _{0}dx=0\ \ \forall \psi _{0}\in 
\mathcal{C}_{0}^{\infty }(Q_{1})  \label{4.101}
\end{equation}%
and 
\begin{equation}
\left\langle A(e_{j}+\nabla w_{j}+\nabla _{y}w_{j,1})\cdot \nabla
_{y}v\right\rangle dx=0\ \ \forall v\in \mathcal{A}^{\infty }.  \label{4.102}
\end{equation}%
To solve (\ref{4.102}), we consider its weak distributional form 
\begin{equation}
\nabla _{y}\cdot \left( A(e_{j}+\nabla w_{j}+\nabla _{y}w_{j,1})\right) =0%
\text{ in }\mathbb{R}^{d}\text{.}  \label{4.103}
\end{equation}%
So fix $\xi \in \mathbb{R}^{d}$ and consider the problem 
\begin{equation}
\nabla _{y}\cdot \left( A(e_{j}+\xi +\nabla _{y}\pi _{j}(\xi )\right) =0%
\text{ in }\mathbb{R}^{d};\ \pi _{j}(\xi )\in B_{\#\mathcal{A}}^{1,2}(%
\mathbb{R}^{d}).  \label{4.104}
\end{equation}%
Then $\pi _{j}(\xi )$ has the form $\pi _{j}(\xi )=\chi _{j}+\theta _{j}(\xi
)$ where $\chi _{j}$ is the solution of the corrector problem (\ref{1.6})
and $\theta _{j}(\xi )$ solves the equation 
\begin{equation}
\nabla _{y}\cdot \left( A(\xi +\nabla _{y}\theta _{j}(\xi )\right) =0\text{
in }\mathbb{R}^{d};\ \theta _{j}(\xi )\in B_{\#\mathcal{A}}^{1,2}(\mathbb{R}%
^{d}),  \label{4.105}
\end{equation}%
that is, $\theta _{j}(\xi )=\xi \cdot \chi $ where $\chi =(\chi _{k})_{1\leq
k\leq d}$ with $\chi _{k}$ being the solution of (\ref{1.6}) corresponding
to $j=k$ therein. It follows that $\pi _{j}(\xi )=\chi _{j}+\xi \cdot \chi $%
, so that the function $w_{j,1}$, which corresponds to $\pi _{j}(\nabla
w_{j})$, has the form $w_{j,1}=\chi _{j}+\chi \cdot \nabla w_{j}$. Coming
back to (\ref{4.101}) and replacing there $w_{j,1}$ by $\chi _{j}+\chi \cdot
\nabla w_{j}$, we obtain 
\begin{equation}
\int_{Q_{1}}\left\langle A(I+\nabla _{y}\chi )\right\rangle (e_{j}+\nabla
w_{j})\cdot \nabla \psi _{0}dx=0\ \ \forall \psi _{0}\in \mathcal{C}%
_{0}^{\infty }(Q_{1})\text{.}  \label{4.106}
\end{equation}%
This shows that $w_{j}\in H_{0}^{1}(Q_{1})$ solves uniquely the equation 
\begin{equation}
-\nabla \cdot (A^{\ast }(e_{j}+\nabla w_{j}))=0\text{ in }Q_{1}\text{,}
\label{3.9}
\end{equation}%
and further we have, as $R\rightarrow \infty $, 
\begin{equation}
A_{R}(e_{j}+\nabla w_{j}^{R})\rightarrow A^{\ast }(e_{j}+\nabla w_{j})\text{
in }L^{2}(Q_{1})^{d}\text{-weak.}  \label{3.10}
\end{equation}%
To see (\ref{3.10}), we observe that the sequence $(A_{R}(e_{j}+\nabla
w_{j}^{R}))_{R}$ is bounded in $L^{2}(Q_{1})^{d}$ and we choose a test
function $\Phi \in \mathcal{C}_{0}^{\infty }(Q_{1})^{d}$; then by the
sigma-convergence (where we take $A(y)\Phi (x)$ as a test function) we have
from the second convergence result in (\ref{4.00}) that 
\begin{eqnarray*}
\int_{Q_{1}}A_{R}(e_{j}+\nabla w_{j}^{R})\cdot \Phi dx &\rightarrow
&\int_{Q_{1}}\left\langle A(e_{j}+\nabla w_{j}+\nabla _{y}w_{j,1})\cdot \Phi
\right\rangle dx \\
&=&\int_{Q_{1}}\left\langle A(e_{j}+\nabla w_{j}+\nabla
_{y}w_{j,1})\right\rangle \cdot \Phi dx.
\end{eqnarray*}%
But according to (\ref{4.106}), we see that 
\begin{equation*}
\left\langle A(e_{j}+\nabla w_{j}+\nabla _{y}w_{j,1})\right\rangle
=\left\langle A(I+\nabla _{y}\chi )\right\rangle (e_{j}+\nabla
w_{j})=A^{\ast }(e_{j}+\nabla w_{j}).
\end{equation*}%
Now, since (\ref{3.9}) has the form $-\nabla \cdot (A^{\ast }\nabla w_{j})=0$
in $Q_{1}$, ($A^{\ast }$ has constant entries) we infer from the ellipticity
property of $A^{\ast }$ and the uniqueness of the solution to $-\nabla \cdot
(A^{\ast }\nabla w_{j})=0$ in $H_{0}^{1}(Q_{1})$ that $w=(w_{1},...,w_{d})=0$%
. Hence the whole sequence $(w_{j}^{R})_{R}$ weakly converges towards $0$ in 
$H_{0}^{1}(Q_{1})$. Therefore, integrating (\ref{3.10}) over $Q_{1}$, we
readily get (denoting $w^{R}=(w_{1}^{R},...,w_{d}^{R})$)%
\begin{equation*}
A_{R}^{\ast }=%
\mathchoice {{\setbox0=\hbox{$\displaystyle{\textstyle
-}{\int}$ } \vcenter{\hbox{$\textstyle -$
}}\kern-.6\wd0}}{{\setbox0=\hbox{$\textstyle{\scriptstyle -}{\int}$ } \vcenter{\hbox{$\scriptstyle -$
}}\kern-.6\wd0}}{{\setbox0=\hbox{$\scriptstyle{\scriptscriptstyle -}{\int}$
} \vcenter{\hbox{$\scriptscriptstyle -$
}}\kern-.6\wd0}}{{\setbox0=\hbox{$\scriptscriptstyle{\scriptscriptstyle
-}{\int}$ } \vcenter{\hbox{$\scriptscriptstyle -$ }}\kern-.6\wd0}}%
\!\int_{Q_{1}}A(I+\nabla w^{R})dx\rightarrow 
\mathchoice {{\setbox0=\hbox{$\displaystyle{\textstyle -}{\int}$ } \vcenter{\hbox{$\textstyle -$
}}\kern-.6\wd0}}{{\setbox0=\hbox{$\textstyle{\scriptstyle -}{\int}$ } \vcenter{\hbox{$\scriptstyle -$
}}\kern-.6\wd0}}{{\setbox0=\hbox{$\scriptstyle{\scriptscriptstyle -}{\int}$
} \vcenter{\hbox{$\scriptscriptstyle -$
}}\kern-.6\wd0}}{{\setbox0=\hbox{$\scriptscriptstyle{\scriptscriptstyle
-}{\int}$ } \vcenter{\hbox{$\scriptscriptstyle -$ }}\kern-.6\wd0}}%
\!\int_{Q_{1}}A^{\ast }(I+\nabla w)dx=A^{\ast }
\end{equation*}%
as $R\rightarrow \infty $, where $I$ is the $d\times d$ identity matrix.
This completes the proof.
\end{proof}

\subsection{Quantitative estimates}

We study the rate of convergence for the approximation scheme of the
previous subsection, under the assumption that the corrector lies in $B_{%
\mathcal{A}}^{2}(\mathbb{R}^{d})$. To this end, instead of considering the
corrector problem (\ref{1.6}) we rather consider its regularized version (%
\ref{4.2}) which we recall here below: 
\begin{equation*}
-\nabla \cdot A(y)(e_{j}+\nabla \chi _{T,j})+T^{-2}\chi _{T,j}=0\text{ in }%
\mathbb{R}^{d}.
\end{equation*}%
We define the regularized homogenized matrix by 
\begin{equation}
A_{T}^{\ast }=\left\langle A(I+\nabla \chi _{T})\right\rangle ,\ \ \chi
_{T}=(\chi _{T,j})_{1\leq j\leq d}  \label{3.11}
\end{equation}%
Recalling that the homogenized matrix has the form $A^{\ast }=\left\langle
A(I+\nabla \chi )\right\rangle $, we show in (\ref{3.19}) below that $%
\left\vert A^{\ast }-A_{T}^{\ast }\right\vert \leq CT^{-1}$, so that $%
A_{T}^{\ast }\rightarrow A^{\ast }$ as $T\rightarrow \infty $.

With this in mind, we define the approximate regularized coefficients 
\begin{equation}
A_{R,T}^{\ast }=%
\mathchoice {{\setbox0=\hbox{$\displaystyle{\textstyle
-}{\int}$ } \vcenter{\hbox{$\textstyle -$
}}\kern-.6\wd0}}{{\setbox0=\hbox{$\textstyle{\scriptstyle -}{\int}$ } \vcenter{\hbox{$\scriptstyle -$
}}\kern-.6\wd0}}{{\setbox0=\hbox{$\scriptstyle{\scriptscriptstyle -}{\int}$
} \vcenter{\hbox{$\scriptscriptstyle -$
}}\kern-.6\wd0}}{{\setbox0=\hbox{$\scriptscriptstyle{\scriptscriptstyle
-}{\int}$ } \vcenter{\hbox{$\scriptscriptstyle -$ }}\kern-.6\wd0}}%
\!\int_{Q_{R}}A(I+\nabla \chi _{T}^{R}),\ \ \chi _{T}^{R}=(\chi
_{T,j}^{R})_{1\leq j\leq d}  \label{3.12}
\end{equation}%
where $\chi _{T,j}^{R}$ (the regularized approximate corrector) solves the
problem 
\begin{equation}
-\nabla \cdot A(e_{j}+\nabla \chi _{T,j}^{R})+T^{-2}\chi _{T,j}^{R}=0\text{
in }Q_{R},\ \chi _{T,j}^{R}\in H_{0}^{1}(Q_{R}).  \label{3.13}
\end{equation}%
Then 
\begin{equation*}
A_{R,T}^{\ast }\underset{(\ast )}{\overset{R\rightarrow \infty }{\rightarrow 
}}A_{T}^{\ast }\underset{(\ast \ast )}{\overset{T\rightarrow \infty }{%
\rightarrow }}A^{\ast }.
\end{equation*}%
Convergence ($\ast \ast $) will result from (\ref{3.19}) below, while for
convergence ($\ast $), we proceed exactly as in the proof of Theorem \ref%
{t3.1}.

The aim here is to estimate the expression $\left\vert A^{\ast
}-A_{R,T}^{\ast }\right\vert $ in terms of $R$ and $T$, and next take $R=T$
to get the suitable rate of convergence. The following theorem is the main
result of this section.

\begin{theorem}
\label{t3.2}Suppose $\chi \in B_{\mathcal{A}}^{2}(\mathbb{R}^{d})^{d}$. Let $%
\delta \in (0,1)$. There exist $C=C(d,\delta ,A)$ and a continuous function $%
\eta _{\delta }:[1,\infty )\rightarrow \lbrack 0,\infty )$, which depends
only on $A$ and $\delta $, such that $\lim_{t\rightarrow \infty }\eta
_{\delta }(t)=0$ and 
\begin{equation}
\left\vert A^{\ast }-A_{T,T}^{\ast }\right\vert \leq C\eta _{\delta }(T)%
\text{ for all }T\geq 1.  \label{3.17}
\end{equation}
\end{theorem}

The proof breaks down into several steps which are of independent interest.

\begin{lemma}
\label{l3.2}Let $u\in B_{\mathcal{A}}^{2}(\mathbb{R}^{d})$. For any $%
0<R<\infty $, 
\begin{equation}
\left\vert 
\mathchoice {{\setbox0=\hbox{$\displaystyle{\textstyle
-}{\int}$ } \vcenter{\hbox{$\textstyle -$
}}\kern-.6\wd0}}{{\setbox0=\hbox{$\textstyle{\scriptstyle -}{\int}$ } \vcenter{\hbox{$\scriptstyle -$
}}\kern-.6\wd0}}{{\setbox0=\hbox{$\scriptstyle{\scriptscriptstyle -}{\int}$
} \vcenter{\hbox{$\scriptscriptstyle -$
}}\kern-.6\wd0}}{{\setbox0=\hbox{$\scriptscriptstyle{\scriptscriptstyle
-}{\int}$ } \vcenter{\hbox{$\scriptscriptstyle -$ }}\kern-.6\wd0}}%
\!\int_{Q_{R}}u-\left\langle u\right\rangle \right\vert \leq \sup_{y\in 
\mathbb{R}^{d}}%
\mathchoice {{\setbox0=\hbox{$\displaystyle{\textstyle
-}{\int}$ } \vcenter{\hbox{$\textstyle -$
}}\kern-.6\wd0}}{{\setbox0=\hbox{$\textstyle{\scriptstyle -}{\int}$ } \vcenter{\hbox{$\scriptstyle -$
}}\kern-.6\wd0}}{{\setbox0=\hbox{$\scriptstyle{\scriptscriptstyle -}{\int}$
} \vcenter{\hbox{$\scriptscriptstyle -$
}}\kern-.6\wd0}}{{\setbox0=\hbox{$\scriptscriptstyle{\scriptscriptstyle
-}{\int}$ } \vcenter{\hbox{$\scriptscriptstyle -$ }}\kern-.6\wd0}}%
\!\int_{Q_{R}}\left\vert u(t+y)-u(t)\right\vert dt.  \label{3.18}
\end{equation}
\end{lemma}

\begin{proof}
Let $u\in B_{\mathcal{A}}^{2}(\mathbb{R}^{d})$. We know that, for any $y\in 
\mathbb{R}^{d}$, 
\begin{equation*}
\mathchoice {{\setbox0=\hbox{$\displaystyle{\textstyle -}{\int}$ } \vcenter{\hbox{$\textstyle -$
}}\kern-.6\wd0}}{{\setbox0=\hbox{$\textstyle{\scriptstyle -}{\int}$ } \vcenter{\hbox{$\scriptstyle -$
}}\kern-.6\wd0}}{{\setbox0=\hbox{$\scriptstyle{\scriptscriptstyle -}{\int}$
} \vcenter{\hbox{$\scriptscriptstyle -$
}}\kern-.6\wd0}}{{\setbox0=\hbox{$\scriptscriptstyle{\scriptscriptstyle
-}{\int}$ } \vcenter{\hbox{$\scriptscriptstyle -$ }}\kern-.6\wd0}}%
\!\int_{Q_{R}(y)}u-%
\mathchoice {{\setbox0=\hbox{$\displaystyle{\textstyle
-}{\int}$ } \vcenter{\hbox{$\textstyle -$
}}\kern-.6\wd0}}{{\setbox0=\hbox{$\textstyle{\scriptstyle -}{\int}$ } \vcenter{\hbox{$\scriptstyle -$
}}\kern-.6\wd0}}{{\setbox0=\hbox{$\scriptstyle{\scriptscriptstyle -}{\int}$
} \vcenter{\hbox{$\scriptscriptstyle -$
}}\kern-.6\wd0}}{{\setbox0=\hbox{$\scriptscriptstyle{\scriptscriptstyle
-}{\int}$ } \vcenter{\hbox{$\scriptscriptstyle -$ }}\kern-.6\wd0}}%
\!\int_{Q_{R}}u=%
\mathchoice {{\setbox0=\hbox{$\displaystyle{\textstyle
-}{\int}$ } \vcenter{\hbox{$\textstyle -$
}}\kern-.6\wd0}}{{\setbox0=\hbox{$\textstyle{\scriptstyle -}{\int}$ } \vcenter{\hbox{$\scriptstyle -$
}}\kern-.6\wd0}}{{\setbox0=\hbox{$\scriptstyle{\scriptscriptstyle -}{\int}$
} \vcenter{\hbox{$\scriptscriptstyle -$
}}\kern-.6\wd0}}{{\setbox0=\hbox{$\scriptscriptstyle{\scriptscriptstyle
-}{\int}$ } \vcenter{\hbox{$\scriptscriptstyle -$ }}\kern-.6\wd0}}%
\!\int_{Q_{R}}\left( u(t+y)-u(t)\right) dt.
\end{equation*}%
Now, let $k>1$ be an integer; we have $Q_{kR}=\cup
_{i=1}^{k^{d}}Q_{R}(x_{i}) $ for some $x_{i}\in \mathbb{R}^{d}$, so that 
\begin{equation*}
\left\vert 
\mathchoice {{\setbox0=\hbox{$\displaystyle{\textstyle -}{\int}$
} \vcenter{\hbox{$\textstyle -$
}}\kern-.6\wd0}}{{\setbox0=\hbox{$\textstyle{\scriptstyle -}{\int}$ } \vcenter{\hbox{$\scriptstyle -$
}}\kern-.6\wd0}}{{\setbox0=\hbox{$\scriptstyle{\scriptscriptstyle -}{\int}$
} \vcenter{\hbox{$\scriptscriptstyle -$
}}\kern-.6\wd0}}{{\setbox0=\hbox{$\scriptscriptstyle{\scriptscriptstyle
-}{\int}$ } \vcenter{\hbox{$\scriptscriptstyle -$ }}\kern-.6\wd0}}%
\!\int_{Q_{kR}}u-%
\mathchoice {{\setbox0=\hbox{$\displaystyle{\textstyle
-}{\int}$ } \vcenter{\hbox{$\textstyle -$
}}\kern-.6\wd0}}{{\setbox0=\hbox{$\textstyle{\scriptstyle -}{\int}$ } \vcenter{\hbox{$\scriptstyle -$
}}\kern-.6\wd0}}{{\setbox0=\hbox{$\scriptstyle{\scriptscriptstyle -}{\int}$
} \vcenter{\hbox{$\scriptscriptstyle -$
}}\kern-.6\wd0}}{{\setbox0=\hbox{$\scriptscriptstyle{\scriptscriptstyle
-}{\int}$ } \vcenter{\hbox{$\scriptscriptstyle -$ }}\kern-.6\wd0}}%
\!\int_{Q_{R}}u\right\vert \leq \frac{1}{k^{d}}\sum_{i=1}^{k^{d}}\left\vert 
\mathchoice {{\setbox0=\hbox{$\displaystyle{\textstyle -}{\int}$ } \vcenter{\hbox{$\textstyle -$
}}\kern-.6\wd0}}{{\setbox0=\hbox{$\textstyle{\scriptstyle -}{\int}$ } \vcenter{\hbox{$\scriptstyle -$
}}\kern-.6\wd0}}{{\setbox0=\hbox{$\scriptstyle{\scriptscriptstyle -}{\int}$
} \vcenter{\hbox{$\scriptscriptstyle -$
}}\kern-.6\wd0}}{{\setbox0=\hbox{$\scriptscriptstyle{\scriptscriptstyle
-}{\int}$ } \vcenter{\hbox{$\scriptscriptstyle -$ }}\kern-.6\wd0}}%
\!\int_{Q_{R}(x_{i})}u-%
\mathchoice {{\setbox0=\hbox{$\displaystyle{\textstyle -}{\int}$ } \vcenter{\hbox{$\textstyle -$
}}\kern-.6\wd0}}{{\setbox0=\hbox{$\textstyle{\scriptstyle -}{\int}$ } \vcenter{\hbox{$\scriptstyle -$
}}\kern-.6\wd0}}{{\setbox0=\hbox{$\scriptstyle{\scriptscriptstyle -}{\int}$
} \vcenter{\hbox{$\scriptscriptstyle -$
}}\kern-.6\wd0}}{{\setbox0=\hbox{$\scriptscriptstyle{\scriptscriptstyle
-}{\int}$ } \vcenter{\hbox{$\scriptscriptstyle -$ }}\kern-.6\wd0}}%
\!\int_{Q_{R}}u\right\vert \leq \sup_{y\in \mathbb{R}^{d}}\left\vert 
\mathchoice {{\setbox0=\hbox{$\displaystyle{\textstyle -}{\int}$ } \vcenter{\hbox{$\textstyle -$
}}\kern-.6\wd0}}{{\setbox0=\hbox{$\textstyle{\scriptstyle -}{\int}$ } \vcenter{\hbox{$\scriptstyle -$
}}\kern-.6\wd0}}{{\setbox0=\hbox{$\scriptstyle{\scriptscriptstyle -}{\int}$
} \vcenter{\hbox{$\scriptscriptstyle -$
}}\kern-.6\wd0}}{{\setbox0=\hbox{$\scriptscriptstyle{\scriptscriptstyle
-}{\int}$ } \vcenter{\hbox{$\scriptscriptstyle -$ }}\kern-.6\wd0}}%
\!\int_{Q_{R}(y)}u-%
\mathchoice {{\setbox0=\hbox{$\displaystyle{\textstyle
-}{\int}$ } \vcenter{\hbox{$\textstyle -$
}}\kern-.6\wd0}}{{\setbox0=\hbox{$\textstyle{\scriptstyle -}{\int}$ } \vcenter{\hbox{$\scriptstyle -$
}}\kern-.6\wd0}}{{\setbox0=\hbox{$\scriptstyle{\scriptscriptstyle -}{\int}$
} \vcenter{\hbox{$\scriptscriptstyle -$
}}\kern-.6\wd0}}{{\setbox0=\hbox{$\scriptscriptstyle{\scriptscriptstyle
-}{\int}$ } \vcenter{\hbox{$\scriptscriptstyle -$ }}\kern-.6\wd0}}%
\!\int_{Q_{R}}u\right\vert .
\end{equation*}%
Letting $k\rightarrow \infty $ we are led to (\ref{3.18}).
\end{proof}

The next result evaluates the difference between $A^{\ast }$ and $%
A_{T}^{\ast }$.

\begin{lemma}
\label{l3.3}Assume that $\chi _{j}$ (defined by \emph{(\ref{1.6})}) belongs
to $B_{\mathcal{A}}^{2}(\mathbb{R}^{d})$. There exists $C=C(d,A)$ such that 
\begin{equation}
\left\vert A^{\ast }-A_{T}^{\ast }\right\vert \leq CT^{-1}.  \label{3.19}
\end{equation}
\end{lemma}

\begin{proof}
First, let us set $v=\chi _{T,j}-\chi _{j}$. Then $v$ solves the equation $%
-\nabla \cdot (A\nabla v)+T^{-2}v=-T^{-2}\chi _{j}$ in $\mathbb{R}^{d}$. It
follows from Lemma \ref{l4.1} that 
\begin{equation*}
\sup_{x\in \mathbb{R}^{d}}%
\mathchoice {{\setbox0=\hbox{$\displaystyle{\textstyle -}{\int}$ } \vcenter{\hbox{$\textstyle -$
}}\kern-.6\wd0}}{{\setbox0=\hbox{$\textstyle{\scriptstyle -}{\int}$ } \vcenter{\hbox{$\scriptstyle -$
}}\kern-.6\wd0}}{{\setbox0=\hbox{$\scriptstyle{\scriptscriptstyle -}{\int}$
} \vcenter{\hbox{$\scriptscriptstyle -$
}}\kern-.6\wd0}}{{\setbox0=\hbox{$\scriptscriptstyle{\scriptscriptstyle
-}{\int}$ } \vcenter{\hbox{$\scriptscriptstyle -$ }}\kern-.6\wd0}}%
\!\int_{Q_{T}(x)}\left( \left\vert \nabla v\right\vert ^{2}+T^{-2}\left\vert
v\right\vert ^{2}\right) \leq CT^{-2}\sup_{x\in \mathbb{R}^{d}}%
\mathchoice {{\setbox0=\hbox{$\displaystyle{\textstyle -}{\int}$ } \vcenter{\hbox{$\textstyle -$
}}\kern-.6\wd0}}{{\setbox0=\hbox{$\textstyle{\scriptstyle -}{\int}$ } \vcenter{\hbox{$\scriptstyle -$
}}\kern-.6\wd0}}{{\setbox0=\hbox{$\scriptstyle{\scriptscriptstyle -}{\int}$
} \vcenter{\hbox{$\scriptscriptstyle -$
}}\kern-.6\wd0}}{{\setbox0=\hbox{$\scriptscriptstyle{\scriptscriptstyle
-}{\int}$ } \vcenter{\hbox{$\scriptscriptstyle -$ }}\kern-.6\wd0}}%
\!\int_{Q_{T}(x)}\left\vert \chi _{j}\right\vert ^{2}\leq CT^{-2}.
\end{equation*}%
In the last inequality above, we have used the fact that $\chi _{j}\in B_{%
\mathcal{A}}^{2}(\mathbb{R}^{d})$, so that 
\begin{equation*}
\sup_{x\in \mathbb{R}^{d},T>0}%
\mathchoice {{\setbox0=\hbox{$\displaystyle{\textstyle -}{\int}$ } \vcenter{\hbox{$\textstyle -$
}}\kern-.6\wd0}}{{\setbox0=\hbox{$\textstyle{\scriptstyle -}{\int}$ } \vcenter{\hbox{$\scriptstyle -$
}}\kern-.6\wd0}}{{\setbox0=\hbox{$\scriptstyle{\scriptscriptstyle -}{\int}$
} \vcenter{\hbox{$\scriptscriptstyle -$
}}\kern-.6\wd0}}{{\setbox0=\hbox{$\scriptscriptstyle{\scriptscriptstyle
-}{\int}$ } \vcenter{\hbox{$\scriptscriptstyle -$ }}\kern-.6\wd0}}%
\!\int_{Q_{T}(x)}\left\vert \chi _{j}\right\vert ^{2}\leq C.
\end{equation*}%
The above inequality stems from the fact that $\lim_{T\rightarrow \infty }%
\mathchoice {{\setbox0=\hbox{$\displaystyle{\textstyle -}{\int}$ } \vcenter{\hbox{$\textstyle -$
}}\kern-.6\wd0}}{{\setbox0=\hbox{$\textstyle{\scriptstyle -}{\int}$ } \vcenter{\hbox{$\scriptstyle -$
}}\kern-.6\wd0}}{{\setbox0=\hbox{$\scriptstyle{\scriptscriptstyle -}{\int}$
} \vcenter{\hbox{$\scriptscriptstyle -$
}}\kern-.6\wd0}}{{\setbox0=\hbox{$\scriptscriptstyle{\scriptscriptstyle
-}{\int}$ } \vcenter{\hbox{$\scriptscriptstyle -$ }}\kern-.6\wd0}}%
\!\int_{Q_{T}(x)}\left\vert \chi _{j}\right\vert ^{2}$ exists uniformly in $%
x\in \mathbb{R}^{d}$. We infer 
\begin{equation}
\sup_{x\in \mathbb{R}^{d}}\left( 
\mathchoice {{\setbox0=\hbox{$\displaystyle{\textstyle -}{\int}$ } \vcenter{\hbox{$\textstyle -$
}}\kern-.6\wd0}}{{\setbox0=\hbox{$\textstyle{\scriptstyle -}{\int}$ } \vcenter{\hbox{$\scriptstyle -$
}}\kern-.6\wd0}}{{\setbox0=\hbox{$\scriptstyle{\scriptscriptstyle -}{\int}$
} \vcenter{\hbox{$\scriptscriptstyle -$
}}\kern-.6\wd0}}{{\setbox0=\hbox{$\scriptscriptstyle{\scriptscriptstyle
-}{\int}$ } \vcenter{\hbox{$\scriptscriptstyle -$ }}\kern-.6\wd0}}%
\!\int_{Q_{T}(x)}\left\vert A\nabla (\chi _{T,j}-\chi _{j})\right\vert
^{2}\right) ^{\frac{1}{2}}\leq \left\Vert A\right\Vert _{\infty }\sup_{x\in 
\mathbb{R}^{d}}\left( 
\mathchoice {{\setbox0=\hbox{$\displaystyle{\textstyle
-}{\int}$ } \vcenter{\hbox{$\textstyle -$
}}\kern-.6\wd0}}{{\setbox0=\hbox{$\textstyle{\scriptstyle -}{\int}$ } \vcenter{\hbox{$\scriptstyle -$
}}\kern-.6\wd0}}{{\setbox0=\hbox{$\scriptstyle{\scriptscriptstyle -}{\int}$
} \vcenter{\hbox{$\scriptscriptstyle -$
}}\kern-.6\wd0}}{{\setbox0=\hbox{$\scriptscriptstyle{\scriptscriptstyle
-}{\int}$ } \vcenter{\hbox{$\scriptscriptstyle -$ }}\kern-.6\wd0}}%
\!\int_{Q_{T}(x)}\left\vert \nabla (\chi _{T,j}-\chi _{j})\right\vert
^{2}\right) ^{\frac{1}{2}}\leq CT^{-1}.  \label{3.20}
\end{equation}%
Now, using Lemma \ref{l3.2} with $u=A\nabla (\chi _{T}-\chi )$, we obtain 
\begin{equation}
\left\vert 
\mathchoice {{\setbox0=\hbox{$\displaystyle{\textstyle -}{\int}$
} \vcenter{\hbox{$\textstyle -$
}}\kern-.6\wd0}}{{\setbox0=\hbox{$\textstyle{\scriptstyle -}{\int}$ } \vcenter{\hbox{$\scriptstyle -$
}}\kern-.6\wd0}}{{\setbox0=\hbox{$\scriptstyle{\scriptscriptstyle -}{\int}$
} \vcenter{\hbox{$\scriptscriptstyle -$
}}\kern-.6\wd0}}{{\setbox0=\hbox{$\scriptscriptstyle{\scriptscriptstyle
-}{\int}$ } \vcenter{\hbox{$\scriptscriptstyle -$ }}\kern-.6\wd0}}%
\!\int_{Q_{T}}A\nabla (\chi -\chi _{T})-(A^{\ast }-A_{T}^{\ast })\right\vert
\leq \sup_{y\in \mathbb{R}^{d}}%
\mathchoice {{\setbox0=\hbox{$\displaystyle{\textstyle -}{\int}$ } \vcenter{\hbox{$\textstyle -$
}}\kern-.6\wd0}}{{\setbox0=\hbox{$\textstyle{\scriptstyle -}{\int}$ } \vcenter{\hbox{$\scriptstyle -$
}}\kern-.6\wd0}}{{\setbox0=\hbox{$\scriptstyle{\scriptscriptstyle -}{\int}$
} \vcenter{\hbox{$\scriptscriptstyle -$
}}\kern-.6\wd0}}{{\setbox0=\hbox{$\scriptscriptstyle{\scriptscriptstyle
-}{\int}$ } \vcenter{\hbox{$\scriptscriptstyle -$ }}\kern-.6\wd0}}%
\!\int_{Q_{T}}\left\vert A\nabla (\chi -\chi _{T})(t+y)-A\nabla (\chi -\chi
_{T})(t)\right\vert dt.  \label{3.21}
\end{equation}%
However, from the equality 
\begin{equation*}
\mathchoice {{\setbox0=\hbox{$\displaystyle{\textstyle -}{\int}$ } \vcenter{\hbox{$\textstyle -$
}}\kern-.6\wd0}}{{\setbox0=\hbox{$\textstyle{\scriptstyle -}{\int}$ } \vcenter{\hbox{$\scriptstyle -$
}}\kern-.6\wd0}}{{\setbox0=\hbox{$\scriptstyle{\scriptscriptstyle -}{\int}$
} \vcenter{\hbox{$\scriptscriptstyle -$
}}\kern-.6\wd0}}{{\setbox0=\hbox{$\scriptscriptstyle{\scriptscriptstyle
-}{\int}$ } \vcenter{\hbox{$\scriptscriptstyle -$ }}\kern-.6\wd0}}%
\!\int_{Q_{T}}A\nabla (\chi -\chi _{T})(t+y)dt=%
\mathchoice {{\setbox0=\hbox{$\displaystyle{\textstyle -}{\int}$ } \vcenter{\hbox{$\textstyle -$
}}\kern-.6\wd0}}{{\setbox0=\hbox{$\textstyle{\scriptstyle -}{\int}$ } \vcenter{\hbox{$\scriptstyle -$
}}\kern-.6\wd0}}{{\setbox0=\hbox{$\scriptstyle{\scriptscriptstyle -}{\int}$
} \vcenter{\hbox{$\scriptscriptstyle -$
}}\kern-.6\wd0}}{{\setbox0=\hbox{$\scriptscriptstyle{\scriptscriptstyle
-}{\int}$ } \vcenter{\hbox{$\scriptscriptstyle -$ }}\kern-.6\wd0}}%
\!\int_{Q_{T}(y)}A\nabla (\chi -\chi _{T})(t)dt
\end{equation*}%
associated to the inequality 
\begin{equation*}
\mathchoice {{\setbox0=\hbox{$\displaystyle{\textstyle -}{\int}$ } \vcenter{\hbox{$\textstyle -$
}}\kern-.6\wd0}}{{\setbox0=\hbox{$\textstyle{\scriptstyle -}{\int}$ } \vcenter{\hbox{$\scriptstyle -$
}}\kern-.6\wd0}}{{\setbox0=\hbox{$\scriptstyle{\scriptscriptstyle -}{\int}$
} \vcenter{\hbox{$\scriptscriptstyle -$
}}\kern-.6\wd0}}{{\setbox0=\hbox{$\scriptscriptstyle{\scriptscriptstyle
-}{\int}$ } \vcenter{\hbox{$\scriptscriptstyle -$ }}\kern-.6\wd0}}%
\!\int_{Q_{T}(y)}\left\vert A\nabla (\chi -\chi _{T})(t)\right\vert dt\leq
\left( 
\mathchoice {{\setbox0=\hbox{$\displaystyle{\textstyle -}{\int}$ } \vcenter{\hbox{$\textstyle -$
}}\kern-.6\wd0}}{{\setbox0=\hbox{$\textstyle{\scriptstyle -}{\int}$ } \vcenter{\hbox{$\scriptstyle -$
}}\kern-.6\wd0}}{{\setbox0=\hbox{$\scriptstyle{\scriptscriptstyle -}{\int}$
} \vcenter{\hbox{$\scriptscriptstyle -$
}}\kern-.6\wd0}}{{\setbox0=\hbox{$\scriptscriptstyle{\scriptscriptstyle
-}{\int}$ } \vcenter{\hbox{$\scriptscriptstyle -$ }}\kern-.6\wd0}}%
\!\int_{Q_{T}(y)}\left\vert A\nabla (\chi -\chi _{T})\right\vert ^{2}\right)
^{\frac{1}{2}},
\end{equation*}%
we deduce that the right-hand side of (\ref{3.21}) is bounded by $%
2\sup_{y\in \mathbb{R}^{d}}\left( 
\mathchoice {{\setbox0=\hbox{$\displaystyle{\textstyle -}{\int}$ } \vcenter{\hbox{$\textstyle -$
}}\kern-.6\wd0}}{{\setbox0=\hbox{$\textstyle{\scriptstyle -}{\int}$ } \vcenter{\hbox{$\scriptstyle -$
}}\kern-.6\wd0}}{{\setbox0=\hbox{$\scriptstyle{\scriptscriptstyle -}{\int}$
} \vcenter{\hbox{$\scriptscriptstyle -$
}}\kern-.6\wd0}}{{\setbox0=\hbox{$\scriptscriptstyle{\scriptscriptstyle
-}{\int}$ } \vcenter{\hbox{$\scriptscriptstyle -$ }}\kern-.6\wd0}}%
\!\int_{Q_{T}(y)}\left\vert A\nabla (\chi -\chi _{T})\right\vert ^{2}\right)
^{\frac{1}{2}}$. Taking into account (\ref{3.20}), we get immediately 
\begin{equation*}
\left\vert 
\mathchoice {{\setbox0=\hbox{$\displaystyle{\textstyle -}{\int}$
} \vcenter{\hbox{$\textstyle -$
}}\kern-.6\wd0}}{{\setbox0=\hbox{$\textstyle{\scriptstyle -}{\int}$ } \vcenter{\hbox{$\scriptstyle -$
}}\kern-.6\wd0}}{{\setbox0=\hbox{$\scriptstyle{\scriptscriptstyle -}{\int}$
} \vcenter{\hbox{$\scriptscriptstyle -$
}}\kern-.6\wd0}}{{\setbox0=\hbox{$\scriptscriptstyle{\scriptscriptstyle
-}{\int}$ } \vcenter{\hbox{$\scriptscriptstyle -$ }}\kern-.6\wd0}}%
\!\int_{Q_{T}}A\nabla (\chi -\chi _{T})-(A^{\ast }-A_{T}^{\ast })\right\vert
\leq CT^{-1}.
\end{equation*}%
It follows that 
\begin{equation*}
\left\vert A^{\ast }-A_{T}^{\ast }\right\vert \leq \left\vert 
\mathchoice {{\setbox0=\hbox{$\displaystyle{\textstyle -}{\int}$ } \vcenter{\hbox{$\textstyle -$
}}\kern-.6\wd0}}{{\setbox0=\hbox{$\textstyle{\scriptstyle -}{\int}$ } \vcenter{\hbox{$\scriptstyle -$
}}\kern-.6\wd0}}{{\setbox0=\hbox{$\scriptstyle{\scriptscriptstyle -}{\int}$
} \vcenter{\hbox{$\scriptscriptstyle -$
}}\kern-.6\wd0}}{{\setbox0=\hbox{$\scriptscriptstyle{\scriptscriptstyle
-}{\int}$ } \vcenter{\hbox{$\scriptscriptstyle -$ }}\kern-.6\wd0}}%
\!\int_{Q_{T}}A\nabla (\chi -\chi _{T})-(A^{\ast }-A_{T}^{\ast })\right\vert
+%
\mathchoice {{\setbox0=\hbox{$\displaystyle{\textstyle -}{\int}$ } \vcenter{\hbox{$\textstyle -$
}}\kern-.6\wd0}}{{\setbox0=\hbox{$\textstyle{\scriptstyle -}{\int}$ } \vcenter{\hbox{$\scriptstyle -$
}}\kern-.6\wd0}}{{\setbox0=\hbox{$\scriptstyle{\scriptscriptstyle -}{\int}$
} \vcenter{\hbox{$\scriptscriptstyle -$
}}\kern-.6\wd0}}{{\setbox0=\hbox{$\scriptscriptstyle{\scriptscriptstyle
-}{\int}$ } \vcenter{\hbox{$\scriptscriptstyle -$ }}\kern-.6\wd0}}%
\!\int_{Q_{T}}\left\vert A\nabla (\chi -\chi _{T})\right\vert \leq CT^{-1}.
\end{equation*}
\end{proof}

We are now in a position to prove the theorem.

\begin{proof}[Proof of Theorem \protect\ref{t3.2}]
We decompose $A^{\ast }-A_{R,T}^{\ast }$ as follows:%
\begin{equation*}
A^{\ast }-A_{R,T}^{\ast }=(A^{\ast }-A_{T}^{\ast })+(A_{T}^{\ast
}-A_{R,T}^{\ast }).
\end{equation*}%
We consider each term separately.

Lemma \ref{l3.3} yields $\left\vert A^{\ast }-A_{T}^{\ast }\right\vert \leq
CT^{-1}$. As regard the term $A_{T}^{\ast }-A_{R,T}^{\ast }$, we observe
that $v=\chi _{T,j}-\chi _{T,j}^{R}$\ solves the equation 
\begin{equation*}
-\nabla \cdot A\nabla v+T^{-2}v=0\text{ in }Q_{R}\text{ and }v=\chi _{T,j}%
\text{ on }\partial Q_{R},
\end{equation*}%
so that, proceeding exactly as in \cite[Proof of Lemma 1]{BP2004} we obtain 
\begin{equation}
\left\vert A_{T}^{\ast }-A_{R,T}^{\ast }\right\vert ^{2}\leq C\left(
T^{2}\exp (-c_{1}TR^{\delta })+R^{\delta -1}\right)  \label{3.22}
\end{equation}%
where $0<\delta <1$, and $C$ and $c_{1}>0$ are independent of $R$ and $T$.
We emphasize that in \cite{BP2004}, the above inequality has been obtained
without any help stemming from the random character of the problem. It
relies only on the bounds of the Green function of the operator $-\nabla
\cdot A\nabla +T^{-2}$ and on the bounds of the regularized corrector $\chi
_{T}$.

Choosing $R=T$ in (\ref{3.22}), we define the function 
\begin{equation*}
\eta _{\delta }(t)=\frac{1}{t}+t\exp \left( -\frac{c_{1}}{2}t^{1+\delta
}\right) +t^{\frac{1}{2}(\delta -1)}\text{ for }t\geq 1\text{.}
\end{equation*}%
Then $\eta _{\delta }$ is continuous with $\lim_{t\rightarrow \infty }\eta
_{\delta }(t)=0$. We see that 
\begin{equation*}
\left\vert A^{\ast }-A_{T,T}^{\ast }\right\vert \leq C\eta _{\delta }(T)%
\text{ for any }T\geq 1\text{.}
\end{equation*}%
This concludes the proof of the theorem.
\end{proof}

\section{Convergence rates: the asymptotic periodic setting}

\subsection{Preliminary results}

Let us consider the corrector problem (\ref{1.6}) in which $A$ satisfies in
addition the assumptions (\ref{2.1}) and (\ref{2.2}) below: $A=A_{0}+A_{per}$
where 
\begin{equation}
A_{per}\in L_{per}^{2}(Y)^{d\times d}\text{ and }\left\{ 
\begin{array}{l}
A_{0}\in L^{2}(\mathbb{R}^{d})^{d\times d}\text{ for }d\geq 3 \\ 
A_{0}\in (L^{2}(\mathbb{R}^{2})\cap L^{2,1}(\mathbb{R}^{2}))^{2\times 2}%
\text{ for }d=2.%
\end{array}%
\right.  \label{2.1}
\end{equation}%
The matrix $A_{per}$ is symmetric and further 
\begin{equation}
\alpha \left\vert \lambda \right\vert ^{2}\leq A_{per}(y)\lambda \cdot
\lambda \leq \beta \left\vert \lambda \right\vert ^{2}\text{ for all }%
\lambda \in \mathbb{R}^{d}\text{ and a.e. }y\in \mathbb{R}^{d}.  \label{2.2}
\end{equation}%
Let $H_{\infty ,per}^{1}(\mathbb{R}^{d})=\{u\in L_{\infty ,per}^{2}(\mathbb{R%
}^{d}):\nabla u\in L_{\infty ,per}^{2}(\mathbb{R}^{d})^{d}\}$ where $%
L_{\infty ,per}^{2}(\mathbb{R}^{d})=L_{0}^{2}(\mathbb{R}^{d})+L_{per}^{2}(Y)$
and $L_{0}^{2}(\mathbb{R}^{d})$ is the completion of $\mathcal{C}_{0}(%
\mathbb{R}^{d})$ with respect to the seminorm (\ref{0.2}).

\begin{proposition}
\label{l2.1}Let $H$ be a function such that $H\in L^{2}(\mathbb{R}^{d})^{d}$
for $d\geq 3$ and $H\in (L^{2}(\mathbb{R}^{d})\cap L^{2,1}(\mathbb{R}%
^{d}))^{d}$ for $d=2$. Assume $A$ satisfies \emph{(\ref{1.2})}. Then there
exists $u_{0}\in L^{p}(\mathbb{R}^{d})$ with $\nabla u_{0}\in L^{2}(\mathbb{R%
}^{d})^{d}$ such that $u_{0}$ solves the equation 
\begin{equation}
-\nabla \cdot A\nabla u_{0}=\nabla \cdot H\text{ in }\mathbb{R}^{d}
\label{8.1}
\end{equation}%
where $p=2^{\ast }\equiv 2d/(d-2)$ for $d\geq 3$ and $p=\infty $ for $d=2$.
\end{proposition}

\begin{proof}
1) We first assume that $d\geq 3$. Let $Y^{1,2}=\{u\in L^{2^{\ast }}(\mathbb{%
R}^{d}):\nabla u\in L^{2}(\mathbb{R}^{d})^{d}\}$ (where $2^{\ast }=2d/(d-2)$%
), and equip $Y^{1,2}$ with the norm $\left\Vert u\right\Vert
_{Y^{1,2}}=\left\Vert u\right\Vert _{L^{2^{\ast }}(\mathbb{R}%
^{d})}+\left\Vert \nabla u\right\Vert _{L^{2}(\mathbb{R}^{d})}$, which makes
it a Banach space. By the Sobolev's inequality (see \cite[Theorem 4.31, page
102]{Adams}), there exists a positive constant $C=C(d)$ such that 
\begin{equation}
\left\Vert u\right\Vert _{L^{2^{\ast }}(\mathbb{R}^{d})}\leq C\left\Vert
\nabla u\right\Vert _{L^{2}(\mathbb{R}^{d})}\ \ \forall u\in Y^{1,2}.
\label{6.3}
\end{equation}%
We deduce from (\ref{6.3}) that (\ref{8.1}) possesses a unique solution in $%
Y^{1,2}$ satisfying the inequality 
\begin{equation}
\left\Vert u_{0}\right\Vert _{Y^{1,2}}\leq C\left\Vert H\right\Vert _{L^{2}(%
\mathbb{R}^{d})}.  \label{6.4}
\end{equation}

2) Now assume that $d=2$. We use $G(x,y)$ defined by (\ref{10.1}) to express 
$u_{0}$ as 
\begin{equation}
u_{0}(x)=-\int_{\mathbb{R}^{d}}\nabla _{y}G(x,y)\cdot H(y)dy.  \label{8.6}
\end{equation}%
The expression (\ref{8.6}) makes sense since we may proceed by approximation
by assuming first that $H\in \mathcal{C}_{0}^{\infty }(\mathbb{R}^{2})^{2}$
and next using the density of $\mathcal{C}_{0}^{\infty }(\mathbb{R}^{2})$ in 
$L^{2,1}(\mathbb{R}^{2})$ together with property (\ref{10.3}) to conclude.
So, using the generalized H\"{o}lder inequality, we get 
\begin{equation}
\left\Vert u_{0}\right\Vert _{L^{\infty }(\mathbb{R}^{2})}\leq \sup_{x\in 
\mathbb{R}^{2}}\left\Vert \nabla _{y}G(x,\cdot )\right\Vert _{L^{2,\infty }(%
\mathbb{R}^{2})}\left\Vert H\right\Vert _{L^{2,1}(\mathbb{R}^{2})}.
\label{8.2}
\end{equation}%
This completes the proof.
\end{proof}

\begin{lemma}
\label{l1.2}Assume that $A=A_{0}+A_{per}$ where $A$ and $A_{per}$ are
uniformly elliptic (see \emph{(\ref{1.2}) }and \emph{(\ref{2.2})}) with $%
A_{0}$ and $A_{per}$ being as in \emph{(\ref{2.1})}. Assume further that $%
A_{per}$ and $A$ are H\"{o}lder continuous. Let the number $p$ be as in
Proposition \emph{\ref{l2.1}}. Let $\chi _{j,per}\in H_{per}^{1}(Y)$ be the
unique solution of 
\begin{equation}
-\nabla _{y}\cdot \left( A_{per}(e_{j}+\nabla _{y}\chi _{j,per})\right) =0%
\text{ in }Y,\ \ \int_{Y}\chi _{j,per}dy=0.  \label{2.3}
\end{equation}%
Then \emph{(\ref{1.6})} possesses a unique solution $\chi _{j}\in H_{\infty
,per}^{1}(Y)$ (in the sense of Theorem \emph{\ref{t4.1}}) satisfying $\chi
_{j}=\chi _{j,0}+\chi _{j,per}$ where $\chi _{j,0}\in L^{p}(\mathbb{R}^{d})$
with $\nabla _{y}\chi _{j,0}\in L^{2}(\mathbb{R}^{d})^{d}$, and 
\begin{equation}
\left\Vert \chi _{j}\right\Vert _{L^{\infty }(\mathbb{R}^{d})}\leq C
\label{*2}
\end{equation}%
where $C=C(d,A)$.
\end{lemma}

\begin{proof}
First, we notice that if $\chi _{j,per}$ solves (\ref{2.3}) then $\chi
_{j,0} $ solves 
\begin{equation*}
-\nabla _{y}\cdot \left( A\nabla _{y}\chi _{j,0}\right) =\nabla _{y}\cdot
\left( A_{0}(e_{j}+\nabla _{y}\chi _{j,per})\right) \text{ in }\mathbb{R}%
^{d}.
\end{equation*}%
Assuming that $A_{per}$ is H\"{o}lder continuous, we get $\nabla _{y}\chi
_{j,per}\in L^{\infty }(Y)$. Because of the property of $A_{0}$ given by (%
\ref{2.1}), it follows that $g=A_{0}(e_{j}+\nabla _{y}\chi _{j,per})$
belongs to $L^{2}(\mathbb{R}^{d})^{d}$ (resp. $(L^{2}(\mathbb{R}^{2})\cap
L^{2}(\mathbb{R}^{2}))^{2}$) for $d\geq 3$ (resp. $d=2$). Proposition \ref%
{l2.1} implies that $\chi _{j,0}\in L^{p}(\mathbb{R}^{d})$ with $\nabla
_{y}\chi _{j,0}\in L^{2}(\mathbb{R}^{d})^{d}$ for $d\geq 3$. Hence in that
case one has $\left\langle \chi _{j,0}\right\rangle =0$ and $\left\langle
\nabla _{y}\chi _{j,0}\right\rangle =0$. This proves that $\chi _{j}=\chi
_{j,per}+\chi _{j,0}\in H_{\infty ,per}^{1}(Y)$ for $d\geq 3$. Now, for $d=2$
we have $\chi _{j,0}\in L_{0}^{2}(\mathbb{R}^{2})$ since $\chi _{j,0}$
vanishes at infinity. Indeed, we use (\ref{8.2}) to get 
\begin{equation*}
\left\Vert \chi _{j,0}\right\Vert _{L^{\infty }(\mathbb{R}^{2})}\leq
\sup_{x\in \mathbb{R}^{d}}\left\Vert \nabla _{y}G(x,\cdot )\right\Vert
_{L^{2,\infty }(\mathbb{R}^{2})}\left\Vert g\right\Vert _{L^{2,1}(\mathbb{R}%
^{2})}
\end{equation*}%
and proceed as in \cite[Section 3, page 14]{Lebris} (first approximate $g$
by smooth functions in $\mathcal{C}_{0}^{\infty }(\mathbb{R}^{2})^{2}$) to
show that $\chi _{j,0}\in L_{0}^{2}(\mathbb{R}^{2})$.

Let us now verify (\ref{*2}). We drop for a while the index $j$ and just
write $\chi =\chi _{0}+\chi _{per}$, where the couple $(\chi _{per},\chi
_{0})$ solves the system 
\begin{equation}
-\nabla _{y}\cdot \left( A_{per}(e_{j}+\nabla _{y}\chi _{per})\right) =0%
\text{ in }Y,  \label{2.4}
\end{equation}%
\begin{equation}
-\nabla _{y}\cdot \left( A\nabla _{y}\chi _{0}\right) =\nabla _{y}\cdot
\left( A_{0}(e_{j}+\nabla _{y}\chi _{per})\right) \text{ in }\mathbb{R}^{d}.
\label{2.5}
\end{equation}%
It is well known that $\chi _{per}$ is bounded in $L^{\infty }(\mathbb{R}%
^{d})$. Let us first deal with $\chi _{0}$. Let $g=A_{0}(e_{j}+\nabla
_{y}\chi _{per})$ and use the Green function defined in Proposition \ref%
{p10.1} to express $\chi _{0}$ as 
\begin{equation}
\chi _{0}(y)=-\int_{\mathbb{R}^{d}}\nabla _{x}G(y,x)g(x)dx.  \label{2.7}
\end{equation}%
We recall that $G$ satisfies the inequality (\ref{10.4}) for $d\geq 3$\ and (%
\ref{10.2}) for $d=2$, respectively.

We first assume that $d\geq 3$. Let $y\in \mathbb{R}^{d}$ and choose $\gamma
\in \mathcal{C}_{0}^{\infty }(B_{2}(y))$ such that $\gamma =1$ on $B_{1}(y)$
and $0\leq \gamma \leq 1$. We write $\chi _{0}$ as 
\begin{align*}
\chi _{0}(y)& =-\int_{\mathbb{R}^{d}}\nabla _{x}G(y,x)\cdot g(x)\gamma
(x)dx-\int_{\mathbb{R}^{d}}\nabla _{x}G(y,x)\cdot g(x)(1-\gamma (x))dx \\
& =v_{1}(y)+v_{2}(y).
\end{align*}%
As for $v_{1}$, owing to (\ref{10.4}), we have 
\begin{equation*}
\left\vert v_{1}(y)\right\vert \leq C\left\Vert g\right\Vert _{L^{\infty }(%
\mathbb{R}^{d})}\int_{B_{2}(y)}\left\vert x-y\right\vert ^{1-d}dx\leq
C\left\Vert g\right\Vert _{L^{\infty }(\mathbb{R}^{d})}
\end{equation*}%
where $C=C(d)$. As for $v_{2}$, (\ref{10.4}) and H\"{o}lder's inequality
imply, 
\begin{equation*}
\left\vert v_{2}(y)\right\vert \leq C\left\Vert g\right\Vert _{L^{2}(\mathbb{%
R}^{d})}\left( \int_{\mathbb{R}^{d}\backslash B_{2}(y)}\left\vert
x-y\right\vert ^{2-2d}dx\right) \leq C\left\Vert g\right\Vert _{L^{2}(%
\mathbb{R}^{d})}
\end{equation*}%
since $2d-2>d$ for $d\geq 3$.

When $d=2$, we use (\ref{10.2}) to get 
\begin{equation*}
\left\Vert \chi _{0}\right\Vert _{L^{\infty }(\mathbb{R}^{2})}\leq
\sup_{x\in \mathbb{R}^{2}}\left\Vert \nabla _{y}G(x,\cdot )\right\Vert
_{L^{2,\infty }(\mathbb{R}^{2})}\left\Vert g\right\Vert _{L^{2,1}(\mathbb{R}%
^{2})}\leq C\left\Vert g\right\Vert _{L^{2,1}(\mathbb{R}^{d})}\text{.}
\end{equation*}
\end{proof}

\begin{lemma}
\label{l1.4}\emph{(i)} Let $g\in L^{2}(\mathbb{R}^{d})+L_{per}^{2}(Y)$ be
such that $\left\langle g\right\rangle =0$. Then there exists at least one
function $u\in H_{\infty ,per}^{1}(Y)$ such that 
\begin{equation}
\Delta u=g\text{ in }\mathbb{R}^{d},\ \left\langle u\right\rangle =0.
\label{2.9}
\end{equation}%
\emph{(ii)} Assume further that $g\in L^{\infty }(\mathbb{R}^{d})$ and $u$
is bounded; then $u,\nabla u\in \mathcal{B}_{\infty ,per}(\mathbb{R}^{d})$
and 
\begin{equation}
\left\Vert \nabla u\right\Vert _{L^{\infty }(\mathbb{R}^{d})}\leq
C\left\Vert g\right\Vert _{L^{\infty }(\mathbb{R}^{d})},  \label{2.10}
\end{equation}%
where $C>0$ depends only on $d$.
\end{lemma}

\begin{proof}
(i) We write $g=g_{0}+g_{per}$ with $g_{0}\in L^{2}(\mathbb{R}^{d})$ and $%
g_{per}\in L_{per}^{2}(Y)$. Since $\left\langle g\right\rangle =0$, we have $%
\left\langle g_{per}\right\rangle =0$. So let $v_{per}\in H_{per}^{1}(Y)$ be
the unique solution of 
\begin{equation*}
\Delta v_{per}=g_{per}\text{ in }Y\text{, }\left\langle v_{per}\right\rangle
=0.
\end{equation*}%
We observe that if $u$ solves (\ref{2.9}), then $u$ has the form $%
u=v_{0}+v_{per}$ where $v_{0}\in H^{1}(\mathbb{R}^{d})$ solves the problem 
\begin{equation*}
\Delta v_{0}=g_{0}\text{ in }\mathbb{R}^{d}\text{, }v_{0}(x)\rightarrow 0%
\text{ as }\left\vert x\right\vert \rightarrow \infty .
\end{equation*}%
Since $g_{0}\in L^{2}(\mathbb{R}^{d})$, $v_{0}$ easily expresses as 
\begin{equation*}
v_{0}(x)=\int_{\mathbb{R}^{d}}\Gamma _{0}(x-y)g_{0}(y)dy
\end{equation*}%
where $\Gamma _{0}$ denotes the fundamental solution of the Laplacian in $%
\mathbb{R}^{d}$ (with pole at the origin). This shows the existence of $u$
in $H^{1}(\mathbb{R}^{d})+H_{per}^{1}(Y)\subset H_{\infty ,per}^{1}(\mathbb{R%
}^{d})$.

Let us check (ii). First, since (\ref{2.9}) is satisfied, $u$ is thus the
Newtonian potential of $g$ in $\mathbb{R}^{d}$, and by \cite[page 71,
Problem 4.8 (a)]{Gilbarg}, $\nabla u\in \mathcal{C}_{loc}^{1/2}(\mathbb{R}%
^{d})$. Using therefore the continuity of $\nabla u$ together with the fact
that $\nabla u$ also lies in $L_{\infty ,per}^{2}(\mathbb{R}^{d})$, we infer
that $\nabla u\in \mathcal{B}_{\infty ,per}(\mathbb{R}^{d})=\mathcal{C}_{0}(%
\mathbb{R}^{d})\oplus \mathcal{C}_{per}(Y)$. We then proceed as in the proof
of Lemma \ref{l1.2} to obtain $u\in \mathcal{B}_{\infty ,per}(\mathbb{R}%
^{d}) $. This completes the proof.
\end{proof}

The following result is a mere consequence of the preceding lemma. Its proof
is therefore left to the reader.

\begin{corollary}
\label{c1.1}Let $\mathbf{g}$\ be a solenoidal vector in $(L^{2}(\mathbb{R}%
^{d})+L_{per}^{2}(Y))^{d}$\ (i.e. $\nabla \cdot \mathbf{g}=0$) with $%
\left\langle \mathbf{g}\right\rangle =0$. Then there exists a skew symmetric
matrix $G$\ with entries in $L_{\infty ,per}^{2}(Y)$\ such that $\mathbf{g}%
=\nabla \cdot G$. If further $\mathbf{g}$\ belongs to $L^{\infty }(\mathbb{R}%
^{d})^{d}$, then $G$\ has entries in $\mathcal{B}_{\infty ,per}(\mathbb{R}%
^{d})$ and 
\begin{equation}
\left\Vert G\right\Vert _{L^{\infty }(\mathbb{R}^{d})}\leq C\left\Vert 
\mathbf{g}\right\Vert _{L^{\infty }(\mathbb{R}^{d})}.  \label{*5}
\end{equation}
\end{corollary}

\subsection{Convergence rates: proof of Theorem \protect\ref{t5.1}}

Let $u_{\varepsilon }$, $u_{0}\in H_{0}^{1}(\Omega )$ be the weak solutions
of (\ref{1.1}) and (\ref{1.4}) respectively. Assume further that $u_{0}\in
H^{2}(\Omega )$. We suppose in addition that $\Omega $ is sufficiently
smooth. For any function $h\in L_{loc}^{2}(\mathbb{R}^{d})$ and $\varepsilon
>0$ we define $h^{\varepsilon }$ by $h^{\varepsilon }(x)=h(x/\varepsilon )$
for $x\in \mathbb{R}^{d}$. We define the first order approximation of $%
u_{\varepsilon }$ by $v_{\varepsilon }=u_{0}+\varepsilon \chi ^{\varepsilon
}\nabla u_{0}$. Let $w_{\varepsilon }=u_{\varepsilon }-v_{\varepsilon
}+z_{\varepsilon }$ where $z_{\varepsilon }\in H^{1}(\Omega )$ is the weak
solution of the following problem 
\begin{equation}
-\nabla \cdot A^{\varepsilon }\nabla z_{\varepsilon }=0\text{ in }\Omega 
\text{, }z_{\varepsilon }=\varepsilon \chi ^{\varepsilon }\nabla u_{0}\text{
on }\partial \Omega .  \label{5.3}
\end{equation}%
$z_{\varepsilon }$ will be used to approximate the difference of $%
u_{\varepsilon }$ and its first order approximation $v_{\varepsilon }$.

\begin{lemma}
\label{l5.1}The function $w_{\varepsilon }$ solves the problem 
\begin{equation}
\left\{ 
\begin{array}{l}
-\nabla \cdot \left( A^{\varepsilon }\nabla w_{\varepsilon }\right) =\nabla
\cdot \left( A^{\varepsilon }(\nabla u_{0}+(\nabla _{y}\chi )^{\varepsilon
}\nabla u_{0}-\left\langle A(\nabla u_{0}+\nabla _{y}\chi \nabla
u_{0})\right\rangle \right) \\ 
\ \ \ \ \ \ \ \ \ \ \ \ \ \ \ \ \ \ \ \ \ \ \ \ \ \ \ \ \ \ \ \ +\varepsilon
\nabla \cdot \left( A^{\varepsilon }\nabla ^{2}u_{0}\chi ^{\varepsilon
}\right) \text{ in }\Omega \\ 
\ \ \ \ \ \ \ \ \ \ \ \ \ \ \ \ \ \ \ w_{\varepsilon }=0\text{ on }\partial
\Omega .%
\end{array}%
\right.  \label{5.4}
\end{equation}
\end{lemma}

\begin{proof}
Let $y=x/\varepsilon $. Then 
\begin{equation*}
A(y)\nabla w_{\varepsilon }=A(y)(\nabla u_{\varepsilon }-\nabla u_{0}-\nabla
_{y}\chi (y)\nabla u_{0}-\varepsilon (\nabla ^{2}u_{0})\chi (y)+\nabla
z_{\varepsilon }),
\end{equation*}%
hence 
\begin{align*}
\nabla \cdot A\left( y\right) \nabla w_{\varepsilon }& =\nabla \cdot
A(y)\nabla u_{\varepsilon }-\nabla \cdot A(y)\nabla u_{0}-\nabla \cdot
A(y)(\nabla _{y}\chi (y)\nabla u_{0}) \\
& -\varepsilon \nabla \cdot (A(y)(\nabla ^{2}u_{0})\chi (y)) \\
& =\nabla \cdot A^{\ast }\nabla u_{0}-\nabla \cdot A(y)\nabla u_{0}-\nabla
\cdot A(y)(\nabla _{y}\chi (y)\nabla u_{0}) \\
& -\varepsilon \nabla \cdot (A(y)(\nabla ^{2}u_{0})\chi (y)).
\end{align*}%
But 
\begin{equation*}
A^{\ast }\nabla u_{0}=\left\langle A(\nabla u_{0}+\nabla _{y}\chi \nabla
u_{0})\right\rangle \equiv \left\langle A(I+\nabla _{y}\chi )\nabla
u_{0}\right\rangle .
\end{equation*}

Thus 
\begin{align*}
-\nabla \cdot A^{\varepsilon }\nabla w_{\varepsilon }& =\nabla \cdot \left[
A\left( y\right) (\nabla u_{0}+\nabla _{y}\chi \nabla u_{0})-\left\langle
A(\nabla u_{0}+\nabla _{y}\chi \nabla u_{0})\right\rangle \right] \\
& +\varepsilon \nabla \cdot (A(y)(\nabla ^{2}u_{0})\chi (y)),
\end{align*}%
which is the statement of the lemma.
\end{proof}

Set 
\begin{equation*}
a_{ij}(y)=b_{ij}(y)+\sum_{k=1}^{d}b_{ik}(y)\frac{\partial \chi _{j}}{%
\partial y_{k}}(y)-b_{ij}^{\ast }
\end{equation*}%
where $A^{\ast }=(b_{ij}^{\ast })_{1\leq i,j\leq d}$ is the homogenized
matrix, and let $a_{j}=(a_{ij})_{1\leq i\leq d}$. Then $a_{j}\in \lbrack
L^{\infty }(\mathbb{R}^{d})\cap L_{\infty ,per}^{2}(Y)]^{d}$ with $\nabla
\cdot a_{j}=0$ and $\left\langle a_{j}\right\rangle =0$. Hence by Corollary %
\ref{c1.1}, there is a skew-symmetric matrix $G_{j}$ with entries in $%
\mathcal{A}=\mathcal{B}_{\infty ,per}(Y)$ such that $a_{j}=\nabla _{y}\cdot
G_{j}$. Moreover in view of (\ref{*5}) in Corollary \ref{c1.1}, we have 
\begin{equation*}
\left\Vert G_{j}\right\Vert _{\infty }\leq C\left\Vert a_{j}\right\Vert
_{\infty }.
\end{equation*}%
With this in mind and recalling that $G_{j}$ is skew-symmetric, Eq. (\ref%
{5.4}) becomes 
\begin{equation}
-\nabla \cdot A\left( \frac{x}{\varepsilon }\right) \nabla w_{\varepsilon
}=\varepsilon \nabla \cdot \left( r_{1}^{\varepsilon }+r_{2}^{\varepsilon
}\right)  \label{5.5}
\end{equation}%
where 
\begin{equation*}
r_{1}^{\varepsilon }(x)=\sum_{j=1}^{d}G_{j}(y)\nabla \frac{\partial u_{0}}{%
\partial x_{j}}(x)\text{ and }r_{2}^{\varepsilon }(x)=A(y)\nabla
^{2}u_{0}(x)\chi (y)\text{ with }y=\frac{x}{\varepsilon }.
\end{equation*}%
Now, since $w_{\varepsilon }\in H_{0}^{1}(\Omega )$, it follows from the
ellipticity of $A$ (see (\ref{1.2})) that 
\begin{align*}
\alpha \left\Vert \nabla w_{\varepsilon }\right\Vert _{L^{2}(\Omega )}& \leq
\varepsilon \left( \left\Vert r_{1}^{\varepsilon }\right\Vert _{L^{2}(\Omega
)}+\left\Vert r_{2}^{\varepsilon }\right\Vert _{L^{2}(\Omega )}\right) \\
& \leq C\varepsilon \left\Vert u_{0}\right\Vert _{H^{2}(\Omega )}
\end{align*}%
where $C=C(d,A,\Omega )$.

We have just proved the following result.

\begin{proposition}
\label{p5.1}Let $\Omega $ be a smooth bounded domain in $\mathbb{R}^{d}$.
Suppose that $A=A_{0}+A_{per}$ and $A$ and $A_{per}$ are uniformly elliptic
(see \emph{(\ref{1.2})} and \emph{(\ref{2.2})}). For $f\in L^{2}(\Omega )$,
let $u_{\varepsilon }$, $u_{0}$ and $v_{\varepsilon }$ be weak solutions of
Dirichlet problems \emph{(\ref{1.1})}, \emph{(\ref{1.4})} and \emph{(\ref%
{5.3})}, respectively. Assume $u_{0}\in H^{2}(\Omega )$. There $%
C=C(d,A,\Omega )$ such that 
\begin{equation}
\left\Vert u_{\varepsilon }-u_{0}-\varepsilon \chi ^{\varepsilon }\nabla
u_{0}+z_{\varepsilon }\right\Vert _{H_{0}^{1}(\Omega )}\leq C\varepsilon
\left\Vert u_{0}\right\Vert _{H^{2}(\Omega )}.  \label{5.6}
\end{equation}
\end{proposition}

The estimate of the deviation of $u_{\varepsilon }$ and $v_{\varepsilon }$
is a consequence of the following lemma whose proof is postponed to the next
section and is obtained as a special case of the proof of a general result
formulated as Lemma \ref{l5.3}. Observe that in Lemma \ref{l5.3} we replace $%
T^{-1}\left\Vert \chi _{T}\right\Vert _{L^{\infty }(\mathbb{R}^{d})}$ by $%
\varepsilon $ (see Remark \ref{r5.2}).

\begin{lemma}
\label{l5.2'}Assume $u_{0}\in H^{2}(\Omega )$. Let $z_{\varepsilon }$ be the
solution of problem \emph{(\ref{5.3})}. There exists $C=C(d,A,\Omega )$ such
that 
\begin{equation}
\left\Vert z_{\varepsilon }\right\Vert _{H^{1}(\Omega )}\leq C\varepsilon ^{%
\frac{1}{2}}\left\Vert u_{0}\right\Vert _{H^{2}(\Omega )}.  \label{5.7}
\end{equation}
\end{lemma}

\begin{proof}[Proof of Theorem \protect\ref{t5.1}]
Since $\Omega $ is a $\mathcal{C}^{1,1}$-bounded domain in $\mathbb{R}^{d}$
and the matrix $A^{\ast }$ has constant entries, it is known that $u_{0}$
satisfies the inequality 
\begin{equation}
\left\Vert u_{0}\right\Vert _{H^{2}(\Omega )}\leq C\left\Vert f\right\Vert
_{L^{2}(\Omega )},\ C=C(d,\alpha ,\Omega )>0.  \label{5.7'}
\end{equation}%
Using (\ref{5.6}) together with (\ref{5.7}) and (\ref{5.7'}), we arrive at 
\begin{align*}
\left\Vert u_{\varepsilon }-u_{0}-\varepsilon \chi ^{\varepsilon }\nabla
u_{0}\right\Vert _{H^{1}(\Omega )}& \leq \left\Vert u_{\varepsilon
}-u_{0}-\varepsilon \chi ^{\varepsilon }\nabla u_{0}+z_{\varepsilon
}\right\Vert _{H_{0}^{1}(\Omega )}+\left\Vert z_{\varepsilon }\right\Vert
_{H^{1}(\Omega )} \\
& \leq C\varepsilon ^{\frac{1}{2}}\left\Vert u_{0}\right\Vert _{H^{2}(\Omega
)}\leq C\varepsilon ^{\frac{1}{2}}\left\Vert f\right\Vert _{L^{2}(\Omega )},
\end{align*}%
and derive the statement of (\ref{5.8}) in Theorem \ref{t5.1}. As for (\ref%
{1.14}) we proceed exactly as in the proof of (\ref{Eq02}) in the proof of
Theorem \ref{t1.4}; see in particular Remark \ref{r5.3} in the next section.
This concludes the proof of Theorem \ref{t5.1}.
\end{proof}

\section{Convergence rates: the asymptotic almost periodic setting}

\subsection{Preliminaries}

We treat the asymptotic almost periodic case in a general way, dropping
restrictions (\ref{2.1}) and (\ref{2.2}). The results in this section extend
those of the preceding section as well as those in the almost periodic
setting obtained in \cite{Shen}.

We recall that a bounded continuous function $u$ defined on $\mathbb{R}^{d}$
is asymptotically almost periodic if there exists a couple $(v,w)\in AP(%
\mathbb{R}^{d})\times \mathcal{C}_{0}(\mathbb{R}^{d})$ such that $u=v+w$. We
denote by $\mathcal{B}_{\infty ,AP}(\mathbb{R}^{d})=AP(\mathbb{R}^{d})+%
\mathcal{C}_{0}(\mathbb{R}^{d})$ the Banach algebra of such functions. We
denote by $H_{\infty ,AP}^{1}(\mathbb{R}^{d})$ the Sobolev-type space
attached to the Besicovitch space $B_{\mathcal{A}}^{2}(\mathbb{R}^{d})\equiv
L_{\infty ,AP}^{2}(\mathbb{R}^{d})=L_{0}^{2}(\mathbb{R}^{d})+B_{AP}^{2}(%
\mathbb{R}^{d})$: $H_{\infty ,AP}^{1}(\mathbb{R}^{d})=\{u\in L_{\infty
,AP}^{2}(\mathbb{R}^{d}):\nabla u\in L_{\infty ,AP}^{2}(\mathbb{R}%
^{d})^{d}\} $. Here $L_{0}^{2}(\mathbb{R}^{d})$ is the completion of $%
\mathcal{C}_{0}(\mathbb{R}^{d})$ with respect to the seminorm (\ref{0.2})
while $B_{AP}^{2}(\mathbb{R}^{d})$ is the Besicovitch space associated to
the algebra $AP(\mathbb{R}^{d})$. We also denote by $\mathcal{C}_{b}(\mathbb{%
R}^{d})$ the algebra of real-valued bounded continuous functions defined on $%
\mathbb{R}^{d}$.

The following characterization of $\mathcal{B}_{\infty ,AP}(\mathbb{R}^{d})$
is a useful tool for the considerations below.

\begin{proposition}
\label{p11.1}Let $u\in \mathcal{C}_{b}(\mathbb{R}^{d})$. Then $u\in \mathcal{%
B}_{\infty ,AP}(\mathbb{R}^{d})$ if and only if 
\begin{equation}
\sup_{y\in \mathbb{R}^{d}}\inf_{z\in \mathbb{R}^{d},\left\vert z\right\vert
\leq L}\left\Vert u(\cdot +y)-u(\cdot +z)\right\Vert _{L^{\infty }(\mathbb{R}%
^{d}\backslash B_{R})}\rightarrow 0\text{ as }L\rightarrow \infty \text{ and 
}R\rightarrow 0\text{.}  \label{11.1}
\end{equation}
\end{proposition}

\begin{proof}
A set $E$ in $\mathbb{R}^{d}$ is relatively dense if there exists $L>0$ such
that $\mathbb{R}^{d}=E+B_{L}$ (where we recall that $B_{L}=B(0,L)$), that
is, any $x\in \mathbb{R}^{d}$ expresses as a sum $y+z$ with $y\in E$ and $%
z\in B_{L}$. This being so, it is known that $u\in \mathcal{C}_{b}(\mathbb{R}%
^{d})$ lies in $\mathcal{B}_{\infty ,AP}(\mathbb{R}^{d})$ if and only if for
any $\varepsilon >0$, there is $R=R(\varepsilon )>0$ such that the set 
\begin{equation*}
\left\{ \tau \in \mathbb{R}^{d}:\left\vert u(t+\tau )-u(t)\right\vert
<\varepsilon \ \ \forall \left\vert t\right\vert \geq R\right\}
\end{equation*}%
is relatively dense; see e.g. \cite[Chap. 5, Theorem 5]{Zaidman}. But this
is shown to be equivalent to (\ref{11.1}).
\end{proof}

\begin{remark}
\label{r11.1}\emph{We notice that, for any }$u\in \mathcal{C}_{b}(\mathbb{R}%
^{d})$\emph{, }%
\begin{equation*}
\lim_{R\rightarrow \infty }\left( \sup_{\left\vert y\right\vert \leq
R}\left\vert u(y)\right\vert \right) =\lim_{R\rightarrow 0}\left(
\sup_{\left\vert y\right\vert \geq R}\left\vert u(y)\right\vert \right) .
\end{equation*}%
\emph{In view of the above equality we may replace (\ref{11.1}) by }%
\begin{equation}
\sup_{y\in \mathbb{R}^{d}}\inf_{\left\vert z\right\vert \leq L}\left\Vert
u(\cdot +y)-u(\cdot +z)\right\Vert _{L^{\infty }(B_{R})}\rightarrow 0\text{%
\emph{\ as }}L,R\rightarrow \infty  \label{11.2}
\end{equation}%
\emph{since the limits in (\ref{11.1}) and (\ref{11.2}) are the same. In
practice we will rather use (\ref{11.2}).}
\end{remark}

\begin{definition}
\label{d5.1}\emph{For a function }$u\in \mathcal{B}_{\infty ,AP}(\mathbb{R}%
^{d})$\emph{\ we define the modulus of asymptotic almost periodicity of }$u$%
\emph{\ by }%
\begin{equation}
\rho _{u}(L,R)=\sup_{y\in \mathbb{R}^{d}}\inf_{\left\vert z\right\vert \leq
L}\left\Vert u(\cdot +y)-u(\cdot +z)\right\Vert _{L^{\infty }(B_{R})}\text{%
\emph{\ for }}L,R>0.  \label{11.3}
\end{equation}%
\emph{In particular we set }%
\begin{equation}
\rho (L,R)=\sup_{y\in \mathbb{R}^{d}}\inf_{\left\vert z\right\vert \leq
L}\left\Vert A(\cdot +y)-A(\cdot +z)\right\Vert _{L^{\infty }(B_{R})}\text{, 
}L,R>0.  \label{11.4}
\end{equation}
\end{definition}

\begin{remark}
\label{r5.1}\emph{Observe that if }$R=\infty $\emph{\ (that is, }$B_{R}=%
\mathbb{R}^{d}$\emph{) in (\ref{11.3}), then }$u\in \mathcal{B}_{\infty ,AP}(%
\mathbb{R}^{d})$\emph{\ is almost periodic if and only if }$\rho
_{u}(L,\infty )\rightarrow 0$\emph{\ as }$L\rightarrow \infty $\emph{.}
\end{remark}

\subsection{Estimates of approximate correctors}

First we recall that the approximate corrector $\chi _{T}=(\chi
_{T,j})_{1\leq j\leq d}$ is defined as the distributional solution of 
\begin{equation}
-\nabla \cdot \left( A(e_{j}+\nabla \chi _{T,j})\right) +T^{-2}\chi _{T,j}=0%
\text{ in }\mathbb{R}^{d},\ \ \chi _{T,j}\in H_{\infty ,AP}^{1}(\mathbb{R}%
^{d})  \label{11.5}
\end{equation}%
where $A\in (L_{\infty ,AP}^{2}(\mathbb{R}^{d})\cap L^{\infty }(\mathbb{R}%
^{d}))^{d\times d}$ is symmetric and uniformly elliptic.

In all that follows in this section we assume that $A\in (\mathcal{B}%
_{\infty ,AP}(\mathbb{R}^{d}))^{d\times d}$.

\begin{theorem}
\label{t11.1}Let $T\geq 1$. Then $\chi _{T}\in \mathcal{B}_{\infty ,AP}(%
\mathbb{R}^{d})$ and for any $x_{0},y,z\in \mathbb{R}^{d}$, 
\begin{equation}
\left\Vert \chi _{T}(\cdot +y)-\chi _{T}(\cdot +z)\right\Vert _{L^{\infty
}(B_{R}(x_{0}))}\leq CT\left\Vert A(\cdot +y)-A(\cdot +z)\right\Vert
_{L^{\infty }(B_{R}(x_{0}))}  \label{11.9}
\end{equation}%
for any $R>2T$, where $C=C(d,A)$.
\end{theorem}

\begin{proof}
Fix $R>2T$. We need to show that, for any $x_{0},y,z\in \mathbb{R}^{d}$ and $%
t\in B_{R}(x_{0})$, 
\begin{equation*}
\left\vert \chi _{T}(t+y)-\chi _{T}(t+z)\right\vert \leq CT\left\Vert
B(\cdot +y)-B(\cdot +z)\right\Vert _{L^{\infty }(B_{R}(x_{0}))}.
\end{equation*}%
We follow the same approach as in the proof of \cite[Theorem 6.3]{Shen}.
Without restriction, assume $x_{0}=0$. We choose $\varphi \in \mathcal{C}%
_{0}^{\infty }(B_{\frac{7}{4}T})$ such that $\varphi =1$ in $B_{\frac{3}{2}%
T} $, $0\leq \varphi \leq 1$ and $\left\vert \nabla \varphi \right\vert \leq
CT^{-1}$. We also assume that $d\geq 3$ (the case $d=2$ follows from the
case $d=3$ by adding a dummy variable). Define $u(x)=\chi _{T,j}(x+y)-\chi
_{T,j}(x+z)$ ($x\in \mathbb{R}^{d}$) and note that $u$ solves the equation 
\begin{eqnarray*}
-\nabla \cdot (A(\cdot +y)\nabla u)+T^{-2}u &=&\nabla \cdot (A(\cdot
+y)-A(\cdot +z))e_{j} \\
&&+\nabla \cdot \lbrack (A(\cdot +y)-A(\cdot +z))\nabla v]\text{ in }\mathbb{%
R}^{d}
\end{eqnarray*}%
where $v(x)=\chi _{T,j}(x+z)$. We have 
\begin{eqnarray}
-\nabla \cdot (A(\cdot +y)\nabla u) &=&-T^{-2}u\varphi +\nabla \cdot
(\varphi (A(\cdot +y)-A(\cdot +z))e_{j})  \label{11.10} \\
&&+\nabla \cdot (\varphi (A(\cdot +y)-A(\cdot +z))\nabla v)  \notag \\
&&-(A(\cdot +y)-A(\cdot +z))e_{j}\nabla \varphi -A(\cdot +y)\nabla u\cdot
\nabla \varphi  \notag \\
&&-\nabla \cdot (uA(\cdot +y)\nabla \varphi ).  \notag
\end{eqnarray}

Denoting by $G^{y}$ the fundamental solution of the operator $-\nabla \cdot
(A(\cdot +y)\nabla )$ in $\mathbb{R}^{d}$, we use the representation formula
in (\ref{11.10}) to get, for $x\in B_{T}$, 
\begin{eqnarray*}
u(x) &=&-T^{-2}\int_{\mathbb{R}^{d}}G^{y}(x,t)u(t)\varphi (t)dt-\int_{%
\mathbb{R}^{d}}\nabla _{t}G^{y}(x,t)\varphi (t)(A(t+y)-A(t+z))e_{j}dt \\
&&-\int_{\mathbb{R}^{d}}\nabla _{t}G^{y}(x,t)\varphi
(t)(A(t+y)-A(t+z))\nabla v(t)dt \\
&&-\int_{\mathbb{R}^{d}}G^{y}(x,t)(A(t+y)-A(t+z))e_{j}\nabla \varphi (t)dt \\
&&-\int_{\mathbb{R}^{d}}G^{y}(x,t)A(t+y)\nabla u(t)\cdot \nabla \varphi (t)dt
\\
&&+\int_{\mathbb{R}^{d}}\nabla _{t}G^{y}(x,t)A(t+y)u(t)\nabla \varphi (t)dt.
\end{eqnarray*}%
It follows that 
\begin{eqnarray}
\left\vert u(x)\right\vert &\leq &CT^{-2}\int_{B_{2T}}\left\vert
G^{y}(x,t)\right\vert \left\vert u(t)\right\vert dt+  \label{e11.10} \\
&&+C\left\Vert A(\cdot +y)-A(\cdot +z)\right\Vert _{L^{\infty
}(B_{R})}\int_{B_{2T}}\left\vert \nabla _{t}G^{y}(x,t)\right\vert dt  \notag
\\
&&+C\left\Vert A(\cdot +y)-A(\cdot +z)\right\Vert _{L^{\infty
}(B_{R})}\int_{B_{2T}}\left\vert \nabla _{t}G^{y}(x,t)\right\vert \left\vert
\nabla v(t)\right\vert dt  \notag \\
&&+C\left\Vert A(\cdot +y)-A(\cdot +z)\right\Vert _{L^{\infty
}(B_{R})}\int_{B_{2T}}\left\vert G^{y}(x,t)\right\vert \left\vert \nabla
\varphi (t)\right\vert dt  \notag \\
&&+C\left( \int_{B_{2T}}\left\vert G^{y}(x,t)\right\vert ^{2}\left\vert
\nabla \varphi (t)\right\vert ^{2}dt\right) ^{\frac{1}{2}}\left(
\int_{B_{2T}}\left\vert \nabla u\right\vert ^{2}\right) ^{\frac{1}{2}} 
\notag \\
&&+C\left( \int_{B_{2T}}\left\vert \nabla _{t}G^{y}(x,t)\right\vert
^{2}\left\vert \nabla \varphi (t)\right\vert ^{2}dt\right) ^{\frac{1}{2}%
}\left( \int_{B_{2T}}\left\vert u\right\vert ^{2}\right) ^{\frac{1}{2}}. 
\notag
\end{eqnarray}%
Let us first deal with the last two terms in (\ref{e11.10}). Let $0<\tau <1$
be such that $B_{\tau T}(x)\subset B_{T}$ (recall that $x\in B_{T}$). Then $%
B_{2T}\backslash B_{\tau T}(x)\subset B_{3T}(x)\backslash B_{\tau T}(x)$ and
since $\nabla \varphi =0$ in $B_{T}$ (and hence in $B_{\tau T}(x)$), it
holds that 
\begin{eqnarray*}
\left( \int_{B_{2T}}\left\vert G^{y}(x,t)\right\vert ^{2}\left\vert \nabla
\varphi (t)\right\vert ^{2}dt\right) ^{\frac{1}{2}} &\leq &CT^{-1}\left(
\int_{B_{3T}(x)\backslash B_{\tau T}(x)}\frac{dt}{\left\vert x-t\right\vert
^{2(d-2)}}\right) ^{\frac{1}{2}} \\
&\leq &CT^{1-\frac{d}{2}};
\end{eqnarray*}%
\begin{eqnarray*}
\left( \int_{B_{2T}}\left\vert \nabla _{t}G^{y}(x,t)\right\vert
^{2}\left\vert \nabla \varphi (t)\right\vert ^{2}dt\right) ^{\frac{1}{2}}
&\leq &CT^{-1}\left( \int_{B_{3T}(x)\backslash B_{\tau T}(x)}\left\vert
\nabla _{t}G^{y}(x,t)\right\vert ^{2}\right) ^{\frac{1}{2}} \\
&\leq &CT^{-1}\left( \sum_{i=\left[ \frac{\ln \tau }{\ln 2}\right]
}^{2}\int_{B_{2^{i+1}T}(x)\backslash B_{2^{i}T}(x)}\left\vert \nabla
_{t}G^{y}(x,t)\right\vert ^{2}dt\right) ^{\frac{1}{2}} \\
&\leq &CT^{-1}\left( \sum_{i=\left[ \frac{\ln \tau }{\ln 2}\right]
}^{2}(2^{i}T)^{2-d}\right) ^{\frac{1}{2}}\leq CT^{-\frac{d}{2}},
\end{eqnarray*}%
where $\left[ \frac{\ln \tau }{\ln 2}\right] $ stands for the integer part
of $\frac{\ln \tau }{\ln 2}$. We infer that the last two terms in (\ref%
{e11.10}) are bounded from above by $T\left( 
\mathchoice {{\setbox0=\hbox{$\displaystyle{\textstyle -}{\int}$ } \vcenter{\hbox{$\textstyle -$
}}\kern-.6\wd0}}{{\setbox0=\hbox{$\textstyle{\scriptstyle -}{\int}$ } \vcenter{\hbox{$\scriptstyle -$
}}\kern-.6\wd0}}{{\setbox0=\hbox{$\scriptstyle{\scriptscriptstyle -}{\int}$
} \vcenter{\hbox{$\scriptscriptstyle -$
}}\kern-.6\wd0}}{{\setbox0=\hbox{$\scriptscriptstyle{\scriptscriptstyle
-}{\int}$ } \vcenter{\hbox{$\scriptscriptstyle -$ }}\kern-.6\wd0}}%
\!\int_{B_{2T}}\left\vert \nabla u\right\vert ^{2}\right) ^{\frac{1}{2}%
}+\left( 
\mathchoice {{\setbox0=\hbox{$\displaystyle{\textstyle -}{\int}$ } \vcenter{\hbox{$\textstyle -$
}}\kern-.6\wd0}}{{\setbox0=\hbox{$\textstyle{\scriptstyle -}{\int}$ } \vcenter{\hbox{$\scriptstyle -$
}}\kern-.6\wd0}}{{\setbox0=\hbox{$\scriptstyle{\scriptscriptstyle -}{\int}$
} \vcenter{\hbox{$\scriptscriptstyle -$
}}\kern-.6\wd0}}{{\setbox0=\hbox{$\scriptscriptstyle{\scriptscriptstyle
-}{\int}$ } \vcenter{\hbox{$\scriptscriptstyle -$ }}\kern-.6\wd0}}%
\!\int_{B_{2T}}\left\vert u\right\vert ^{2}\right) ^{\frac{1}{2}}$. Next,
for any $R>2T$, we appeal to (\ref{4.3}) in Lemma \ref{l4.1} to get in (\ref%
{11.10}), 
\begin{eqnarray*}
\mathchoice {{\setbox0=\hbox{$\displaystyle{\textstyle -}{\int}$ } \vcenter{\hbox{$\textstyle -$
}}\kern-.6\wd0}}{{\setbox0=\hbox{$\textstyle{\scriptstyle -}{\int}$ } \vcenter{\hbox{$\scriptstyle -$
}}\kern-.6\wd0}}{{\setbox0=\hbox{$\scriptstyle{\scriptscriptstyle -}{\int}$
} \vcenter{\hbox{$\scriptscriptstyle -$
}}\kern-.6\wd0}}{{\setbox0=\hbox{$\scriptscriptstyle{\scriptscriptstyle
-}{\int}$ } \vcenter{\hbox{$\scriptscriptstyle -$ }}\kern-.6\wd0}}%
\!\int_{B_{2T}}\left( \left\vert \nabla u\right\vert ^{2}+T^{-2}\left\vert
u\right\vert ^{2}\right) &\leq &\sup_{x\in \mathbb{R}^{d}}%
\mathchoice {{\setbox0=\hbox{$\displaystyle{\textstyle -}{\int}$ } \vcenter{\hbox{$\textstyle -$
}}\kern-.6\wd0}}{{\setbox0=\hbox{$\textstyle{\scriptstyle -}{\int}$ } \vcenter{\hbox{$\scriptstyle -$
}}\kern-.6\wd0}}{{\setbox0=\hbox{$\scriptstyle{\scriptscriptstyle -}{\int}$
} \vcenter{\hbox{$\scriptscriptstyle -$
}}\kern-.6\wd0}}{{\setbox0=\hbox{$\scriptscriptstyle{\scriptscriptstyle
-}{\int}$ } \vcenter{\hbox{$\scriptscriptstyle -$ }}\kern-.6\wd0}}%
\!\int_{B_{2T}(x)}\left( \left\vert \nabla u\right\vert
^{2}+T^{-2}\left\vert u\right\vert ^{2}\right) \\
&\leq &C\sup_{x\in \mathbb{R}^{d}}%
\mathchoice {{\setbox0=\hbox{$\displaystyle{\textstyle -}{\int}$ } \vcenter{\hbox{$\textstyle -$
}}\kern-.6\wd0}}{{\setbox0=\hbox{$\textstyle{\scriptstyle -}{\int}$ } \vcenter{\hbox{$\scriptstyle -$
}}\kern-.6\wd0}}{{\setbox0=\hbox{$\scriptstyle{\scriptscriptstyle -}{\int}$
} \vcenter{\hbox{$\scriptscriptstyle -$
}}\kern-.6\wd0}}{{\setbox0=\hbox{$\scriptscriptstyle{\scriptscriptstyle
-}{\int}$ } \vcenter{\hbox{$\scriptscriptstyle -$ }}\kern-.6\wd0}}%
\!\int_{B_{2T}(x)}\left\vert A(t+y)-A(t+z)\right\vert ^{2}dt \\
&&+C\sup_{x\in \mathbb{R}^{d}}%
\mathchoice {{\setbox0=\hbox{$\displaystyle{\textstyle -}{\int}$ } \vcenter{\hbox{$\textstyle -$
}}\kern-.6\wd0}}{{\setbox0=\hbox{$\textstyle{\scriptstyle -}{\int}$ } \vcenter{\hbox{$\scriptstyle -$
}}\kern-.6\wd0}}{{\setbox0=\hbox{$\scriptstyle{\scriptscriptstyle -}{\int}$
} \vcenter{\hbox{$\scriptscriptstyle -$
}}\kern-.6\wd0}}{{\setbox0=\hbox{$\scriptscriptstyle{\scriptscriptstyle
-}{\int}$ } \vcenter{\hbox{$\scriptscriptstyle -$ }}\kern-.6\wd0}}%
\!\int_{B_{2T}(x)}\left\vert A(t+y)-A(t+z)\right\vert ^{2}\left\vert \nabla
v\right\vert ^{2}dt \\
&\leq &C\left\Vert A(\cdot +y)-A(\cdot +z)\right\Vert _{L^{\infty
}(B_{R})}^{2},
\end{eqnarray*}%
where we have used the facts that $R>2T$ and 
\begin{equation*}
\sup_{x\in \mathbb{R}^{d}}%
\mathchoice {{\setbox0=\hbox{$\displaystyle{\textstyle -}{\int}$ } \vcenter{\hbox{$\textstyle -$
}}\kern-.6\wd0}}{{\setbox0=\hbox{$\textstyle{\scriptstyle -}{\int}$ } \vcenter{\hbox{$\scriptstyle -$
}}\kern-.6\wd0}}{{\setbox0=\hbox{$\scriptstyle{\scriptscriptstyle -}{\int}$
} \vcenter{\hbox{$\scriptscriptstyle -$
}}\kern-.6\wd0}}{{\setbox0=\hbox{$\scriptscriptstyle{\scriptscriptstyle
-}{\int}$ } \vcenter{\hbox{$\scriptscriptstyle -$ }}\kern-.6\wd0}}%
\!\int_{B_{2T}(x)}\left\vert \nabla v\right\vert ^{2}dt\leq C\text{ (see (%
\ref{4.3}) in Lemma \ref{l4.1}).}
\end{equation*}%
It follows at once that 
\begin{equation}
T\left( 
\mathchoice {{\setbox0=\hbox{$\displaystyle{\textstyle -}{\int}$ } \vcenter{\hbox{$\textstyle -$
}}\kern-.6\wd0}}{{\setbox0=\hbox{$\textstyle{\scriptstyle -}{\int}$ } \vcenter{\hbox{$\scriptstyle -$
}}\kern-.6\wd0}}{{\setbox0=\hbox{$\scriptstyle{\scriptscriptstyle -}{\int}$
} \vcenter{\hbox{$\scriptscriptstyle -$
}}\kern-.6\wd0}}{{\setbox0=\hbox{$\scriptscriptstyle{\scriptscriptstyle
-}{\int}$ } \vcenter{\hbox{$\scriptscriptstyle -$ }}\kern-.6\wd0}}%
\!\int_{B_{2T}}\left\vert \nabla u\right\vert ^{2}\right) ^{\frac{1}{2}%
}+\left( 
\mathchoice {{\setbox0=\hbox{$\displaystyle{\textstyle -}{\int}$ } \vcenter{\hbox{$\textstyle -$
}}\kern-.6\wd0}}{{\setbox0=\hbox{$\textstyle{\scriptstyle -}{\int}$ } \vcenter{\hbox{$\scriptstyle -$
}}\kern-.6\wd0}}{{\setbox0=\hbox{$\scriptstyle{\scriptscriptstyle -}{\int}$
} \vcenter{\hbox{$\scriptscriptstyle -$
}}\kern-.6\wd0}}{{\setbox0=\hbox{$\scriptscriptstyle{\scriptscriptstyle
-}{\int}$ } \vcenter{\hbox{$\scriptscriptstyle -$ }}\kern-.6\wd0}}%
\!\int_{B_{2T}}\left\vert u\right\vert ^{2}\right) ^{\frac{1}{2}}\leq
CT\left\Vert A(\cdot +y)-A(\cdot +z)\right\Vert _{L^{\infty }(B_{R})}.
\label{e1.1}
\end{equation}

Concerning the second term in the right-hand side of (\ref{e11.10}), we have 
\begin{eqnarray}
\int_{B_{2T}}\left\vert \nabla _{t}G^{y}(x,t)\right\vert dt &\leq
&C\int_{B_{3T}(x)}\left\vert \nabla _{t}G^{y}(x,t)\right\vert dt
\label{e2.1} \\
&\leq &C\sum_{i=-\infty }^{1}\int_{B_{2^{i+1}T}(x)\backslash
B_{2^{i}T}(x)}\left\vert \nabla _{t}G^{y}(x,t)\right\vert dt\leq
C\sum_{i=-\infty }^{1}2^{i}T\leq CT,  \notag
\end{eqnarray}%
where we have used for the first inequality in (\ref{e2.1}), the fact that $%
B_{2T}\subset B_{3T}(x)$ (recall that $x\in B_{T}$), and for the last
inequality, (\ref{2.8}) (for $q=1$). It follows that 
\begin{equation*}
C\left\Vert A(\cdot +y)-A(\cdot +z)\right\Vert _{L^{\infty
}(B_{R})}\int_{B_{2T}}\left\vert \nabla _{t}G^{y}(x,t)\right\vert dt\leq
CT\left\Vert A(\cdot +y)-A(\cdot +z)\right\Vert _{L^{\infty }(B_{R})}.
\end{equation*}%
As for the third term in the right-hand side of (\ref{e11.10}) is concerned,
we concentrate on the control of the integral 
\begin{equation*}
I=\int_{B_{2T}}\left\vert \nabla _{t}G^{y}(x,t)\right\vert \left\vert \nabla
v(t)\right\vert dt.
\end{equation*}%
First, we note that the function $v$ solves the equation 
\begin{equation*}
-\nabla \cdot (A(\cdot +z)\nabla v)+T^{-2}v=\nabla \cdot (A(\cdot +z)e_{j})%
\text{ in }\mathbb{R}^{d}
\end{equation*}%
so that appealing to (\ref{4.3}), 
\begin{equation}
\left( 
\mathchoice {{\setbox0=\hbox{$\displaystyle{\textstyle -}{\int}$ } \vcenter{\hbox{$\textstyle -$
}}\kern-.6\wd0}}{{\setbox0=\hbox{$\textstyle{\scriptstyle -}{\int}$ } \vcenter{\hbox{$\scriptstyle -$
}}\kern-.6\wd0}}{{\setbox0=\hbox{$\scriptstyle{\scriptscriptstyle -}{\int}$
} \vcenter{\hbox{$\scriptscriptstyle -$
}}\kern-.6\wd0}}{{\setbox0=\hbox{$\scriptscriptstyle{\scriptscriptstyle
-}{\int}$ } \vcenter{\hbox{$\scriptscriptstyle -$ }}\kern-.6\wd0}}%
\!\int_{B_{2T}}\left\vert \nabla v\right\vert ^{2}\right) ^{\frac{1}{2}}\leq
C\text{.}  \label{e11.15}
\end{equation}

Next, H\"{o}lder inequality and (\ref{e11.15}) lead to 
\begin{eqnarray*}
I &\leq &CT^{\frac{d}{2}}\left( \int_{B_{2T}}\left\vert \nabla
_{t}G^{y}(x,t)\right\vert ^{2}dt\right) ^{\frac{1}{2}}\leq CT^{\frac{d}{2}%
}\left( \int_{B_{3T}(x)\backslash B_{\tau T}(x)}\left\vert \nabla
_{t}G^{y}(x,t)\right\vert ^{2}\right) ^{\frac{1}{2}} \\
&\leq &CT^{\frac{d}{2}}\left( \sum_{i=\left[ \frac{\ln \tau }{\ln 2}\right]
}^{2}\int_{B_{2^{i+1}T}(x)\backslash B_{2^{i}T}(x)}\left\vert \nabla
_{t}G^{y}(x,t)\right\vert ^{2}dt\right) ^{\frac{1}{2}}\leq CT^{\frac{d}{2}%
}\left( \sum_{i=\left[ \frac{\ln \tau }{\ln 2}\right] }^{2}(2^{i}T)^{2-d}%
\right) ^{\frac{1}{2}} \\
&\leq &CT^{\frac{d}{2}}T^{1-\frac{d}{2}}=CT.
\end{eqnarray*}%
For the fourth term in the right-hand side of (\ref{e11.10}), we have 
\begin{equation*}
\int_{B_{2T}}\left\vert G^{y}(x,t)\right\vert \left\vert \nabla \varphi
(t)\right\vert dt\leq CT^{-1}\int_{B_{3T}(x)}\frac{dt}{\left\vert
x-t\right\vert ^{d-2}}dt\leq CT.
\end{equation*}%
We have therefore shown that 
\begin{equation}
\left\vert u(x)\right\vert \leq CT^{-2}\int_{B_{2T}}\frac{\left\vert
u(t)\right\vert }{\left\vert x-t\right\vert ^{d-2}}dt+CT\left\Vert A(\cdot
+y)-A(\cdot +z)\right\Vert _{L^{\infty }(B_{R})}.  \label{e11.17}
\end{equation}%
Using the well known fractional integral estimates, (\ref{e11.17}) yields 
\begin{equation*}
\left( 
\mathchoice {{\setbox0=\hbox{$\displaystyle{\textstyle -}{\int}$ } \vcenter{\hbox{$\textstyle -$
}}\kern-.6\wd0}}{{\setbox0=\hbox{$\textstyle{\scriptstyle -}{\int}$ } \vcenter{\hbox{$\scriptstyle -$
}}\kern-.6\wd0}}{{\setbox0=\hbox{$\scriptstyle{\scriptscriptstyle -}{\int}$
} \vcenter{\hbox{$\scriptscriptstyle -$
}}\kern-.6\wd0}}{{\setbox0=\hbox{$\scriptscriptstyle{\scriptscriptstyle
-}{\int}$ } \vcenter{\hbox{$\scriptscriptstyle -$ }}\kern-.6\wd0}}%
\!\int_{B_{T}}\left\vert u\right\vert ^{q}\right) ^{\frac{1}{q}}\leq C\left( 
\mathchoice {{\setbox0=\hbox{$\displaystyle{\textstyle -}{\int}$ } \vcenter{\hbox{$\textstyle -$
}}\kern-.6\wd0}}{{\setbox0=\hbox{$\textstyle{\scriptstyle -}{\int}$ } \vcenter{\hbox{$\scriptstyle -$
}}\kern-.6\wd0}}{{\setbox0=\hbox{$\scriptstyle{\scriptscriptstyle -}{\int}$
} \vcenter{\hbox{$\scriptscriptstyle -$
}}\kern-.6\wd0}}{{\setbox0=\hbox{$\scriptscriptstyle{\scriptscriptstyle
-}{\int}$ } \vcenter{\hbox{$\scriptscriptstyle -$ }}\kern-.6\wd0}}%
\!\int_{B_{2T}}\left\vert u\right\vert ^{p}\right) ^{\frac{1}{p}%
}+CT\left\Vert A(\cdot +y)-A(\cdot +z)\right\Vert _{L^{\infty }(B_{R})}
\end{equation*}%
where $1<p<q\leq \infty $ with $\frac{1}{p}-\frac{1}{q}<\frac{2}{d}$.
However from (\ref{e1.1}) we derive the estimate 
\begin{equation*}
\left( 
\mathchoice {{\setbox0=\hbox{$\displaystyle{\textstyle -}{\int}$ } \vcenter{\hbox{$\textstyle -$
}}\kern-.6\wd0}}{{\setbox0=\hbox{$\textstyle{\scriptstyle -}{\int}$ } \vcenter{\hbox{$\scriptstyle -$
}}\kern-.6\wd0}}{{\setbox0=\hbox{$\scriptstyle{\scriptscriptstyle -}{\int}$
} \vcenter{\hbox{$\scriptscriptstyle -$
}}\kern-.6\wd0}}{{\setbox0=\hbox{$\scriptscriptstyle{\scriptscriptstyle
-}{\int}$ } \vcenter{\hbox{$\scriptscriptstyle -$ }}\kern-.6\wd0}}%
\!\int_{B_{2T}}\left\vert u\right\vert ^{2}\right) ^{\frac{1}{2}}\leq
CT\left\Vert A(\cdot +y)-A(\cdot +z)\right\Vert _{L^{\infty }(B_{R})},
\end{equation*}%
so that by an iteration argument, we are led to 
\begin{equation*}
\left\Vert u\right\Vert _{L^{\infty }(B_{T})}\leq CT\left\Vert A(\cdot
+y)-A(\cdot +z)\right\Vert _{L^{\infty }(B_{R})}.
\end{equation*}%
This yields (recalling that $x_{0}=0$) 
\begin{equation*}
\left\vert u(0)\right\vert \leq CT\left\Vert A(\cdot +y)-A(\cdot
+z)\right\Vert _{L^{\infty }(B_{R})}.
\end{equation*}%
Recalling that $0$ may be replaced by any $t\in B_{R}$, this completes the
proof.
\end{proof}

\begin{theorem}
\label{t11.2}Let $T\geq 1$ and $R>2T$. For any $0<L\leq T$ and $\sigma \in
(0,1)$, there is $C_{\sigma }=C_{\sigma }(\sigma ,A)$ such that 
\begin{equation}
T^{-1}\left\Vert \chi _{T}\right\Vert _{L^{\infty }(\mathbb{R}^{d})}\leq
C_{\sigma }\left( \rho (L,R)+\left( \frac{L}{T}\right) ^{\sigma }\right) .
\label{11.11}
\end{equation}
\end{theorem}

\begin{proof}
Let $y,z\in \mathbb{R}^{d}$ with $\left\vert z\right\vert \leq L\leq T$.
Then 
\begin{equation*}
\left\vert \chi _{T}(y)\right\vert \leq \left\vert \chi _{T}(y)-\chi
_{T}(0)\right\vert +\left\vert \chi _{T}(0)\right\vert
\end{equation*}%
and 
\begin{eqnarray*}
\left\vert \chi _{T}(y)-\chi _{T}(0)\right\vert &\leq &\left\vert \chi
_{T}(y)-\chi _{T}(z)\right\vert +\left\vert \chi _{T}(z)-\chi
_{T}(0)\right\vert \\
&=&\left\vert \chi _{T}(0+y)-\chi _{T}(0+z)\right\vert +\left\vert \chi
_{T}(z)-\chi _{T}(0)\right\vert \\
&\leq &\sup_{x\in B_{R}}\left\vert \chi _{T}(x+y)-\chi _{T}(x+z)\right\vert
+\left\vert \chi _{T}(z)-\chi _{T}(0)\right\vert \\
&\leq &CT\left\Vert A(\cdot +y)-A(\cdot +z)\right\Vert _{L^{\infty
}(B_{R})}+C_{\sigma }T^{1-\sigma }L^{\sigma },
\end{eqnarray*}%
where for the last inequality above we have used (\ref{e5.8}) (in Lemma \ref%
{l11.1}) and (\ref{11.9}) (in Theorem \ref{t11.1}). It follows readily that 
\begin{equation}
\sup_{y\in \mathbb{R}^{d}}\left\vert \chi _{T}(y)-\chi _{T}(0)\right\vert
\leq T\left( C\rho (L,R)+C_{\sigma }\left( \frac{L}{T}\right) ^{\sigma
}\right) .  \label{11.12}
\end{equation}%
On the other hand, observing that 
\begin{eqnarray*}
\left\vert \chi _{T}(0)\right\vert &\leq &\left\vert 
\mathchoice {{\setbox0=\hbox{$\displaystyle{\textstyle -}{\int}$ } \vcenter{\hbox{$\textstyle -$
}}\kern-.6\wd0}}{{\setbox0=\hbox{$\textstyle{\scriptstyle -}{\int}$ } \vcenter{\hbox{$\scriptstyle -$
}}\kern-.6\wd0}}{{\setbox0=\hbox{$\scriptstyle{\scriptscriptstyle -}{\int}$
} \vcenter{\hbox{$\scriptscriptstyle -$
}}\kern-.6\wd0}}{{\setbox0=\hbox{$\scriptscriptstyle{\scriptscriptstyle
-}{\int}$ } \vcenter{\hbox{$\scriptscriptstyle -$ }}\kern-.6\wd0}}%
\!\int_{B_{r}}(\chi _{T}(t)-\chi _{T}(0))dt\right\vert +\left\vert 
\mathchoice {{\setbox0=\hbox{$\displaystyle{\textstyle -}{\int}$ } \vcenter{\hbox{$\textstyle -$
}}\kern-.6\wd0}}{{\setbox0=\hbox{$\textstyle{\scriptstyle -}{\int}$ } \vcenter{\hbox{$\scriptstyle -$
}}\kern-.6\wd0}}{{\setbox0=\hbox{$\scriptstyle{\scriptscriptstyle -}{\int}$
} \vcenter{\hbox{$\scriptscriptstyle -$
}}\kern-.6\wd0}}{{\setbox0=\hbox{$\scriptscriptstyle{\scriptscriptstyle
-}{\int}$ } \vcenter{\hbox{$\scriptscriptstyle -$ }}\kern-.6\wd0}}%
\!\int_{B_{r}}\chi _{T}(t)dt\right\vert \\
&\leq &\sup_{y\in \mathbb{R}^{d}}\left\vert \chi _{T}(y)-\chi
_{T}(0)\right\vert +\left\vert 
\mathchoice {{\setbox0=\hbox{$\displaystyle{\textstyle -}{\int}$ } \vcenter{\hbox{$\textstyle -$
}}\kern-.6\wd0}}{{\setbox0=\hbox{$\textstyle{\scriptstyle -}{\int}$ } \vcenter{\hbox{$\scriptstyle -$
}}\kern-.6\wd0}}{{\setbox0=\hbox{$\scriptstyle{\scriptscriptstyle -}{\int}$
} \vcenter{\hbox{$\scriptscriptstyle -$
}}\kern-.6\wd0}}{{\setbox0=\hbox{$\scriptscriptstyle{\scriptscriptstyle
-}{\int}$ } \vcenter{\hbox{$\scriptscriptstyle -$ }}\kern-.6\wd0}}%
\!\int_{B_{r}}\chi _{T}(t)dt\right\vert
\end{eqnarray*}%
and letting $r\rightarrow \infty $, we use the fact that $\left\langle \chi
_{T}\right\rangle =0$ to get 
\begin{equation*}
\left\vert \chi _{T}(0)\right\vert \leq \sup_{y\in \mathbb{R}^{d}}\left\vert
\chi _{T}(y)-\chi _{T}(0)\right\vert .
\end{equation*}%
The above inequality associated to (\ref{11.12}) yield (\ref{11.11}).
\end{proof}

Now, we set (for $T\geq 1$ and $\sigma \in (0,1]$) 
\begin{equation}
\Theta _{\sigma }(T)=\inf_{0<L<T}\left( \rho (L,3T)+\left( \frac{L}{T}%
\right) ^{\sigma }\right)  \label{11.13}
\end{equation}%
where $\rho (L,R)$ is given by (\ref{11.4}). Then $T\mapsto \Theta _{\sigma
}(T)$ is a continuous decreasing function satisfying $\Theta _{\sigma
}(T)\rightarrow 0$ when $T\rightarrow \infty $ (this stems from the
asymptotic almost periodicity of $A$, so that $\rho (L,3T)\rightarrow 0$ as $%
T\rightarrow \infty $). We infer from (\ref{11.11}) that 
\begin{equation}
T^{-1}\left\Vert \chi _{T}\right\Vert _{L^{\infty }(\mathbb{R}^{d})}\leq
C_{\sigma }\Theta _{\sigma }(T)  \label{11.14}
\end{equation}%
and hence 
\begin{equation*}
T^{-1}\left\Vert \chi _{T}\right\Vert _{L^{\infty }(\mathbb{R}%
^{d})}\rightarrow 0\text{ as }T\rightarrow \infty .
\end{equation*}

As in \cite{Shen} we state the following result.

\begin{lemma}
\label{l11.2}Let $g\in L_{\infty ,AP}^{2}(\mathbb{R}^{d})\cap L^{\infty }(%
\mathbb{R}^{d})$ with $\left\langle g\right\rangle =0$ and 
\begin{equation}
\sup_{x\in \mathbb{R}^{d}}\left( 
\mathchoice {{\setbox0=\hbox{$\displaystyle{\textstyle -}{\int}$ } \vcenter{\hbox{$\textstyle -$
}}\kern-.6\wd0}}{{\setbox0=\hbox{$\textstyle{\scriptstyle -}{\int}$ } \vcenter{\hbox{$\scriptstyle -$
}}\kern-.6\wd0}}{{\setbox0=\hbox{$\scriptstyle{\scriptscriptstyle -}{\int}$
} \vcenter{\hbox{$\scriptscriptstyle -$
}}\kern-.6\wd0}}{{\setbox0=\hbox{$\scriptscriptstyle{\scriptscriptstyle
-}{\int}$ } \vcenter{\hbox{$\scriptscriptstyle -$ }}\kern-.6\wd0}}%
\!\int_{B_{r}(x)}\left\vert g\right\vert ^{2}\right) ^{\frac{1}{2}}\leq
C_{0}\left( \frac{T}{r}\right) ^{1-\sigma }\text{ for }0<r\leq T
\label{11.15}
\end{equation}%
where $\sigma \in (0,1]$. Then there is a unique $u\in H_{\infty ,AP}^{1}(%
\mathbb{R}^{d})$ such that 
\begin{equation}
-\Delta u+T^{-2}u=g\text{ in }\mathbb{R}^{d},\ \ \left\langle u\right\rangle
=0  \label{11.16}
\end{equation}%
and 
\begin{equation}
T^{-2}\left\Vert u\right\Vert _{L^{\infty }(\mathbb{R}^{d})}+T^{-1}\left%
\Vert \nabla u\right\Vert _{L^{\infty }(\mathbb{R}^{d})}\leq C,
\label{11.17}
\end{equation}%
\begin{equation}
\left\vert \nabla u(x)-\nabla u(y)\right\vert \leq C_{\sigma }T^{1-\sigma
}\left\vert x-y\right\vert ^{\sigma }\ \ \forall x,y\in \mathbb{R}^{d}
\label{11.18}
\end{equation}%
where $C=C(d)$ and $C_{\sigma }=C_{\sigma }(d,\sigma )$. Moreover $u$ and $%
\nabla u$ belong to $\mathcal{B}_{\infty ,AP}(\mathbb{R}^{d})$ with 
\begin{equation}
T^{-2}\left\Vert u\right\Vert _{L^{\infty }(\mathbb{R}^{d})}\leq C\Theta
_{1}(T)  \label{11.19}
\end{equation}%
and 
\begin{equation}
T^{-1}\left\Vert \nabla u\right\Vert _{L^{\infty }(\mathbb{R}^{d})}\leq
C\Theta _{\sigma }(T)  \label{11.20}
\end{equation}%
where $\Theta _{\sigma }(T)$ is defined by \emph{(\ref{11.13})} and $%
C=C(d,\sigma ,g)$.
\end{lemma}

\begin{proof}
If we proceed as in the proof of Lemma \ref{l4.1}, we derive the existence
of a unique $u\in H_{\infty ,AP}^{1}(\mathbb{R}^{d})$ solving (\ref{11.16});
we may also refer to \cite{Po-Yu89} for another proof. Next using the
fundamental solution of $-\Delta +T^{-2}$, we easily get (\ref{11.17}). We
infer from (\ref{11.17}) that $u,\nabla u\in \mathcal{B}_{\infty ,AP}(%
\mathbb{R}^{d})$. In order to obtain (\ref{11.18}) we use (\ref{11.15}) and
proceed as in \cite[Lemma 7.1]{Shen}. It remains to check (\ref{11.19}) and (%
\ref{11.20}). To that end, we apply (\ref{11.17}) to the function 
\begin{equation*}
\frac{u(\cdot +y)-u(\cdot +z)}{\left\Vert A(\cdot +y)-A(\cdot +z)\right\Vert
_{L^{\infty }(B_{R})}}
\end{equation*}%
with $u$ solution of (\ref{11.16}). Then 
\begin{equation}
T^{-2}\left\Vert u(\cdot +y)-u(\cdot +z)\right\Vert _{L^{\infty
}(B_{R})}\leq C\left\Vert A(\cdot +y)-A(\cdot +z)\right\Vert _{L^{\infty
}(B_{R})}  \label{11.21}
\end{equation}%
and 
\begin{equation}
T^{-1}\left\Vert \nabla u(\cdot +y)-\nabla u(\cdot +z)\right\Vert
_{L^{\infty }(B_{R})}\leq C\left\Vert A(\cdot +y)-A(\cdot +z)\right\Vert
_{L^{\infty }(B_{R})}.  \label{11.22}
\end{equation}%
Using the boundedness of the gradient (see (\ref{11.17})), we obtain 
\begin{equation}
\left\vert u(x)-u(t)\right\vert \leq CT\left\vert x-t\right\vert \ \ \forall
x,t\in \mathbb{R}^{d}.  \label{11.23}
\end{equation}%
Next assuming that $\left\vert z\right\vert \leq L\leq T$, we have 
\begin{eqnarray*}
T^{-2}\left\vert u(y)-u(0)\right\vert &\leq &T^{-2}\left\vert
u(y)-u(z)\right\vert +T^{-2}\left\vert u(z)-u(0)\right\vert \\
&\leq &C\left\Vert A(\cdot +y)-A(\cdot +z)\right\Vert _{L^{\infty
}(B_{R})}+CT^{-1}L
\end{eqnarray*}%
where we used (\ref{11.21}) and (\ref{11.23}). Hence 
\begin{equation}
\sup_{y\in \mathbb{R}^{d}}T^{-2}\left\vert u(y)-u(0)\right\vert \leq C(\rho
(L,R)+T^{-1}L)  \label{11.24}
\end{equation}%
for any $R>2T$ and $L>0$. Also, using the inequality 
\begin{eqnarray*}
T^{-2}\left\vert u(0)\right\vert &\leq &T^{-2}\left\vert 
\mathchoice {{\setbox0=\hbox{$\displaystyle{\textstyle -}{\int}$ } \vcenter{\hbox{$\textstyle -$
}}\kern-.6\wd0}}{{\setbox0=\hbox{$\textstyle{\scriptstyle -}{\int}$ } \vcenter{\hbox{$\scriptstyle -$
}}\kern-.6\wd0}}{{\setbox0=\hbox{$\scriptstyle{\scriptscriptstyle -}{\int}$
} \vcenter{\hbox{$\scriptscriptstyle -$
}}\kern-.6\wd0}}{{\setbox0=\hbox{$\scriptscriptstyle{\scriptscriptstyle
-}{\int}$ } \vcenter{\hbox{$\scriptscriptstyle -$ }}\kern-.6\wd0}}%
\!\int_{B_{r}}(u(t)-u(0))dt\right\vert +T^{-2}\left\vert 
\mathchoice {{\setbox0=\hbox{$\displaystyle{\textstyle -}{\int}$ } \vcenter{\hbox{$\textstyle -$
}}\kern-.6\wd0}}{{\setbox0=\hbox{$\textstyle{\scriptstyle -}{\int}$ } \vcenter{\hbox{$\scriptstyle -$
}}\kern-.6\wd0}}{{\setbox0=\hbox{$\scriptstyle{\scriptscriptstyle -}{\int}$
} \vcenter{\hbox{$\scriptscriptstyle -$
}}\kern-.6\wd0}}{{\setbox0=\hbox{$\scriptscriptstyle{\scriptscriptstyle
-}{\int}$ } \vcenter{\hbox{$\scriptscriptstyle -$ }}\kern-.6\wd0}}%
\!\int_{B_{r}}u(t)dt\right\vert \\
&\leq &T^{-2}\sup_{y\in \mathbb{R}^{d}}\left\vert u(y)-u(0)\right\vert
+T^{-2}\left\vert 
\mathchoice {{\setbox0=\hbox{$\displaystyle{\textstyle -}{\int}$ } \vcenter{\hbox{$\textstyle -$
}}\kern-.6\wd0}}{{\setbox0=\hbox{$\textstyle{\scriptstyle -}{\int}$ } \vcenter{\hbox{$\scriptstyle -$
}}\kern-.6\wd0}}{{\setbox0=\hbox{$\scriptstyle{\scriptscriptstyle -}{\int}$
} \vcenter{\hbox{$\scriptscriptstyle -$
}}\kern-.6\wd0}}{{\setbox0=\hbox{$\scriptscriptstyle{\scriptscriptstyle
-}{\int}$ } \vcenter{\hbox{$\scriptscriptstyle -$ }}\kern-.6\wd0}}%
\!\int_{B_{r}}u(t)dt\right\vert
\end{eqnarray*}%
together with the fact that $\left\langle u\right\rangle =0$, we get (after
letting $r\rightarrow \infty $) 
\begin{equation}
T^{-2}\left\vert u(0)\right\vert \leq C(\rho (L,R)+T^{-1}L)\ \ \forall \
0<L\leq T  \label{11.25}
\end{equation}%
where we have also used (\ref{11.24}). Putting together (\ref{11.24}) and (%
\ref{11.25}), and choosing in the resulting inequality $R=3T$, and finally
taking the $\inf_{0<L<T}$, we are led to (\ref{11.19}).

Proceeding as above using this time (\ref{11.18}) and (\ref{11.22}) we
arrive at (\ref{11.20}).
\end{proof}

\begin{lemma}
\label{l5.2}Let $\chi _{T,j}$ be defined by \emph{(\ref{11.5})}, and let $%
\Omega $ be an open bounded set of class $\mathcal{C}^{1,1}$ in $\mathbb{R}%
^{d}$. Then 
\begin{equation}
\int_{\Omega }\left\vert \left( \nabla _{y}\chi _{T,j}\right) \left( \frac{x%
}{\varepsilon }\right) w(x)\right\vert ^{2}dx\leq C\int_{\Omega }(\left\vert
w\right\vert ^{2}+\delta ^{2}\left\vert \nabla w\right\vert ^{2})dx,\text{\
all }w\in H^{1}(\Omega )  \label{5.30}
\end{equation}%
where $\delta =T^{-1}\left\Vert \chi _{T}\right\Vert _{L^{\infty }(\mathbb{R}%
^{d})}$ with $T=\varepsilon ^{-1}$, and $C=C(A,\Omega ,d)>0$.
\end{lemma}

\begin{proof}
By a density argument, it is sufficient to prove (\ref{5.30}) for $w\in 
\mathcal{C}_{0}^{\infty }(\Omega )$. We recall that $\chi _{T,j}$ solves the
equation 
\begin{equation}
-\nabla \cdot (A(e_{j}+\nabla \chi _{T,j}))+T^{-2}\chi _{T,j}=0\text{ in }%
\mathbb{R}^{d}\text{.}  \label{5.32}
\end{equation}%
Testing (\ref{5.32}) with $\psi (y)=\varphi (\varepsilon y)$ where $\varphi
\in H_{loc}^{1}(\mathbb{R}^{d})$ with compact support, and next making the
change of variable $x=\varepsilon y$, we get 
\begin{equation*}
\int_{\mathbb{R}^{d}}\left[ (A^{\varepsilon }(e_{j}+(\nabla _{y}\chi
_{T,j})^{\varepsilon })\cdot \nabla \varphi +T^{-2}\chi _{T,j}^{\varepsilon
}\varphi \right] dx=0
\end{equation*}%
where $u^{\varepsilon }(x)=u(x/\varepsilon )$ for $u\in H_{loc}^{1}(\mathbb{R%
}^{d})$. Choosing $\varphi (x)=\chi _{T,j}(x/\varepsilon )\left\vert
w(x)\right\vert ^{2}$ with $w\in \mathcal{C}_{0}^{\infty }(\Omega )$, we
obtain 
\begin{equation*}
\int_{\Omega }\left[ (A^{\varepsilon }(e_{j}+(\nabla _{y}\chi
_{T,j})^{\varepsilon })\cdot \left( \frac{1}{\varepsilon }(\nabla _{y}\chi
_{T,j})^{\varepsilon }\left\vert w\right\vert ^{2}+2w\chi
_{T,j}^{\varepsilon }\nabla w\right) +T^{-2}\left\vert \chi
_{T,j}^{\varepsilon }\right\vert ^{2}\left\vert w\right\vert ^{2}\right]
dx=0,
\end{equation*}%
or 
\begin{eqnarray}
\int_{\Omega }A^{\varepsilon }(\nabla _{y}\chi _{T,j})^{\varepsilon }w\cdot
(\nabla _{y}\chi _{T,j})^{\varepsilon }wdx &=&-2\varepsilon \int_{\Omega
}A^{\varepsilon }(\nabla _{y}\chi _{T,j})^{\varepsilon }w\cdot \chi
_{T,j}^{\varepsilon }\nabla wdx  \label{5.31} \\
&&-\int_{\Omega }w(A^{\varepsilon }e_{j})\cdot (\nabla _{y}\chi
_{T,j})^{\varepsilon }wdx  \notag \\
&&-2\varepsilon \int_{\Omega }w(A^{\varepsilon }e_{j})\cdot \chi
_{T,j}^{\varepsilon }\nabla wdx  \notag \\
&&-\varepsilon T^{-2}\int_{\Omega }\left\vert \chi _{T,j}^{\varepsilon
}\right\vert ^{2}\left\vert w\right\vert ^{2}dx  \notag \\
&=&I_{1}+I_{2}+I_{3}+I_{4}.  \notag
\end{eqnarray}%
The left hand-side of (\ref{5.31}) is estimated from below by $\alpha
\int_{\Omega }\left\vert (\nabla _{y}\chi _{T,j})^{\varepsilon }w\right\vert
^{2}dx$ while, for the respective terms of the right hand-side of (\ref{5.31}%
) we have, after the use of H\"{o}lder and Young inequalities, 
\begin{equation*}
\left\vert I_{1}\right\vert \leq \frac{\alpha }{3}\int_{\Omega }\left\vert
(\nabla _{y}\chi _{T,j})^{\varepsilon }w\right\vert ^{2}dx+C\varepsilon
^{2}\int_{\Omega }\left\vert \chi _{T,j}^{\varepsilon }\right\vert
^{2}\left\vert \nabla w\right\vert ^{2}dx;
\end{equation*}%
\begin{equation*}
\left\vert I_{2}\right\vert \leq \frac{\alpha }{3}\int_{\Omega }\left\vert
(\nabla _{y}\chi _{T,j})^{\varepsilon }w\right\vert ^{2}dx+C\int_{\Omega
}\left\vert w\right\vert ^{2}dx;
\end{equation*}%
\begin{equation*}
\left\vert I_{3}\right\vert \leq C\int_{\Omega }\left\vert w\right\vert
^{2}dx+C\varepsilon ^{2}\int_{\Omega }\left\vert \chi _{T,j}^{\varepsilon
}\right\vert ^{2}\left\vert \nabla w\right\vert ^{2}dx\text{ and }\left\vert
I_{4}\right\vert \leq C\int_{\Omega }\left\vert w\right\vert ^{2}dx.
\end{equation*}%
It follows that 
\begin{eqnarray*}
\int_{\Omega }\left\vert (\nabla _{y}\chi _{T,j})^{\varepsilon }w\right\vert
^{2}dx &\leq &C\varepsilon ^{2}\int_{\Omega }\left\vert \chi
_{T,j}^{\varepsilon }\right\vert ^{2}\left\vert \nabla w\right\vert
^{2}dx+C\int_{\Omega }\left\vert w\right\vert ^{2}dx \\
&\leq &C\varepsilon ^{2}\left\Vert \chi _{T,j}\right\Vert _{L^{\infty }(%
\mathbb{R}^{d})}^{2}\int_{\Omega }\left\vert \nabla w\right\vert
^{2}dx+C\int_{\Omega }\left\vert w\right\vert ^{2}dx.
\end{eqnarray*}%
Since $T=\varepsilon ^{-1}$, we get (\ref{5.30}), taking into account that $%
T^{-1}\left\Vert \chi _{T},j\right\Vert _{L^{\infty }(\mathbb{R}^{d})}\leq
T^{-1}\left\Vert \chi _{T}\right\Vert _{L^{\infty }(\mathbb{R}^{d})}$.
\end{proof}

\begin{remark}
\label{r5.2}I\emph{n the case of asymptotic periodic functions, we replace }$%
\chi _{T,j}$\emph{\ by }$\chi _{j}\in H_{\infty ,per}^{1}(Y)$\emph{\
solution of the corrector problem (\ref{1.6}) and we have (in view of Lemma %
\ref{l1.2}) }$\left\Vert \chi _{j}\right\Vert _{L^{\infty }(\mathbb{R}%
^{d})}\leq C$\emph{. It follows that }%
\begin{equation*}
\int_{\Omega }\left\vert \left( \nabla _{y}\chi _{j}\right) \left( \frac{x}{%
\varepsilon }\right) w(x)\right\vert ^{2}dx\leq C\int_{\Omega }(\left\vert
w\right\vert ^{2}+\varepsilon ^{2}\left\vert \nabla w\right\vert ^{2})dx,%
\text{\emph{\ for all }}w\in H^{1}(\Omega )
\end{equation*}%
\emph{where }$C=C(A,\Omega ,d)$\emph{.}
\end{remark}

Let $u_{0}\in H_{0}^{1}(\Omega )$ be the weak solution of (\ref{1.4}). Let $%
z_{\varepsilon }\in H^{1}(\Omega )$ be the unique weak solution of 
\begin{equation}
-\nabla \cdot (A^{\varepsilon }\nabla z_{\varepsilon })=0\text{ in }\Omega 
\text{, }z_{\varepsilon }=\varepsilon \chi _{T}^{\varepsilon }\nabla u_{0}%
\text{ on }\partial \Omega  \label{6.31}
\end{equation}%
where $\Omega $ is as in Lemma \ref{l5.2}. Then we have

\begin{lemma}
\label{l5.3}Let $z_{\varepsilon }$ be as in \emph{(\ref{6.31})} with $%
T=\varepsilon ^{-1}$. Then there exists $\varepsilon _{0}\in \lbrack 0,1)$
such that 
\begin{equation}
\left\Vert z_{\varepsilon }\right\Vert _{H^{1}(\Omega )}\leq C\left(
T^{-1}\left\Vert \chi _{T}\right\Vert _{L^{\infty }(\mathbb{R}^{d})}\right)
^{\frac{1}{2}}\left\Vert u_{0}\right\Vert _{H^{2}(\Omega )},\ 0<\varepsilon
\leq \varepsilon _{0},  \label{6.38}
\end{equation}%
where $C=C(A,\Omega )>0$.
\end{lemma}

It follows from (\ref{6.38}) that for any $\sigma \in (0,1)$, there exists $%
C_{\sigma }=C_{\sigma }(\sigma ,A,\Omega )>0$ such that 
\begin{equation}
\left\Vert z_{\varepsilon }\right\Vert _{H^{1}(\Omega )}\leq C_{\sigma
}(\Theta _{\sigma }(\varepsilon ^{-1}))^{\frac{1}{2}}\left\Vert
u_{0}\right\Vert _{H^{2}(\Omega )},\ 0<\varepsilon \leq \varepsilon _{0}
\label{6.39}
\end{equation}%
where $\Theta _{\sigma }$ is defined by (\ref{11.13}).

For the proof of Lemma \ref{l5.3}, we need the following result whose proof
can be found in \cite{Suslina}.

\begin{lemma}[{\protect\cite[Lemma 5.1]{Suslina}}]
\label{l5.4}Let $\Omega $ be as in Lemma \emph{\ref{l5.2}}. Then there
exists $\delta _{0}\in (0,1]$ depending on $\Omega $ such that, for any $%
u\in H^{1}(\Omega )$, 
\begin{equation}
\int_{\Gamma _{\delta }}\left\vert u\right\vert ^{2}dx\leq C\delta
\left\Vert u\right\Vert _{L^{2}(\Omega )}\left\Vert u\right\Vert
_{H^{1}(\Omega )}\text{, }0<\delta \leq \delta _{0}  \label{5.34}
\end{equation}%
where $C=C(\Omega )$ and $\Gamma _{\delta }=\Omega _{\delta }\cap \Omega $
with $\Omega _{\delta }=\{x\in \mathbb{R}^{d}:\mathrm{dist}(x,\partial
\Omega )<\delta \}$.
\end{lemma}

\begin{proof}[Proof of Lemma \protect\ref{l5.3}]
We set $w=\nabla u_{0}$ and $u=z_{\varepsilon }$. Assuming $u_{0}\in
H^{2}(\Omega )$, we have that $w\in H^{1}(\Omega )^{d}$. Since $\delta
:=T^{-1}\left\Vert \chi _{T}\right\Vert _{L^{\infty }(\mathbb{R}%
^{d})}\rightarrow 0$ as $T\rightarrow \infty $, we may assume that $0<\delta
\leq \delta _{0}$ where $\delta _{0}$ is as in Lemma \ref{l5.4}. Let $\theta
_{\delta }$ be a cut-off function in a neighborhood of $\partial \Omega $
with support in $\Omega _{2\delta }$ (a $2\delta $-neighborhood of $\partial
\Omega $), $\Omega _{\rho }$ being defined as in Lemma \ref{l5.4}: 
\begin{equation}
\theta _{\delta }\in \mathcal{C}_{0}^{\infty }(\mathbb{R}^{d})\text{, 
\textrm{supp}}\theta _{\delta }\subset \Omega _{2\delta }\text{, }0\leq
\theta _{\delta }\leq 1\text{, }\theta _{\delta }=1\text{ on }\Omega
_{\delta }\text{, }\theta _{\delta }=0\text{ on }\mathbb{R}^{d}\backslash
\Omega _{2\delta }\text{ and }\delta \left\vert \nabla \theta _{\delta
}\right\vert \leq C.  \label{5.35}
\end{equation}%
We set $\Phi _{\varepsilon }(x)=\varepsilon \theta _{\delta }(x)\chi
_{T}(x/\varepsilon )w(x)$. Then 
\begin{equation*}
\left\Vert u\right\Vert _{H^{1}(\Omega )}\leq C\varepsilon \left\Vert \chi
_{T}^{\varepsilon }w\right\Vert _{H^{1/2}(\partial \Omega )}\leq C\left\Vert
\Phi _{\varepsilon }\right\Vert _{H^{1}(\Omega )}.
\end{equation*}%
So we need to estimate $\left\Vert \nabla \Phi _{\varepsilon }\right\Vert
_{L^{2}(\Omega )}$. But 
\begin{eqnarray*}
\nabla \Phi _{\varepsilon } &=&\varepsilon \chi _{T}^{\varepsilon }w\nabla
\theta _{\delta }+(\nabla _{y}\chi _{T})^{\varepsilon }w\theta _{\delta
}+\varepsilon \chi _{T}^{\varepsilon }\theta _{\delta }\nabla w \\
&=&J_{1}+J_{2}+J_{3}.
\end{eqnarray*}%
We have 
\begin{eqnarray*}
\left\Vert J_{1}\right\Vert _{L^{2}(\Omega )}^{2} &\leq &C\varepsilon
^{2}\left\Vert \chi _{T}\right\Vert _{L^{\infty }(\mathbb{R}^{d})}^{2}\delta
^{-2}\int_{\Gamma _{2\delta }}\left\vert w\right\vert ^{2}dx \\
&\leq &C\int_{\Gamma _{2\delta }}\left\vert w\right\vert ^{2}dx\leq C\delta
\left\Vert w\right\Vert _{L^{2}(\Omega )}\left\Vert w\right\Vert
_{H^{1}(\Omega )}
\end{eqnarray*}%
where we have used (\ref{5.34}) for the last inequality above. For $J_{2}$,
we have (using (\ref{5.30}) and (\ref{5.34})) 
\begin{eqnarray*}
\left\Vert J_{2}\right\Vert _{L^{2}(\Omega )}^{2} &\leq &\int_{\Omega
}\left\vert (\nabla _{y}\chi _{T})^{\varepsilon }w\theta _{\delta
}\right\vert ^{2}dx\leq C\int_{\Omega }\left( \left\vert w\theta _{\delta
}\right\vert ^{2}+\delta ^{2}\left\vert \nabla (w\theta _{\delta
})\right\vert ^{2}\right) dx \\
&\leq &C\int_{\Gamma _{2\delta }}\left\vert w\right\vert ^{2}dx+C\delta
^{2}\int_{\Omega }\left\vert \nabla (w\theta _{\delta })\right\vert ^{2}dx \\
&\leq &C\delta \left\Vert w\right\Vert _{L^{2}(\Omega )}\left\Vert
w\right\Vert _{H^{1}(\Omega )}+C\delta ^{2}\int_{\Omega }\left\vert \nabla
(w\theta _{\delta })\right\vert ^{2}dx.
\end{eqnarray*}%
But $\nabla (w\theta _{\delta })=w\nabla \theta _{\delta }+\theta _{\delta
}\nabla w$, and 
\begin{eqnarray*}
\int_{\Omega }\left\vert \nabla (w\theta _{\delta })\right\vert ^{2}dx &\leq
&C\int_{\Gamma _{2\delta }}\left\vert \nabla \theta _{\delta }\right\vert
^{2}\left\vert w\right\vert ^{2}dx+C\int_{\Omega }\left\vert \theta _{\delta
}\nabla w\right\vert ^{2}dx \\
&\leq &C\delta ^{-1}\left\Vert w\right\Vert _{L^{2}(\Omega )}\left\Vert
w\right\Vert _{H^{1}(\Omega )}+C\int_{\Omega }\left\vert \nabla w\right\vert
^{2}dx.
\end{eqnarray*}%
Hence 
\begin{equation*}
\left\Vert J_{2}\right\Vert _{L^{2}(\Omega )}^{2}\leq C\delta \left\Vert
w\right\Vert _{L^{2}(\Omega )}\left\Vert w\right\Vert _{H^{1}(\Omega
)}+C\delta ^{2}\left\Vert w\right\Vert _{H^{1}(\Omega )}^{2}.
\end{equation*}%
As for $J_{3}$, 
\begin{equation*}
\left\Vert J_{3}\right\Vert _{L^{2}(\Omega )}^{2}\leq C\varepsilon
^{2}\int_{\Omega }\left\vert \chi _{T}^{\varepsilon }\right\vert
^{2}\left\vert \nabla w\right\vert ^{2}dx\leq C\delta ^{2}\left\Vert
w\right\Vert _{H^{1}(\Omega )}^{2}.
\end{equation*}%
Finally, using Young's inequality together with the fact that $\delta
^{2}\leq \delta $ we are led to 
\begin{eqnarray}
\left\Vert \nabla \Phi _{\varepsilon }\right\Vert _{L^{2}(\Omega )}^{2}
&\leq &C\delta \left\Vert w\right\Vert _{L^{2}(\Omega )}\left\Vert
w\right\Vert _{H^{1}(\Omega )}+C\delta ^{2}\left\Vert w\right\Vert
_{H^{1}(\Omega )}^{2}  \label{5.35'} \\
&\leq &C\delta \left\Vert w\right\Vert _{H^{1}(\Omega )}^{2}+C\delta
^{2}\left\Vert w\right\Vert _{H^{1}(\Omega )}^{2}  \notag \\
&\leq &C\delta \left\Vert w\right\Vert _{H^{1}(\Omega )}^{2}.  \notag
\end{eqnarray}%
So we choose $\varepsilon _{0}$ such that $0<\delta \leq \delta _{0}$ for $%
0<\varepsilon \leq \varepsilon _{0}$ (recall that $0<\delta \rightarrow 0$
as $0<\varepsilon \rightarrow 0$). We thus derive (\ref{6.38}) since $%
\left\Vert \Phi _{\varepsilon }\right\Vert _{L^{2}(\Omega )}^{2}\leq \delta
\left\Vert w\right\Vert _{H^{1}(\Omega )}^{2}$.
\end{proof}

\subsection{Convergence rates: proof of Theorem \protect\ref{t1.4}}

Assume that $\Omega $ is of class $\mathcal{C}^{1,1}$. Let $u_{\varepsilon }$%
, $u_{0}\in H_{0}^{1}(\Omega )$ be the weak solutions of (\ref{1.1}) and (%
\ref{1.4}) respectively. Let $\chi _{T}^{\varepsilon }(x)=\chi
_{T}(x/\varepsilon )$ for $x\in \Omega $ and define 
\begin{equation}
w_{\varepsilon }=u_{\varepsilon }-u_{0}-\varepsilon \chi _{T}^{\varepsilon
}\nabla u_{0}+z_{\varepsilon }  \label{6.30}
\end{equation}%
where $T=\varepsilon ^{-1}$ and $z_{\varepsilon }\in H^{1}(\Omega )$ is the
weak solution of (\ref{6.31}).

\begin{theorem}
\label{t6.1}Suppose that $A$ is as in the preceding subsection. Assume that $%
u_{0}\in H^{2}(\Omega )$. Then for any $\sigma \in (0,1)$ there exists $%
C_{\sigma }=C_{\sigma }(\sigma ,A,\Omega )$ such that 
\begin{equation}
\left\Vert w_{\varepsilon }\right\Vert _{H^{1}(\Omega )}\leq C_{\sigma
}\left( \left\Vert \nabla \chi -\nabla \chi _{\varepsilon ^{-1}}\right\Vert
_{2}+\Theta _{\sigma }(\varepsilon ^{-1})\right) \left\Vert u_{0}\right\Vert
_{H^{2}(\Omega )}.  \label{6.37}
\end{equation}
\end{theorem}

\begin{proof}
Set 
\begin{equation*}
A_{T}=A+A\nabla _{y}\chi _{T}-A^{\ast }
\end{equation*}%
where $A^{\ast }$ is the homogenized matrix and where we have taken $%
T=\varepsilon ^{-1}$. Then by simple computations as in Lemma \ref{l5.1} we
get

\begin{equation*}
-\nabla \cdot \left( A^{\varepsilon }\nabla w_{\varepsilon }\right) =\nabla
\cdot \left( A_{T}^{\varepsilon }\nabla u_{0}\right) +\varepsilon \nabla
\cdot (A^{\varepsilon }\nabla ^{2}u_{0}\chi _{T}^{\varepsilon }).
\end{equation*}%
This implies that 
\begin{equation}
\left\Vert \nabla w_{\varepsilon }\right\Vert _{L^{2}(\Omega )}\leq
C\left\Vert A_{T}^{\varepsilon }\nabla u_{0}\right\Vert _{L^{2}(\Omega
)}+C\varepsilon \left\Vert A^{\varepsilon }\nabla ^{2}u_{0}\chi
_{T}^{\varepsilon }\right\Vert _{L^{2}(\Omega )}.  \label{10.6}
\end{equation}%
We use (\ref{11.14}) to get 
\begin{eqnarray}
\varepsilon \left\Vert A^{\varepsilon }\nabla ^{2}u_{0}\chi
_{T}^{\varepsilon }\right\Vert _{L^{2}(\Omega )} &\leq &C\varepsilon
\left\Vert \chi _{T}\right\Vert _{L^{\infty }(\mathbb{R}^{d})}\left\Vert
\nabla ^{2}u_{0}\right\Vert _{L^{2}(\Omega )}  \label{10.7} \\
&\leq &C\Theta _{\sigma }(T)\left\Vert \nabla ^{2}u_{0}\right\Vert
_{L^{2}(\Omega )}.  \notag
\end{eqnarray}%
Concerning the term $\left\Vert A_{T}^{\varepsilon }\nabla u_{0}\right\Vert
_{L^{2}(\Omega )}$, we need to replace $A_{T}$ by a matrix $\mathcal{A}_{T}$
whose mean value is zero. So, we let $\mathcal{A}_{T}=A_{T}-\left\langle
A_{T}\right\rangle $ so that $\left\langle \mathcal{A}_{T}\right\rangle =0$
and $A_{T}^{\varepsilon }\nabla u_{0}=\mathcal{A}_{T}^{\varepsilon }\nabla
u_{0}+\left\langle A_{T}\right\rangle \nabla u_{0}$. The inequality $%
\left\vert \left\langle A_{T}\right\rangle \right\vert \leq C\left\Vert
\nabla \chi -\nabla \chi _{T}\right\Vert _{2}$ yields readily 
\begin{equation}
\left\Vert \left\langle A_{T}\right\rangle \nabla u_{0}\right\Vert
_{L^{2}(\Omega )}\leq C\left\Vert \nabla \chi -\nabla \chi _{T}\right\Vert
_{2}\left\Vert \nabla u_{0}\right\Vert _{L^{2}(\Omega )}.  \label{10.8}
\end{equation}%
It remains to estimate $\left\Vert \mathcal{A}_{T}^{\varepsilon }\nabla
u_{0}\right\Vert _{L^{2}(\Omega )}$. We denote by $a_{T,ij}$ the entries of $%
\mathcal{A}_{T}$: $a_{T,ij}=b_{T,ij}-\left\langle b_{T,ij}\right\rangle
\equiv a_{ij}$ where 
\begin{equation*}
b_{T,ij}(y)=b_{ij}(y)+\sum_{k=1}^{d}b_{ik}(y)\frac{\partial \chi _{T,j}}{%
\partial y_{k}}(y)-b_{ij}^{\ast }.
\end{equation*}%
In view of Lemma \ref{l11.2}, let $f_{T,ij}\equiv f_{ij}\in H_{\infty
,AP}^{1}(\mathbb{R}^{d})$ be the unique solution of 
\begin{equation*}
-\Delta f_{ij}+T^{-2}f_{ij}=a_{ij}\text{ in }\mathbb{R}^{d},\ \ \left\langle
f_{ij}\right\rangle =0.
\end{equation*}%
Owing to (\ref{e5.7}), we see that $a_{ij}$ verifies (\ref{11.15}), so that (%
\ref{11.19}) and (\ref{11.20}) are satisfied, that is: 
\begin{equation}
T^{-2}\left\Vert f_{ij}\right\Vert _{L^{\infty }(\mathbb{R}^{d})}\leq
C\Theta _{1}(T)\text{ and }T^{-1}\left\Vert \nabla f_{ij}\right\Vert
_{L^{\infty }(\mathbb{R}^{d})}\leq C\Theta _{\sigma }(T).  \label{11.26}
\end{equation}%
We set $\mathbf{f}=(f_{ij})_{1\leq i,j\leq d}$. Then writing (formally) 
\begin{equation*}
a_{ij}=-\sum_{k=1}^{d}\left( \frac{\partial }{\partial y_{k}}\left( \frac{%
\partial f_{ij}}{\partial y_{k}}-\frac{\partial f_{kj}}{\partial y_{i}}%
\right) +\frac{\partial }{\partial y_{i}}\left( \frac{\partial f_{kj}}{%
\partial y_{k}}\right) \right) +T^{-2}f_{ij}
\end{equation*}%
and using the fact that 
\begin{equation*}
\sum_{i,k=1}^{d}\frac{\partial ^{2}}{\partial y_{i}\partial y_{k}}\left( 
\frac{\partial f_{ij}}{\partial y_{k}}-\frac{\partial f_{kj}}{\partial y_{i}}%
\right) =0,
\end{equation*}%
we readily get 
\begin{align}
-\nabla \cdot (\mathcal{A}_{T}^{\varepsilon }\nabla u_{0})& =\nabla \cdot
\left( (\Delta \mathbf{f})^{\varepsilon }\nabla u_{0}\right) -T^{-2}\nabla
\cdot (\mathbf{f}^{\varepsilon }\nabla u_{0})  \label{11.27} \\
& =\sum_{i,j,k=1}^{d}\frac{\partial }{\partial x_{i}}\left( \frac{\partial }{%
\partial x_{k}}\left( \frac{\partial f_{ij}}{\partial x_{k}}-\frac{\partial
f_{kj}}{\partial x_{i}}\right) \left( \frac{x}{\varepsilon }\right) \frac{%
\partial u_{0}}{\partial x_{j}}\right)  \notag \\
& +\sum_{i,j,k=1}^{d}\frac{\partial }{\partial x_{i}}\left( \frac{\partial
^{2}f_{kj}}{\partial x_{k}\partial x_{i}}\left( \frac{x}{\varepsilon }%
\right) \frac{\partial u_{0}}{\partial x_{j}}\right) -T^{-2}\nabla \cdot (%
\mathbf{f}^{\varepsilon }\nabla u_{0})  \notag \\
& =-\sum_{i,j,k=1}^{d}\frac{\partial }{\partial x_{i}}\left( \varepsilon
\left( \frac{\partial f_{ij}}{\partial x_{k}}-\frac{\partial f_{kj}}{%
\partial x_{i}}\right) \left( \frac{x}{\varepsilon }\right) \frac{\partial
^{2}u_{0}}{\partial x_{k}\partial x_{j}}\right)  \notag \\
& +\sum_{i,j,k=1}^{d}\frac{\partial }{\partial x_{i}}\left( \frac{\partial
^{2}f_{kj}}{\partial x_{k}\partial x_{i}}\left( \frac{x}{\varepsilon }%
\right) \frac{\partial u_{0}}{\partial x_{j}}\right) -T^{-2}\nabla \cdot (%
\mathbf{f}^{\varepsilon }\nabla u_{0}).  \notag
\end{align}%
Testing (\ref{11.27}) with $\varphi \in H_{0}^{1}(\Omega )$, we obtain 
\begin{eqnarray}
\left\Vert \mathcal{A}_{T}^{\varepsilon }\nabla u_{0}\right\Vert
_{L^{2}(\Omega )} &\leq &C\varepsilon \left( \int_{\Omega }\left\vert \nabla 
\mathbf{f}\left( \frac{x}{\varepsilon }\right) \right\vert ^{2}\left\vert
\nabla ^{2}u_{0}\right\vert ^{2}dx\right) ^{\frac{1}{2}}  \label{6.35} \\
&&+C\sum_{j=1}^{d}\left( \int_{\Omega }\left\vert \nabla h_{T,j}\left( \frac{%
x}{\varepsilon }\right) \right\vert ^{2}\left\vert \nabla u_{0}\right\vert
^{2}dx\right) ^{\frac{1}{2}}+\left\vert \left\langle A_{T}\right\rangle
\right\vert \left\Vert \nabla u_{0}\right\Vert _{L^{2}(\Omega )}  \notag \\
&&+CT^{-2}\left( \int_{\Omega }\left\vert \mathbf{f}^{\varepsilon
}\right\vert ^{2}\left\vert \nabla u_{0}\right\vert ^{2}dx\right) ^{\frac{1}{%
2}}  \notag \\
&=&I_{1}+I_{2}+I_{3}+I_{4}  \notag
\end{eqnarray}%
where $h_{T,j}=\sum_{k=1}^{d}\frac{\partial f_{kj}}{\partial y_{k}}\in
L_{\infty ,AP}^{2}(\mathbb{R}^{d})$. We estimate each term above separately.
Let us first deal with $I_{2}$. Observe that $h_{T,j}=\func{div}f_{.j}$
where $f_{.j}=(f_{kj})_{1\leq k\leq d}$. It follows from the definition of $%
f_{ij}$ that 
\begin{equation*}
-\Delta f_{.j}+T^{-2}f_{.j}=A(e_{j}+\nabla \chi _{T,j})-\left\langle
A(e_{j}+\nabla \chi _{T,j})\right\rangle ,
\end{equation*}%
so that, owing to the definition of $\chi _{T,j}$, 
\begin{equation}
-\Delta h_{T,j}+T^{-2}h_{T,j}=T^{-2}\chi _{T,j}.  \label{6.36}
\end{equation}%
Next, since the function $g=T^{-1}\chi _{T,j}$ satisfies assumption (\ref%
{11.15}) of Lemma \ref{l11.2} with $\sigma =1$, it follows that $h_{T,j}$
satisfies estimate (\ref{11.20}), that is, 
\begin{equation*}
T^{-1}\left\Vert \nabla h_{T,j}\right\Vert _{L^{\infty }(\mathbb{R}%
^{d})}\leq C_{\tau }\Theta _{\tau }(T)\ \ \forall \tau \in (0,1).
\end{equation*}%
Therefore 
\begin{equation*}
\left\vert I_{2}\right\vert \leq C\varepsilon \left\Vert \nabla
h_{T,j}\right\Vert _{L^{\infty }(\mathbb{R}^{d})}\left\Vert \nabla
u_{0}\right\Vert _{L^{2}(\Omega )}\leq C\Theta _{\sigma }(T)\left\Vert
\nabla u_{0}\right\Vert _{L^{2}(\Omega )}.
\end{equation*}

As regard $I_{1}$, we infer from (\ref{11.26}) that 
\begin{equation*}
\left\vert I_{1}\right\vert \leq C\varepsilon \left\Vert \nabla \mathbf{f}%
\right\Vert _{L^{\infty }(\mathbb{R}^{d})}\left\Vert \nabla
^{2}u_{0}\right\Vert _{L^{2}(\Omega )}\leq C\Theta _{\sigma }(T)\left\Vert
\nabla ^{2}u_{0}\right\Vert _{L^{2}(\Omega )}.
\end{equation*}%
Concerning $I_{4}$, we use the first inequality in (\ref{11.26}) to get 
\begin{equation*}
\left\vert I_{4}\right\vert \leq C\Theta _{1}(T)\left\Vert \nabla
u_{0}\right\Vert _{L^{2}(\Omega )}
\end{equation*}%
where we have put $T=\varepsilon ^{-1}$. Finally, using the inequality $%
\left\vert \left\langle A_{T}\right\rangle \right\vert \leq C\left\Vert
\nabla \chi -\nabla \chi _{T}\right\Vert _{2}$ we get 
\begin{equation*}
\left\vert I_{3}\right\vert \leq C\left\Vert \nabla \chi -\nabla \chi
_{T}\right\Vert _{2}\left\Vert \nabla u_{0}\right\Vert _{L^{2}(\Omega )}.
\end{equation*}%
The result follows thereby.
\end{proof}

We are now in a position to prove Theorem \ref{t1.4}.

\begin{proof}[Proof of Theorem \protect\ref{t1.4}]
Using (\ref{6.37}) together with (\ref{6.39}) we get, for any $\sigma \in
(0,1)$, 
\begin{equation*}
\begin{array}{l}
\left\Vert u_{\varepsilon }-u_{0}-\varepsilon \chi _{T=\varepsilon
^{-1}}^{\varepsilon }\nabla u_{0}\right\Vert _{H^{1}(\Omega )} \\ 
\leq \left\Vert u_{\varepsilon }-u_{0}-\varepsilon \chi _{T=\varepsilon
^{-1}}^{\varepsilon }\nabla u_{0}+z_{\varepsilon }\right\Vert _{H^{1}(\Omega
)}+\left\Vert z_{\varepsilon }\right\Vert _{H^{1}(\Omega )} \\ 
\leq C\left( \left\Vert \nabla \chi -\nabla \chi _{\varepsilon
^{-1}}\right\Vert _{2}+\Theta _{\sigma }(\varepsilon ^{-1})\right)
\left\Vert u_{0}\right\Vert _{H^{2}(\Omega )}+C_{\sigma }(\Theta _{\sigma
}(\varepsilon ^{-1}))^{\frac{1}{2}}\left\Vert u_{0}\right\Vert
_{H^{2}(\Omega )} \\ 
\leq C\left( \left\Vert \nabla \chi -\nabla \chi _{\varepsilon
^{-1}}\right\Vert _{2}+\left( \Theta _{1}(\varepsilon ^{-1})\right) ^{\sigma
}+\left( \Theta _{1}(\varepsilon ^{-1})\right) ^{\frac{\sigma }{2}}\right)
\left\Vert u_{0}\right\Vert _{H^{2}(\Omega )} \\ 
\leq C\left( \left\Vert \nabla \chi -\nabla \chi _{\varepsilon
^{-1}}\right\Vert _{2}+\left( \Theta _{1}(\varepsilon ^{-1})\right) ^{\sigma
}\right) ^{\frac{1}{2}}\left\Vert u_{0}\right\Vert _{H^{2}(\Omega )},%
\end{array}%
\end{equation*}%
the last inequality above stemming from the fact that $\left\Vert \nabla
\chi -\nabla \chi _{\varepsilon ^{-1}}\right\Vert _{2}+\left( \Theta
_{1}(\varepsilon ^{-1})\right) ^{\sigma }\rightarrow 0$ when $\varepsilon
\rightarrow 0$, so that we may assume 
\begin{equation*}
\left\Vert \nabla \chi -\nabla \chi _{\varepsilon ^{-1}}\right\Vert
_{2}+\left( \Theta _{1}(\varepsilon ^{-1})\right) ^{\sigma }<1\text{ for
sufficiently small }\varepsilon .
\end{equation*}%
Choosing $\sigma =\frac{1}{2}$, we obtain 
\begin{equation}
\left\Vert u_{\varepsilon }-u_{0}-\varepsilon \chi _{T=\varepsilon
^{-1}}^{\varepsilon }\nabla u_{0}\right\Vert _{H^{1}(\Omega )}\leq C\left(
\left\Vert \nabla \chi -\nabla \chi _{\varepsilon ^{-1}}\right\Vert
_{2}+\left( \Theta _{1}(\varepsilon ^{-1})\right) ^{\frac{1}{2}}\right) ^{%
\frac{1}{2}}\left\Vert u_{0}\right\Vert _{H^{2}(\Omega )}.  \label{5.37''}
\end{equation}%
We recall that, since $\Omega $ is a $\mathcal{C}^{1,1}$-bounded domain in $%
\mathbb{R}^{d}$ and the matrix $A^{\ast }$ has constant entries, it holds
that 
\begin{equation}
\left\Vert u_{0}\right\Vert _{H^{2}(\Omega )}\leq C\left\Vert f\right\Vert
_{L^{2}(\Omega )},\ C=C(d,\alpha ,\Omega )>0.  \label{5.37'}
\end{equation}%
Next, set for $\varepsilon \in (0,1]$, 
\begin{equation*}
\eta (\varepsilon )=\left( \left\Vert \nabla \chi -\nabla \chi _{\varepsilon
^{-1}}\right\Vert _{2}+\left( \Theta _{1}(\varepsilon ^{-1})\right) ^{\frac{1%
}{2}}\right) ^{\frac{1}{2}}.
\end{equation*}%
Since $\eta (\varepsilon )\rightarrow 0$ as $\varepsilon \rightarrow 0$, we
obtain from (\ref{5.37''}) and (\ref{5.37'}), the statement of (\ref{Eq03})
in Theorem \ref{t1.4}.

It remains to check the near optimal convergence rates result (\ref{Eq02}).
We proceed in two parts.\medskip

\textit{Part I}. We first check that 
\begin{equation}
\left\Vert u_{\varepsilon }\right\Vert _{H^{1}(\Gamma _{2\delta })}\leq
C\eta (\varepsilon )\left\Vert f\right\Vert _{L^{2}(\Omega )}\text{ where }%
\delta =\left( \eta (\varepsilon )\right) ^{2}.  \label{5.37}
\end{equation}%
Indeed, we have $u_{\varepsilon }=(u_{\varepsilon }-u_{0}-\varepsilon \chi
_{T}^{\varepsilon }\nabla u_{0})+u_{0}+\varepsilon \chi _{T}^{\varepsilon
}\nabla u_{0}$, so that 
\begin{equation*}
\left\Vert u_{\varepsilon }\right\Vert _{H^{1}(\Gamma _{2\delta })}\leq
\left\Vert u_{\varepsilon }-u_{0}-\varepsilon \chi _{T}^{\varepsilon }\nabla
u_{0}\right\Vert _{H^{1}(\Gamma _{2\delta })}+\left\Vert u_{0}\right\Vert
_{H^{1}(\Gamma _{2\delta })}+\left\Vert \varepsilon \chi _{T}^{\varepsilon
}\nabla u_{0}\right\Vert _{H^{1}(\Gamma _{2\delta })}.
\end{equation*}%
It follows from (\ref{Eq03}) and (\ref{5.37'}) that 
\begin{equation}
\left\Vert u_{\varepsilon }-u_{0}-\varepsilon \chi _{T}^{\varepsilon }\nabla
u_{0}\right\Vert _{H^{1}(\Gamma _{2\delta })}\leq C\eta (\varepsilon
)\left\Vert u_{0}\right\Vert _{H^{2}(\Omega )}\leq C\eta (\varepsilon
)\left\Vert f\right\Vert _{L^{2}(\Omega )}.  \label{Eq04}
\end{equation}%
Using (\ref{5.34}) we obtain 
\begin{equation}
\left\Vert u_{0}\right\Vert _{H^{1}(\Gamma _{2\delta })}\leq C\delta ^{\frac{%
1}{2}}\left\Vert u_{0}\right\Vert _{H^{2}(\Omega )}\leq C\delta ^{\frac{1}{2}%
}\left\Vert f\right\Vert _{L^{2}(\Omega )}.  \label{Eq05}
\end{equation}%
To estimate $\left\Vert \varepsilon \chi _{T}^{\varepsilon }\nabla
u_{0}\right\Vert _{H^{1}(\Gamma _{2\delta })}$, we consider a cut-off
function $\theta _{2\delta }$ of the same form as in (\ref{5.35}), but with $%
\delta $ replaced there by $2\delta $. Letting $w=\nabla u_{0}$, we observe
that $\varepsilon \chi _{T}^{\varepsilon }w=\varepsilon \theta _{2\delta
}\chi _{T}^{\varepsilon }w$ on $\Gamma _{2\delta }$, so that 
\begin{equation*}
\nabla (\varepsilon \chi _{T}^{\varepsilon }w)=\varepsilon \chi
_{T}^{\varepsilon }w\nabla \theta _{2\delta }+(\nabla _{y}\chi
_{T})^{\varepsilon }w\theta _{2\delta }+\varepsilon \chi _{T}^{\varepsilon
}\theta _{2\delta }\nabla w\text{ on }\Gamma _{2\delta }.
\end{equation*}%
Following the same procedure as in the proof of Lemma \ref{l5.3}, we get 
\begin{equation}
\left\Vert \varepsilon \chi _{T}^{\varepsilon }\nabla u_{0}\right\Vert
_{H^{1}(\Gamma _{2\delta })}\leq C\delta ^{\frac{1}{2}}\left\Vert
u_{0}\right\Vert _{H^{2}(\Omega )}\leq C\delta ^{\frac{1}{2}}\left\Vert
f\right\Vert _{L^{2}(\Omega )}.  \label{Eq06}
\end{equation}%
Choosing $\delta =\left( \eta (\varepsilon )\right) ^{2}$ in (\ref{Eq05})
and (\ref{Eq06}), and taking into account (\ref{Eq04}), we readily get (\ref%
{5.37}).\medskip

\textit{Part II}. Note that (\ref{6.37}) implies 
\begin{equation}
\left\Vert u_{\varepsilon }-u_{0}-\varepsilon \chi _{T}^{\varepsilon }\nabla
u_{0}+z_{\varepsilon }\right\Vert _{L^{2}(\Omega )}\leq C\left( \eta
(\varepsilon )\right) ^{2}\left\Vert f\right\Vert _{L^{2}(\Omega )}.
\label{5.38}
\end{equation}%
Thus, using the inequality 
\begin{equation}
\left\Vert \varepsilon \chi _{T}^{\varepsilon }\nabla u_{0}\right\Vert
_{L^{2}(\Omega )}\leq C\left( \Theta _{1}(\varepsilon ^{-1})\right) ^{\frac{1%
}{2}}\left\Vert u_{0}\right\Vert _{H^{2}(\Omega )}\leq C\left( \eta
(\varepsilon )\right) ^{2}\left\Vert f\right\Vert _{L^{2}(\Omega )},
\label{5.39}
\end{equation}%
we see that proving (\ref{Eq02}) amounts to prove that 
\begin{equation}
\left\Vert z_{\varepsilon }\right\Vert _{L^{2}(\Omega )}\leq C\left( \eta
(\varepsilon )\right) ^{2}\left\Vert f\right\Vert _{L^{2}(\Omega )}
\label{5.40}
\end{equation}%
where $C=C(d,A,\Omega )$. To that end, we consider the function 
\begin{equation}
v_{\varepsilon }=z_{\varepsilon }-\Phi _{\varepsilon }\text{, where }\Phi
_{\varepsilon }=\varepsilon \theta _{\delta }\chi _{T}^{\varepsilon }\nabla
u_{0}\text{ with }\delta =\left( \eta (\varepsilon )\right) ^{2}.
\label{5.41}
\end{equation}%
Then $v_{\varepsilon }\in H_{0}^{1}(\Omega )$ and $-\nabla \cdot
(A^{\varepsilon }\nabla v_{\varepsilon })=F_{\varepsilon }\equiv \nabla
\cdot (A^{\varepsilon }\nabla \Phi _{\varepsilon })$ in $\Omega $. As shown
in (\ref{5.35'}) (where we use the inequality (\ref{5.7'})), we have 
\begin{equation}
\left\Vert \nabla \Phi _{\varepsilon }\right\Vert _{L^{2}(\Omega )}\leq
C\delta ^{\frac{1}{2}}\left\Vert f\right\Vert _{L^{2}(\Omega )}\text{ and }%
\left\Vert \Phi _{\varepsilon }\right\Vert _{L^{2}(\Omega )}\leq C\delta
\left\Vert f\right\Vert _{L^{2}(\Omega )}.  \label{5.42}
\end{equation}%
Now, let $F\in L^{2}(\Omega )$ be arbitrarily fixed, and let $t_{\varepsilon
}\in H_{0}^{1}(\Omega )$ be the solution of 
\begin{equation}
-\nabla \cdot (A^{\varepsilon }\nabla t_{\varepsilon })=F\text{ in }\Omega .
\label{5.43}
\end{equation}%
Following the homogenization process of (\ref{1.1}) (see the proof of
Theorem \ref{t1.1} in Section 2), we deduce the existence of a function $%
t_{0}\in H_{0}^{1}(\Omega )$ such that $t_{\varepsilon }\rightarrow t_{0}$
in $H_{0}^{1}(\Omega )$-weak and $t_{0}$ solves uniquely the equation $%
-\nabla \cdot (A^{\ast }\nabla t_{0})=F$ in $\Omega $. It follows from (\ref%
{5.37}) that 
\begin{equation}
\left\Vert \nabla t_{\varepsilon }\right\Vert _{L^{2}(\Gamma _{2\delta
})}\leq C\eta (\varepsilon )\left\Vert F\right\Vert _{L^{2}(\Omega )}.
\label{5.44}
\end{equation}%
Taking in the variational form of (\ref{5.43}) $v_{\varepsilon }$ test
function, we obtain 
\begin{eqnarray}
\int_{\Omega }Fv_{\varepsilon }dx &=&\int_{\Omega }A^{\varepsilon }\nabla
t_{\varepsilon }\cdot \nabla v_{\varepsilon }dx=\int_{\Omega }\nabla
t_{\varepsilon }\cdot A^{\varepsilon }\nabla v_{\varepsilon }dx=\left(
F_{\varepsilon },t_{\varepsilon }\right)  \label{5.45} \\
&=&-\int_{\Omega }A^{\varepsilon }\nabla \Phi _{\varepsilon }\cdot \nabla
t_{\varepsilon }dx=-\int_{\Gamma _{2\delta }}A^{\varepsilon }\nabla \Phi
_{\varepsilon }\cdot \nabla t_{\varepsilon }dx  \notag
\end{eqnarray}%
where in (\ref{5.45}), the second equality stems from the fact that the
matrix $A$ is symmetric, and in the last equality we have used the
definition and properties of $\Phi _{\varepsilon }$. Hence, using together
(the first inequality in) (\ref{5.42}) and (\ref{5.44}), we are led to 
\begin{eqnarray*}
\left\vert \int_{\Omega }Fv_{\varepsilon }dx\right\vert &\leq &C\left\Vert
\nabla \Phi _{\varepsilon }\right\Vert _{L^{2}(\Omega )}\left\Vert \nabla
t_{\varepsilon }\right\Vert _{L^{2}(\Gamma _{2\delta })}\leq C\delta ^{\frac{%
1}{2}}\left\Vert f\right\Vert _{L^{2}(\Omega )}\delta ^{\frac{1}{2}%
}\left\Vert F\right\Vert _{L^{2}(\Omega )} \\
&\leq &C\delta \left\Vert f\right\Vert _{L^{2}(\Omega )}\left\Vert
F\right\Vert _{L^{2}(\Omega )}.
\end{eqnarray*}%
Since $F$ is arbitrary, it emerges 
\begin{equation}
\left\Vert v_{\varepsilon }\right\Vert _{L^{2}(\Omega )}\leq C\delta
\left\Vert f\right\Vert _{L^{2}(\Omega )}\text{ with }\delta =(\eta
(\varepsilon ))^{2}.  \label{5.46}
\end{equation}%
Combining (\ref{5.46}) with the second estimate in (\ref{5.42}) yields (\ref%
{5.40}). This concludes the proof of Theorem \ref{t1.4}.
\end{proof}

\begin{remark}
\label{r5.3}\emph{In the asymptotic periodic setting of the preceding
section, we replace }$\chi _{T}$\emph{\ by }$\chi $\emph{\ so that }$%
\left\Vert \nabla \chi -\nabla \chi _{\varepsilon ^{-1}}\right\Vert _{2}=0$%
\emph{. Moreover, if we look carefully at the proof of (\ref{Eq02}), we
notice that, in view of Remark \ref{r5.2}, we may replace }$\eta
(\varepsilon )$\emph{\ by }$\varepsilon ^{1/2}$\emph{, so that (\ref{Eq02})
becomes }%
\begin{equation*}
\left\Vert u_{\varepsilon }-u_{0}\right\Vert _{L^{2}(\Omega )}\leq
C\varepsilon \left\Vert f\right\Vert _{L^{2}(\Omega )}
\end{equation*}%
\emph{where }$C=C(d,\alpha ,\Omega )$\emph{. This shows the optimal }$L^{2}$%
\emph{-rates of convergence in Theorem \ref{t5.1}.}
\end{remark}

\section{Some examples}

\subsection{Applications of Theorem \protect\ref{t3.2}}

Theorem \ref{t3.2} has been proved under the assumption that the corrector $%
\chi _{j}$ lies in $B_{\mathcal{A}}^{2}(\mathbb{R}^{d})$ for each $1\leq
j\leq d$. We provide some examples in which this hypothesis is fulfilled.

\subsubsection{\textbf{The almost periodic setting}}

We assume here that the entries of the matrix $A$ are almost periodic in the
sense of Besicovitch \cite{Besicovitch}. Then this falls into the scope of
Theorem \ref{t1.1} by taking there $\mathcal{A}=AP(\mathbb{R}^{d})$.

Now, we distinguish two special cases.

\textbf{Case 1}. The entries of $A$ are continuous quasi-periodic functions
and satisfy the frequency condition (see \cite{Jikov2}). We recall that a
function $b$ defined on $\mathbb{R}^{d}$ is quasi-periodic if $b(y)=\mathcal{%
B}(\omega _{1}\cdot y,...,\omega _{m}\cdot y)$ where $\mathcal{B}\equiv 
\mathcal{B}(z_{1},...,z_{m})$ is a $1$-periodic function with respect to
every argument $z_{1}$,..., $z_{m}$. The $\omega ^{1},...,\omega ^{m}$ are
the frequency vectors, and $\omega _{j}\cdot y=\sum_{i=1}^{d}\omega
_{j}^{i}y_{i}$ is the inner product of vectors in $\mathbb{R}^{d}$. The
frequency condition on the vectors $\omega ^{1},...,\omega ^{m}\in \mathbb{R}%
^{d}$ amounts to the following assumption:

\begin{itemize}
\item[(\textbf{FC})] There is $c_{0},\tau >0$ such that 
\begin{equation}
\left\vert \sum_{j=1}^{m}k_{j}\omega _{i}^{j}\right\vert \geq
c_{0}\left\vert k\right\vert ^{-\tau }\text{ for all }k\in \mathbb{Z}%
^{m}\backslash \{0\}\text{ and }1\leq i\leq d.  \label{FC}
\end{equation}
\end{itemize}

It is clear that if (\textbf{FC}) is satisfied, then the vectors $\omega
^{1},...,\omega ^{m}$ \ are rationally independent, that is, 
\begin{equation*}
\sum_{j=1}^{m}k_{j}\omega _{i}^{j}\neq 0\text{ for every }1\leq i\leq d\text{
and all }k\in \mathbb{Z}^{m}\backslash \{0\}\text{.}
\end{equation*}%
Then as shown in \cite[Lemma 2.1]{Jikov2}, the corrector problem (\ref{1.6})
possesses a solution which is quasi-periodic. So, it belongs to $B_{AP}^{2}(%
\mathbb{R}^{d})$ (the space $B_{\mathcal{A}}^{2}(\mathbb{R}^{d})$ with $%
\mathcal{A}=AP(\mathbb{R}^{d})$) since any quasi-periodic function is almost
periodic. We may hence apply Theorem \ref{t3.2}.

\textbf{Case 2}. The entries of $A$ are continuous almost periodic
functions. In \cite[Theorem 1.1]{Armstrong} are formulated the assumptions
implying the existence of bounded almost periodic solution to the problem (%
\ref{1.6}). Hence the conclusion of Theorem \ref{t3.2} holds. Notice that
this class of solutions contains continuous quasi-periodic ones (provided
that the assumptions of \cite[Theorem 1.1]{Armstrong} are satisfied) but
also some other almost periodic functions that are not quasi-periodic as
shown in \cite[Section 4]{Armstrong}.

\subsubsection{\textbf{The asymptotic periodic setting}}

We assume that $A=A_{0}+A_{per}$ where $A_{0}\in L^{2}(\mathbb{R}%
^{d})^{d\times d}$ and $A_{per}\in L_{per}^{2}(Y)^{d\times d}$. We are here
in the framework of asymptotic periodic homogenization corresponding to $%
\mathcal{A}=\mathcal{B}_{\infty ,per}(\mathbb{R}^{d})=\mathcal{C}_{0}(%
\mathbb{R}^{d})\oplus \mathcal{C}_{per}(Y)$. In the proof of Lemma \ref{l1.2}%
, we showed that the corrector lies in $L_{\infty ,per}^{2}(Y)=L_{0}^{2}(%
\mathbb{R}^{d})+L_{per}^{2}(Y)$, which is nothing else but the space $B_{%
\mathcal{A}}^{2}(\mathbb{R}^{d})$ with $\mathcal{A}=\mathcal{B}_{\infty
,per}(\mathbb{R}^{d})$. So Theorem \ref{t3.2} applies to this setting.

\begin{remark}
\label{r3.1}\emph{Assume (i) }$A=A_{0}+A_{ap}$ \emph{with }$A_{0}\in 
\mathcal{C}_{0}(\mathbb{R}^{d})^{d\times d}$\emph{\ and }$A_{ap}\in AP(%
\mathbb{R}^{d})^{d\times d}$\emph{, (ii) the entries of }$A_{ap}$\emph{\
either are quasi-periodic and satisfy the frequency condition, or fulfill
the hypotheses of \cite[Theorem 1.1]{Armstrong}. we may use the same trick
as in Lemma \ref{l1.2} to show that the corrector lies, in each of these
cases, in }$B_{\infty ,AP}^{2}(\mathbb{R}^{d})=L_{0}^{2}(\mathbb{R}%
^{d})+B_{AP}^{2}(\mathbb{R}^{d})$\emph{. Therefore the conclusion of Theorem %
\ref{t3.2} holds true.}
\end{remark}

\subsection{Applications of Theorems \protect\ref{t1.4} and \protect\ref%
{t5.1}}

Here we give some concrete examples of functions for which Theorems \ref%
{t1.4} and \ref{t5.1} hold. Let $I_{d}$ denote the identity matrix in $%
\mathbb{R}^{d\times d}$.

\subsubsection{\textbf{The asymptotic periodic setting}}

We assume that $A=A_{0}+A_{per}$ where $A_{0}=b_{c}I_{d}$ with $%
b_{c}(y)=\exp (-c\left\vert y\right\vert ^{2})$ for any fixed $c>0$. $%
A_{per} $ is any continuous periodic symmetric matrix function satisfying
the ellipticity condition (\ref{2.2}). In the special $2$-dimension setting,
we may take $A_{0}=b_{1}I_{2}$ and 
\begin{equation*}
A_{per}=\left( 
\begin{array}{ll}
a_{1} & 0 \\ 
0 & a_{2}%
\end{array}%
\right) \text{ with }a_{1}(y)=4+\cos (2\pi y_{1})+\sin (2\pi y_{2}),\
a_{2}(y)=3+\cos (2\pi y_{1})+\cos (2\pi y_{2}).
\end{equation*}%
This special example is used for numerical tests in the next section.

\subsubsection{\textbf{The asymptotic almost periodic setting}}

As in the preceding subsection, we take $A_{0}=b_{c}I_{d}$ with $%
b_{c}(y)=\exp (-c\left\vert y\right\vert ^{2})$. We assume that $%
A=A_{0}+A_{ap}$ with $A_{ap}$ being any matrix with continuous almost
periodic entries such that $A$ satisfies hypothesis (\ref{1.2}). In the
special $2$-dimension setting used for numerical tests below, we take $%
A_{0}=b_{1}I_{2}$ and 
\begin{equation*}
A_{ap}=\left( 
\begin{array}{ll}
a_{1} & 0 \\ 
0 & a_{2}%
\end{array}%
\right) \text{ with }a_{1}(y)=4+\sin (2\pi y_{1})+\cos (\sqrt{2}\pi y_{2}),\
a_{2}(y)=3+\sin (\sqrt{3}\pi y_{1})+\cos (\pi y_{2}).
\end{equation*}

\section{Numerical simulations}

Our goal in this section is to check numerically the theoretical results
derived in the previous sections. We will consider the finite volume method
with two-point flux approximation. Of course multi-point flux approximation
can be considered when the matrix $A$ is non-diagonal. Even we will not
provide similar results for the discrete problem from numerical
approximation, similar results should normally be observed when the space
discretization step is small enough (fine grid) as the convergence of the
finite volume method for such elliptic problems is well known \cite{FV}.

\subsection{Finite volume methods}

The finite volume methods are widely applied when the differential equations
are in divergence form. To obtain a finite volume discretization, the domain 
$\Omega $ is subdivided into subdomains $(K_{i})_{i\in \mathcal{I}},\;%
\mathcal{I}$ being the corresponding set of indices, called control volumes
or control domains such that the collection of all those subdomains forms a
partition of $\Omega $. The common feature of all finite volume methods is
to integrate the equation over each control volume $K_{i},\;i\in \mathcal{I}$
and apply Gauss's divergence theorem to convert the volume integral to a
surface integral. An advantage of the two-point approximation is that it
provides monotonicity properties, under the form of a local maximum
principle. It is efficient and mostly used in industrial simulations. The
main drawback is that finite volume method with two-point approximation is
applicable in the so called admissible mesh \cite{FV, Antonio3} and not in a
general mesh. This drawback has been filled by finite volume methods with
multi-point flux approximations \cite{B, H} which allow to handle anisotropy
in more general geometries.

For illustration, we consider the problem find $u\in H_{0}^{1}(\Omega )$ 
\begin{equation}
-\nabla \cdot (A(x)\nabla u)=f\,\,\text{in}\,\,\,\Omega .  \label{pb1}
\end{equation}%
We assume that $f\in L^{2}(\Omega )$ and that $A$ is diagonal, so a
rectangular grid should be an admissible mesh \cite{FV, Antonio3}. Consider
an admissible mesh $\mathcal{T}$ with the corresponding control volume $%
(K_{i})_{i\in \mathcal{I}}$, we denote by $\mathcal{E}$ the set of edges of
control volumes of $\mathcal{T},\;\mathcal{E}_{int}$ the set of interior
edges of control volume of $\mathcal{T}$, $u_{i}$ the approximation of $u$
at the center (or at any point) of the control volume $K_{i}\in \mathcal{T}$
and $u_{\sigma }$ the approximation of $U$ at the center (or at any point)
of the edge $\sigma \in \mathcal{E}$. For a control volume $K_{i}\in 
\mathcal{T}$, we denote by $\mathcal{E}_{i}$ the set of edges of $K_{i}$, so
that $\partial K_{i}=\underset{\sigma \in \mathcal{E}_{i}}{\bigcup }\sigma $.

We integrate \eqref{pb1} over any control volume $K_{i}\in \mathcal{T}$, and
use the divergence theorem to convert the integral over $K_{i}$ to a surface
integral, 
\begin{equation*}
-\int_{\partial K_{i}}A(x)\nabla u\cdot \mathbf{n}_{i,\sigma
}ds=\int_{K_{i}}f(x)dx.
\end{equation*}%
To obtain the finite volume scheme with two-point approximation, the
following finite difference approximations are needed 
\begin{eqnarray}
\underset{\sigma \in \mathcal{E}_{i}}{\sum }F_{i,\sigma } &\approx
&\int_{\partial K_{i}}A(x)\nabla u\cdot \mathbf{n}_{i,\sigma }ds \\
F_{i,\sigma } &=&-\mathrm{meas}(\sigma )\;C_{i,\sigma }\dfrac{u_{\sigma
}-u_{i}}{d_{i,\sigma }} \\
C_{i,\sigma } &=&|C_{K_{i}}\,\mathbf{n}_{i,\sigma }|,\quad A_{K_{i}}=\dfrac{1%
}{\mathrm{meas}(K_{i})}\int_{K_{i}}A(x)dx
\end{eqnarray}%
Here \ $\mathbf{n}_{i,\sigma }$ is the normal unit vector to $\sigma $
outward to $K_{i}$, 
$\mathrm{meas}(\sigma )$ is the Lebesgue measure of the edge $\sigma \in 
\mathcal{E}_{i}$ and $d_{i,\sigma }$ the distance between the center of $%
K_{i}$ and the edge $\sigma $. Since the flux is continuous at the interface
of two control volumes $K_{i}$ and $K_{j}$ (denoted by $i\mid j$) we
therefore have $F_{i,\sigma }=-F_{j,\sigma }$ for $\sigma =i\mid j$\footnote{%
interface of the control volumes $K_{i}$ and $K_{j}$}, which yields 
\begin{equation*}
\left\{ 
\begin{array}{l}
F_{i,\sigma }=-\tau _{\sigma }\left( u_{j}-u_{i}\right) =-\dfrac{\mu
_{\sigma }\,\mathrm{meas}(\sigma )}{d_{i,j}}\left( u_{j}-u_{i}\right)
,\,\sigma =i\mid j\quad \newline
\\ 
\tau _{\sigma }=\mathrm{meas}(\sigma )\dfrac{C_{i,\sigma }C_{j,\sigma }}{%
C_{i,\sigma }d_{i,\sigma }+C_{j,\sigma }d_{j,\sigma }}\quad (\text{%
transmissibility through}\,\sigma )%
\end{array}%
\right.
\end{equation*}%
with 
\begin{equation*}
\mu _{\sigma }=d_{i,j}\dfrac{C_{i,\sigma }C_{j,\sigma }}{C_{i,\sigma
}d_{i,\sigma }+C_{j,\sigma }d_{j,\sigma }},
\end{equation*}%
where $d_{i,j}$ is the distance between the center of $K_{i}$ and center of $%
K_{j}$. We will set $d_{i,j}=d_{i,\sigma }$ for $\sigma =\mathcal{E}_{i}\cap
\partial \Omega $. For $\sigma \subset \partial \Omega $ ($\sigma \notin 
\mathcal{E}_{int}$ ), we also write 
\begin{eqnarray*}
F_{i,\sigma } &=&-\tau _{\sigma }\left( u_{\sigma }-u_{i}\right) \\
&=&-\dfrac{\mathrm{meas}(\sigma )\mu _{\sigma }}{d_{i,\sigma }}\left(
u_{\sigma }-u_{i}\right) .
\end{eqnarray*}%
The finite volume discretization is therefore given by 
\begin{eqnarray}
\underset{\sigma \in \mathcal{E}_{i}}{\sum }F_{i,\sigma } &=&f_{K_{i}}
\label{ode} \\
f_{K_{i}} &=&\int_{K_{i}}f(x)dx
\end{eqnarray}%
Let $h=~$size$(\mathcal{T})=\underset{i\in \mathcal{I}}{\sup }\underset{%
(x,y)\in K_{i}^{2}}{\sup }|x-y|$ be the maximum size of $\mathcal{T}$. We
set $u_{h}=(u_{i})_{i\in \mathcal{I}}$, $N_{h}=|\mathcal{I}|$ and $%
F=(f_{K_{i}})_{i\in \mathcal{I}}+bc$ , $bc$ being the contribution of the
boundary condition \footnote{%
Here $bc$ is null as we are looking for solution in $H_{0}^{1}(\Omega )$}.
Applying \eqref{ode} through all control volumes, the corresponding finite
volume scheme is given by 
\begin{equation}
A_{h}u_{h}=F,  \label{fv}
\end{equation}%
where $A_{h}$ is an $N_{h}\times N_{h}$ matrix. The structure of $A_{h}$
depends of the dimension $d$ and the geometrical shape of the control
volume. For diagonal $A$, if $\Omega $ is a rectangular or parallelepiped
domain, any rectangular grid ($d=2$) or parallelepiped grid ($d=3$) is an
admissible mesh and yields a 5-point scheme ($d=2$) or 7-point scheme ($d=3$%
) for the problem \eqref{pb1}. To solve efficiently the linear system %
\eqref{fv}, we have used the Matlab linear solver bicgstab with ILU(0)
preconditioners.

\subsection{Simulations in dimension 2}

\subsubsection{The Asymptotic periodic setting}

\label{ssect5} For the numerical tests, we consider problems (\ref{1.1}) and
(\ref{1.4}) in dimension $d=2$ with the finite volume method scheme %
\eqref{fv}. We denote by $I_{d}$ the square identity matrix in $\mathbb{R}%
^{d\times d}$. We take $A=A_{0}+A_{per}$ with 
\begin{align*}
A_{0}& =b_{0}I_{2}\text{ with }b_{0}(x_{1},x_{2})=\exp
(-(x_{1}^{2}+x_{2}^{2}))\text{ and } \\
A_{per}& =\left( 
\begin{array}{ll}
b_{1} & 0 \\ 
0 & b_{2}%
\end{array}%
\right) \text{ with }b_{1}=4+\cos (2\pi x_{1})+\sin (2\pi x_{2}),\
b_{2}=3+\cos (2\pi x_{1})+\cos (2\pi x_{2}).
\end{align*}%
The right-hand side function $f$ is given by $f=1$. 
The computational domain is $\Omega =(-1,1)^{2}$. 
We take $\varepsilon =1/N$ for some integer $N$. We will choose $N$ in the
set $\{2,3,4,5,6\}$.

The aim in this section is to compute numerically the \textquotedblright
exact solution\textquotedblright\ $u_{\varepsilon }$ (for a fixed $%
\varepsilon >0$) coming from the finite volume scheme with small $h$, and
compare it with its first order asymptotic periodic approximation $%
v_{\varepsilon }(x)=u_{0}(x_{1},x_{2})+\varepsilon \chi (\frac{x_{1}}{%
\varepsilon },\frac{x_{2}}{\varepsilon })\cdot \nabla u_{0}(x_{1},x_{2})$.

For this purpose, the strategy is carried out as follows:

\begin{enumerate}
\item We compute the exact solution of (\ref{1.1}) with our finite volume
scheme on a rectangular fine mesh of size $h>0$, with $h$ sufficiently small
to ensure that the discretization error is much smaller than $\varepsilon$,
which is the order of the error associated to the homogenization
approximation (see either Proposition \ref{p5.1} or Theorem \ref{t5.1}).

\item We compute the corrector functions $\chi_{1}$ and $\chi_{2}$
associated to the respective directions $e_{1}=(1,0)$ and $e_{2}=(0,1)$. To
this end, we rather consider their approximations by the finite volume
scheme \eqref{fv}, which are solutions to Eq. (\ref{3.3}), and we perform
this computation on the domain $Q_{6}=(-6,6)^{2}$ with Dirichlet boundary
conditions (as in (\ref{3.3})). We also compute their gradients $%
\nabla\chi_{1}$ and $\nabla\chi_{2}$. Here we take the mesh size $h=8\times
10^{-3}$ independent of $\varepsilon$.

\item With $\nabla \chi _{1}$ and $\nabla \chi _{2}$ computed as above, we
compute the homogenized matrix $A_{6}^{\ast }$ as in (\ref{eq5}), namely 
\begin{equation*}
A_{6}^{\ast }=\left( \frac{1}{12}\right) ^{2}\int_{Q_{6}}A(x)(I_{2}+\nabla
\chi (x))dx
\end{equation*}%
where here, $\chi =(\chi _{1},\chi _{2})$ so that $\nabla \chi $ is the
square matrix with entries $c_{ij}=\frac{\partial \chi _{j}}{\partial x_{i}}$%
.

\item With $A_{6}^{\ast }$ now being denoted by $A^{\ast }$, we compute the
exact solution $u_{0}$ of (\ref{1.4}).

\item Finally we compute the first order approximation $v_{%
\varepsilon}(x)=u_{0}(x)+\varepsilon \chi (x/\varepsilon )\cdot \nabla
u_{0}(x)$ and we compare it to the exact solution $u_{\varepsilon }$, which
has been computed at step 1.
\end{enumerate}

The goal is to check the convergence result in Theorem \ref{t5.1} given by %
\eqref{5.8}, but with the numerical solution using finite volume method.
Indeed we want to evaluate the following error 
\begin{equation}
Err(\varepsilon )=\dfrac{\Vert u_{\varepsilon }-u_{0}-\varepsilon \chi
^{\varepsilon }\nabla u_{0}\Vert _{H^{1}(\Omega )}}{\Vert u_{0}\Vert
_{H^{2}(\Omega )}}=\dfrac{\Vert u_{\varepsilon }-v_{\varepsilon }\Vert
_{H^{1}(\Omega )}}{\Vert u_{0}\Vert _{H^{2}(\Omega )}}.  \label{error}
\end{equation}%
As we already mentioned, $u_{0}$, $u_{\varepsilon }$ and $v_{\varepsilon }$
are computed numerical using the finite volume scheme for a fixed $h=8\times
10^{-3}$ independent of a fixed $\varepsilon $. All the norms involved in %
\eqref{error} are computed using their discrete forms \cite{FV, Antonio3}.
The coefficients of $A$ and $f$ are $\mathcal{C}^{\infty }(\Omega )$, so the
corresponding solutions $u_{0}$, $u_{\varepsilon }$ and $v_{\varepsilon }$
should be regular enough. Their graphs are given in Figure \ref{FIG03}. As
we can observe in Table \ref{phi}, the error decreases when $\varepsilon $
decreases, and therefore the convergence of $u_{\varepsilon }$ and $%
v_{\varepsilon }$ towards $u_{0}$ when $\varepsilon \rightarrow 0$ is
ensured. We can also observe that the corrector plays a key role as graph of 
$u_{\varepsilon }$ is close to the one of $v_{\varepsilon }$. The numerical
value of $A_{6}^{\ast }\equiv A^{\ast }$ obtained and used for $u_{0}$ and $%
v_{\varepsilon }$ is given by 
\begin{equation*}
A_{6}^{\ast }=\left( 
\begin{array}{ll}
3.895923 & 0.00001 \\ 
0 & 2.849959%
\end{array}%
\right) .
\end{equation*}

\begin{table}[h!]
\begin{center}
\begin{tabular}{|l|l|l|l|l|l|}
\hline
$1/\varepsilon $ & 2 & 3 & 4 & 5 & 6 \\ \hline
Err$(\varepsilon )$ & 0.5298 & 0.1382 & 0.0620 & 0.0577 & 0.0573 \\ \hline
\end{tabular}%
\end{center}
\caption{ $Err(\protect\varepsilon )$ with the corresponding $1/\protect%
\varepsilon $ for a fixed $h=2\times 10^{-3}$ independent of a fixed $%
\protect\varepsilon $.}
\label{phi}
\end{table}

\begin{figure}[h!]
\subfigure[]{
    \label{FIG03a}
    \includegraphics[width=0.45\textwidth]{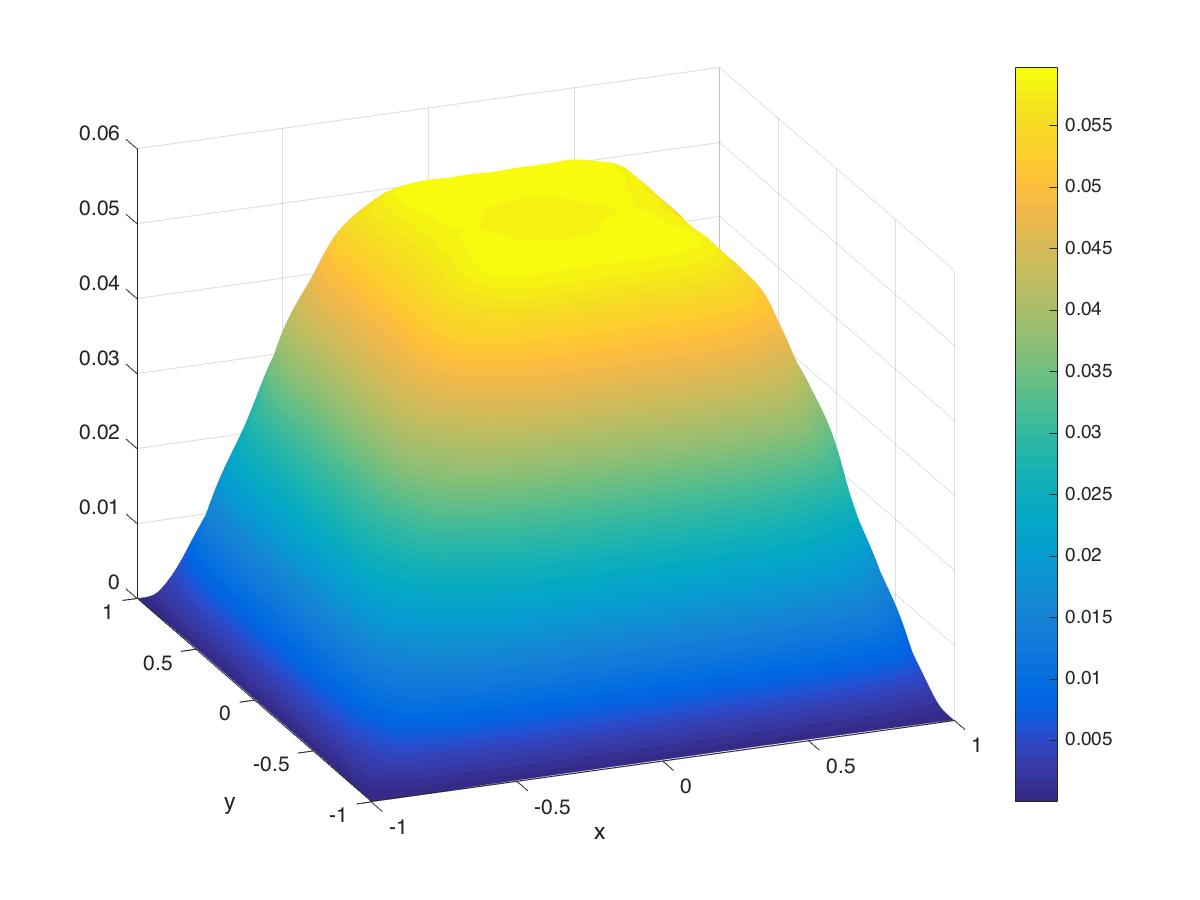}} \hskip 0.01\textwidth
\subfigure[]{
    \label{FIG03b}
    \includegraphics[width=0.45\textwidth]{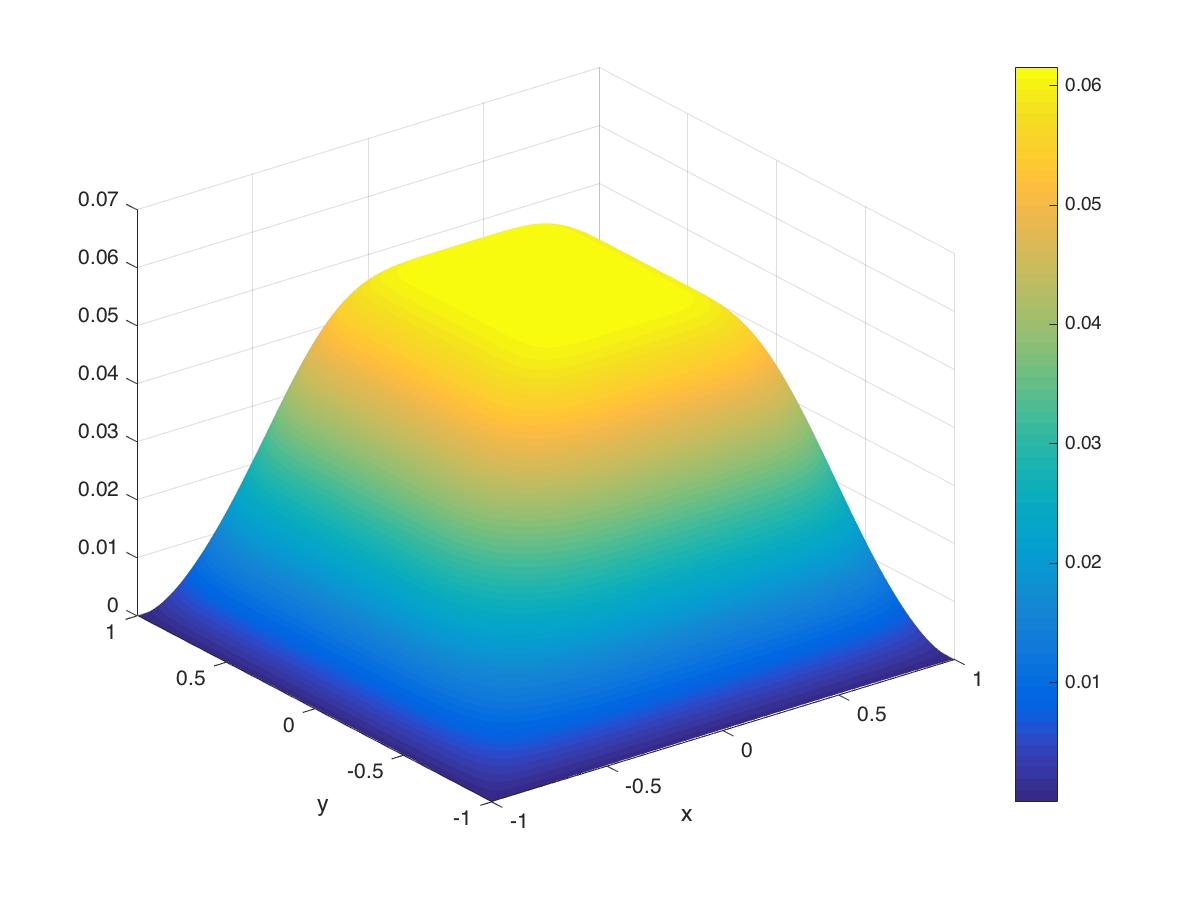}} \hskip 0.01\textwidth
\subfigure[]{
    \label{FIG03c}
    \includegraphics[width=0.45\textwidth]{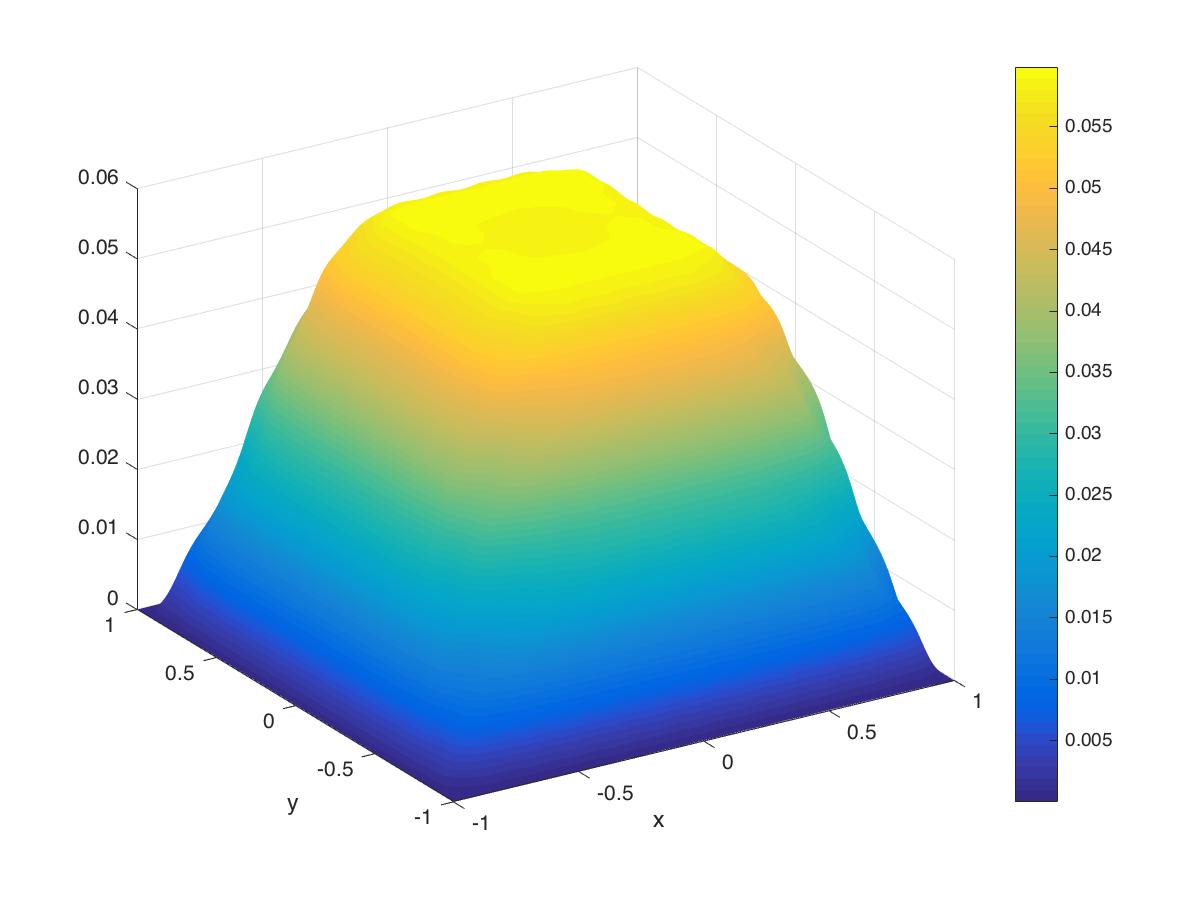}} \hskip 0.01\textwidth
\subfigure[]{
    \label{FIG03d}
    \includegraphics[width=0.45\textwidth]{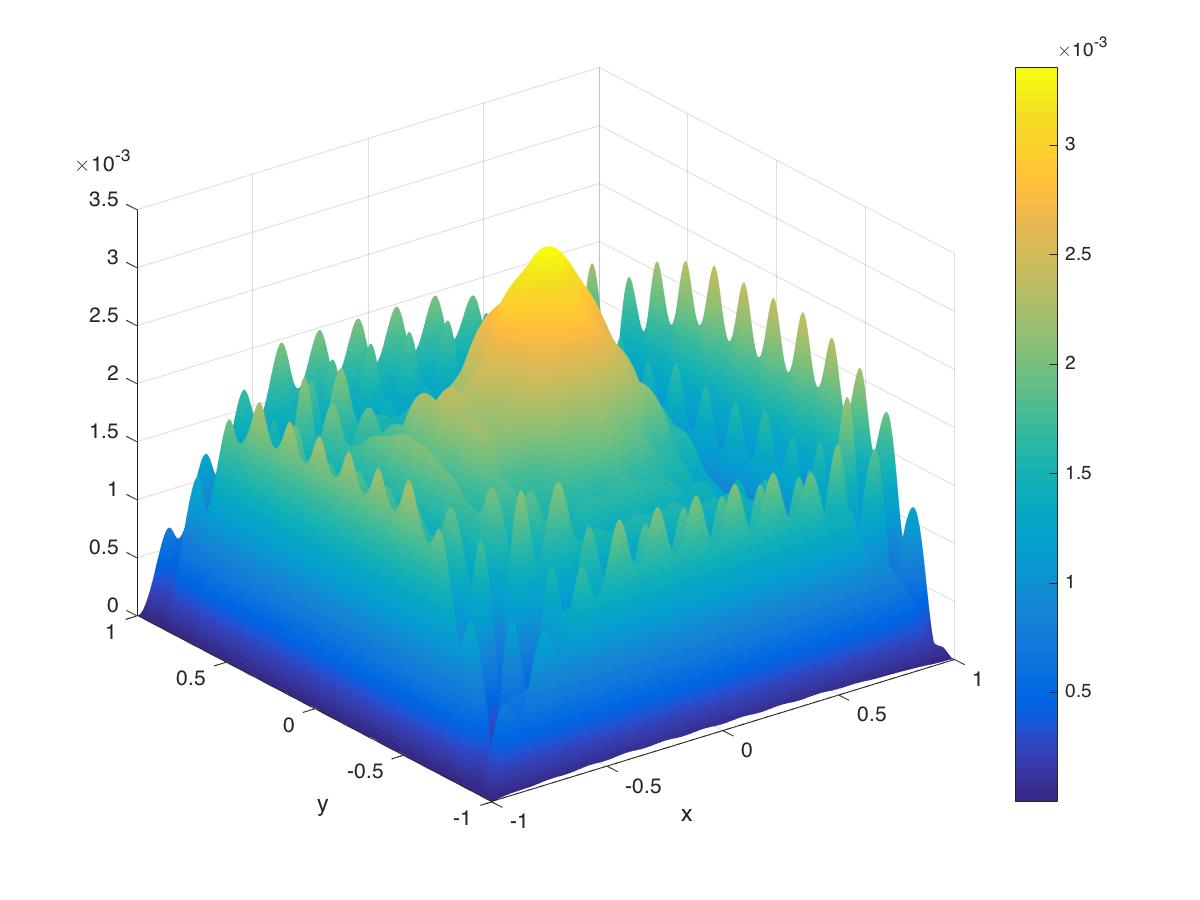}} 
\subfigure[]{
    \label{FIG03e}
    \includegraphics[width=0.45\textwidth]{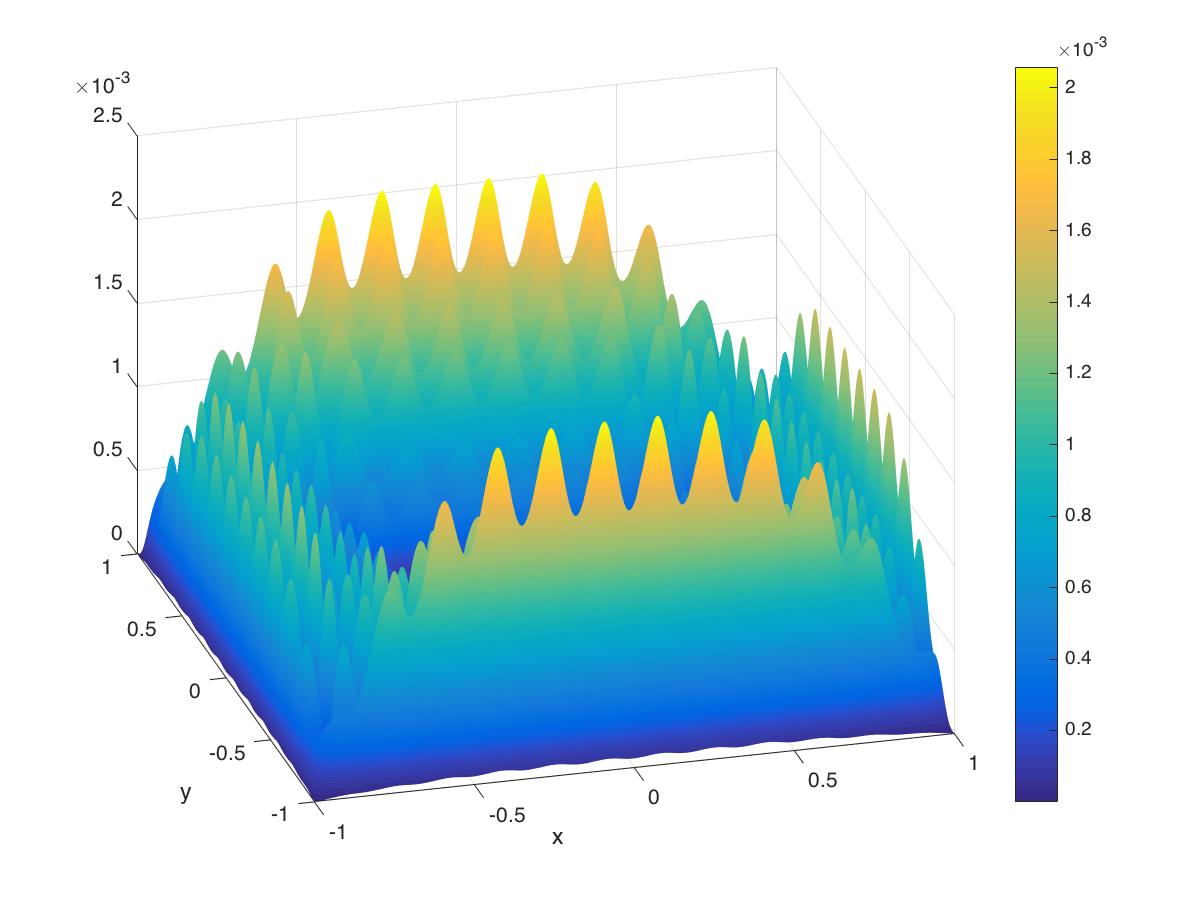}}
\caption{The graphs of $u_{\protect\varepsilon }$, $u_0$, $v_{\protect%
\varepsilon }$, $\vert u_{\protect\varepsilon}-u_0\vert $ and $\vert u_{%
\protect\varepsilon}-v_{\protect\varepsilon} \vert$ in the asymptotic
periodic setting, are shown in (a), (b), (c), (d) and (e) respectively for $%
\protect\varepsilon=1/6$ and $h=2\times10^{-3}$.}
\label{FIG03}
\end{figure}


\subsubsection{The asymptotic almost periodic setting}

Here we take $A=A_{0}+A_{ap}$ with 
\begin{align*}
A_{0}& =b_{0}I_{2}\text{ with }b_{0}(x_{1},x_{2})=\exp
(-(x_{1}^{2}+x_{2}^{2}))\text{ and } \\
A_{ap}& =\left( 
\begin{array}{ll}
b_{1} & 0 \\ 
0 & b_{2}%
\end{array}%
\right) \text{ with }b_{1}=4+\sin (2\pi x_{1})+\cos (\sqrt{2}\pi x_{2}),\
b_{2}=3+\sin (\sqrt{3}\pi x_{1})+\cos (\pi x_{2}).
\end{align*}%
The right-hand side function $f$ is given by $f(x_{1},x_{2})=\cos (\pi
x_{1})\cos (\sqrt{5}\pi x_{2})$. The computational domain is as above, that
is, $\Omega =(-1,1)^{2}$. We follow the same steps as above. The
corresponding value of $A_{6}^{\ast }$ is

\begin{equation*}
A_{6}^{\ast }=\left( 
\begin{array}{ll}
4.0118 & 0.0002 \\ 
0.0032 & 3.0206%
\end{array}%
\right) .
\end{equation*}%
We solve (\ref{1.4}) using finite volume method with multi-point flux
approximation \cite{B, H}. From Table \ref{phi2} and Figure \ref{FIG033}, we
can draw the same conclusion as in Section \ref{ssect5}. 
\begin{table}[h]
\begin{center}
\begin{tabular}{|l|l|l|l|l|l|}
\hline
$1/\varepsilon $ & 2 & 3 & 4 & 5 & 6 \\ \hline
Err$(\varepsilon )$ & 0.24 & 0.1520 & 0.1284 & 0.0768 & 0.0265 \\ \hline
\end{tabular}%
\end{center}
\caption{ $Err(\protect\varepsilon )$ with the corresponding $1/\protect%
\varepsilon $ for a fixed $h=2\times 10^{-3}$ independent of a fixed $%
\protect\varepsilon $.}
\label{phi2}
\end{table}

\begin{figure}[h!]
\subfigure[]{
    \label{FIG03a}
    \includegraphics[width=0.45\textwidth]{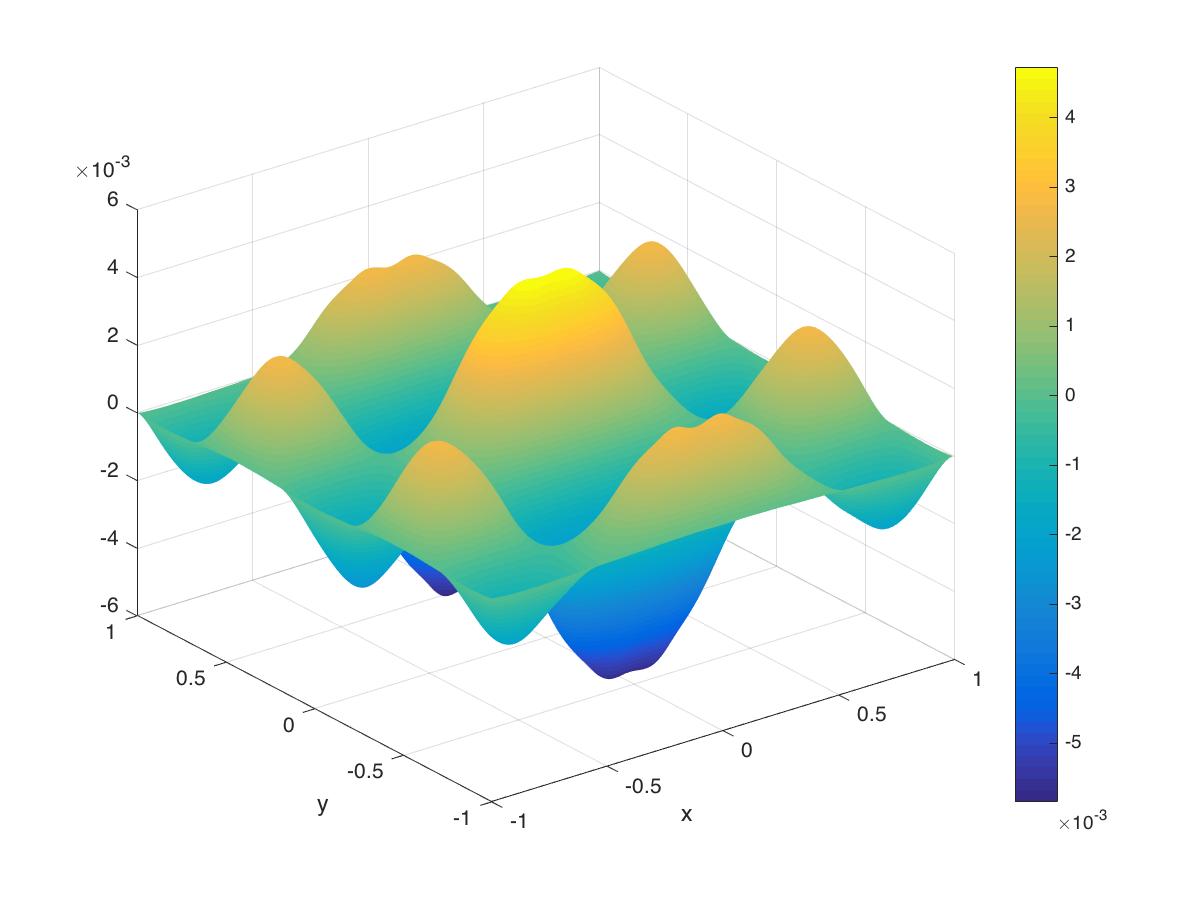}} \hskip 0.01\textwidth
\subfigure[]{
    \label{FIG03b}
    \includegraphics[width=0.45\textwidth]{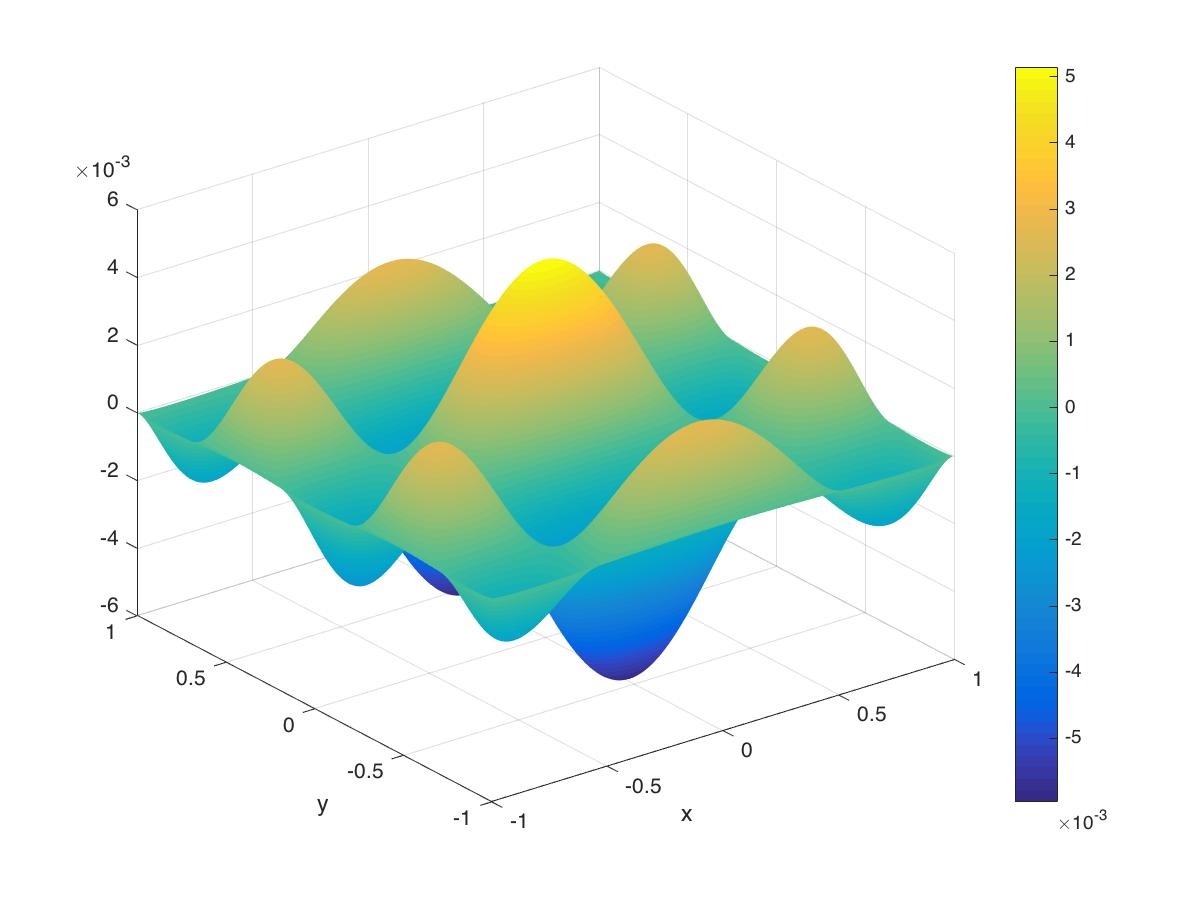}} \hskip 0.01\textwidth
\subfigure[]{
    \label{FIG03b}
    \includegraphics[width=0.45\textwidth]{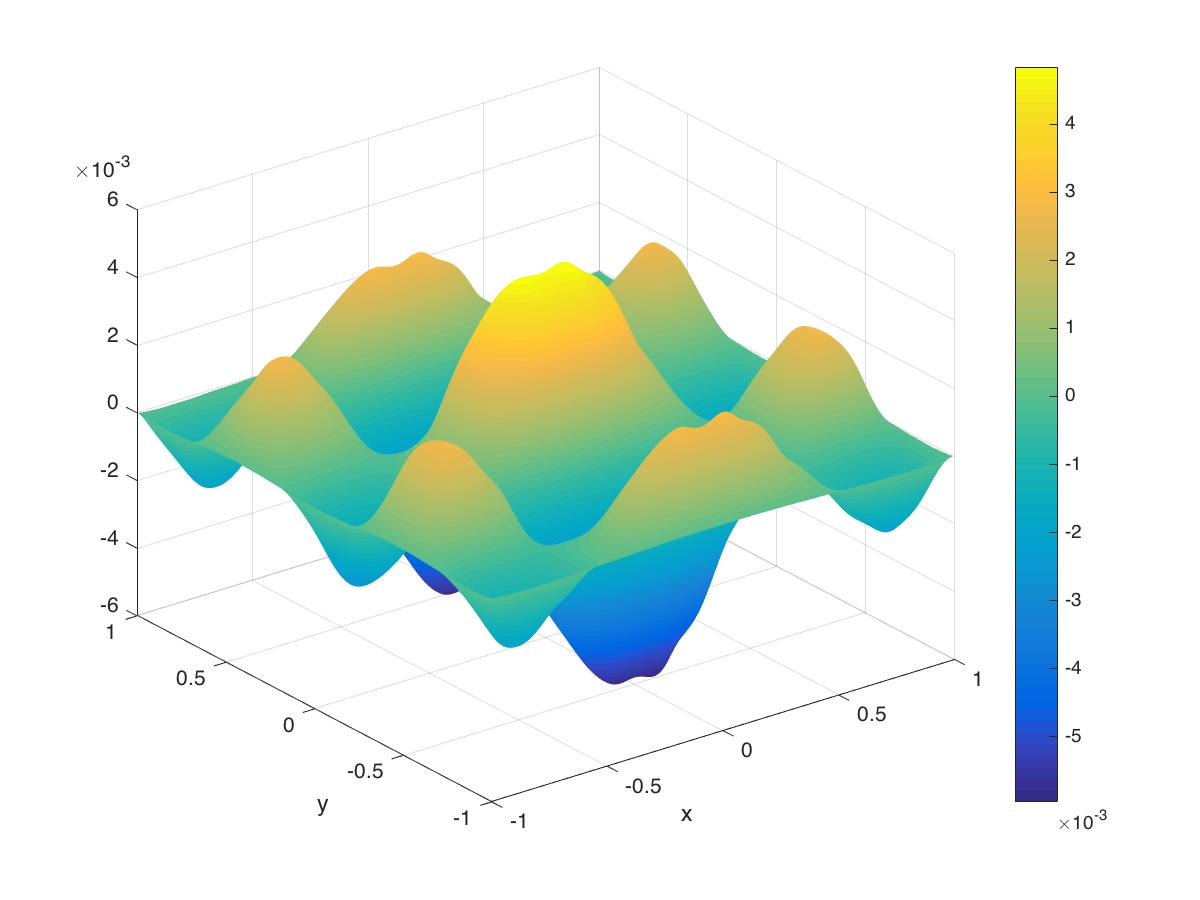}} \hskip 0.01\textwidth
\subfigure[]{
    \label{FIG03b}
    \includegraphics[width=0.45\textwidth]{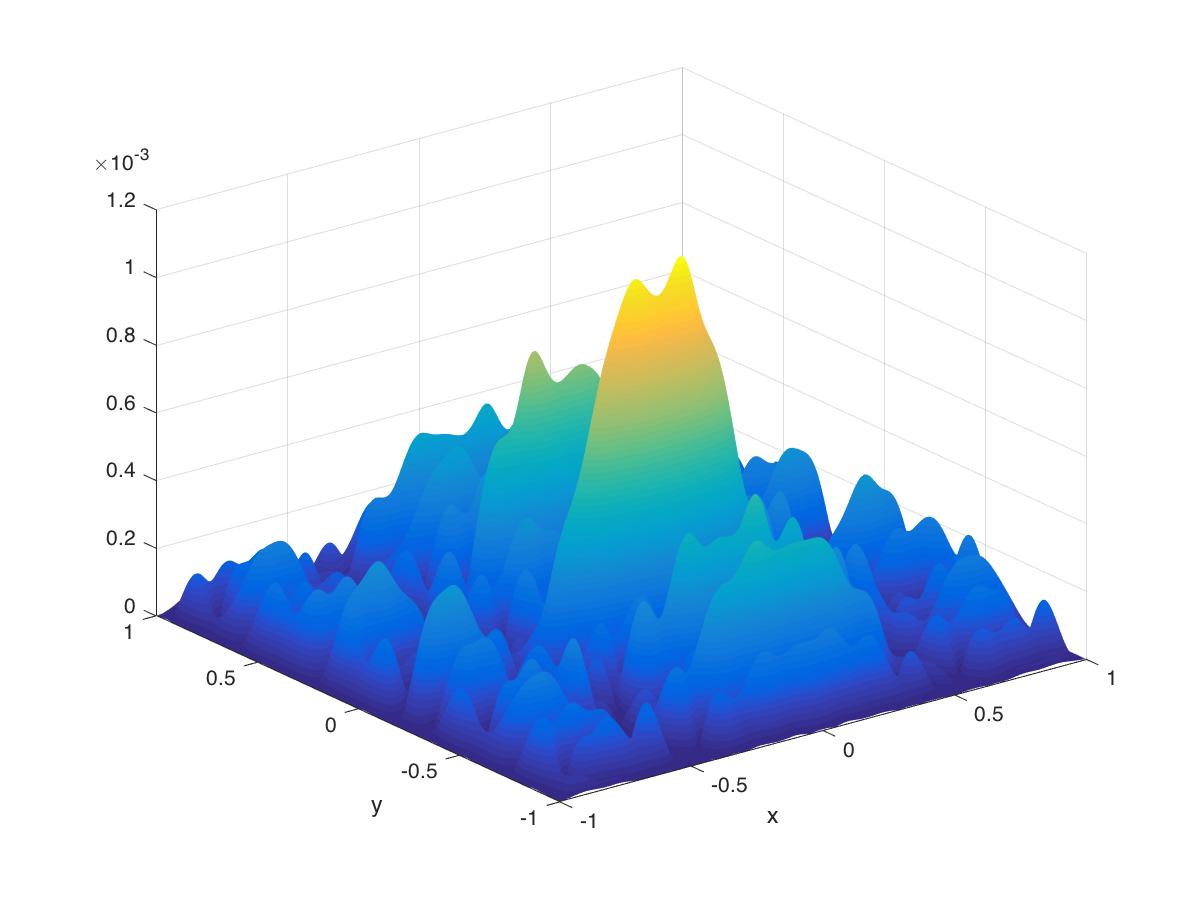}} 
\subfigure[]{
    \label{FIG03b}
    \includegraphics[width=0.45\textwidth]{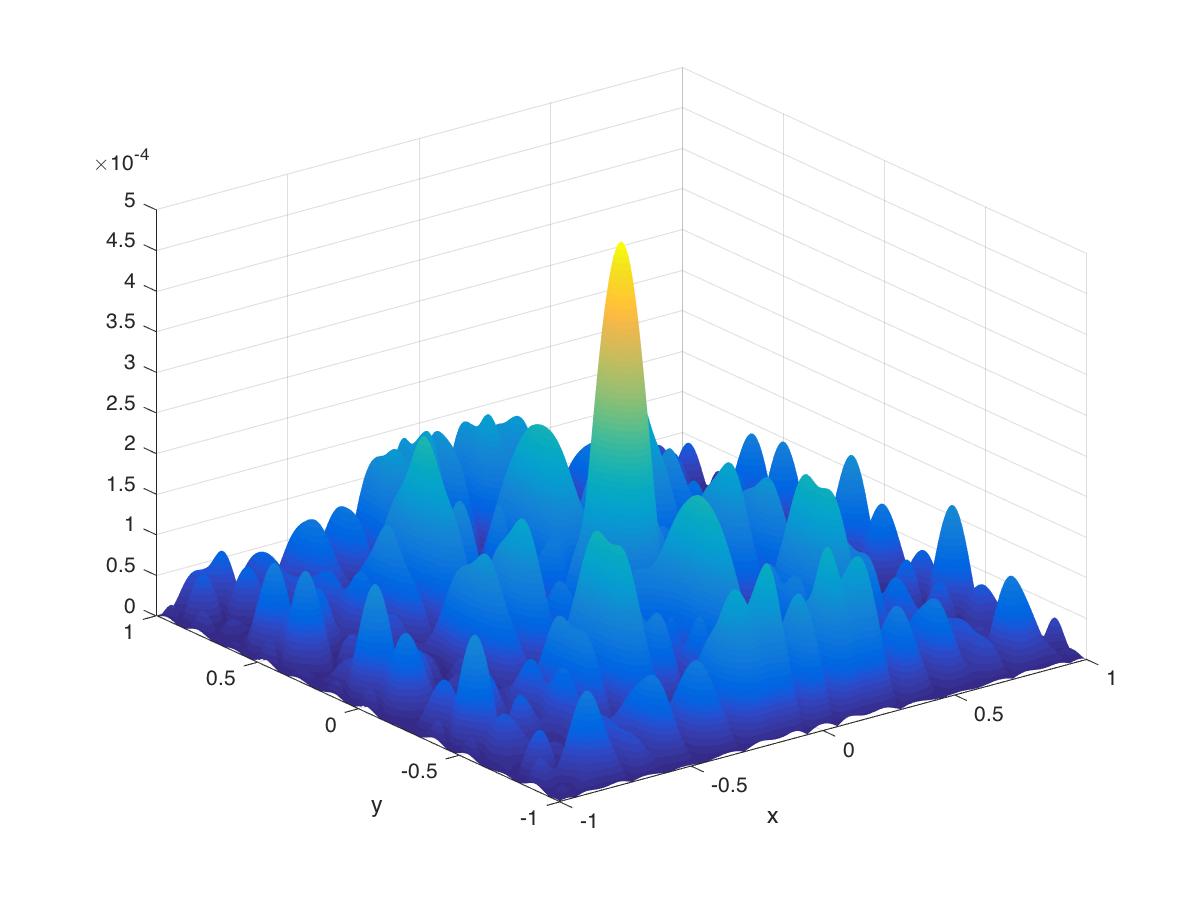}}
\caption{The graphs of $u_{\protect\varepsilon }$, $u_0$, $v_{\protect%
\varepsilon }$ $\vert u_{\protect\varepsilon}-u_0\vert $ and $\vert u_{%
\protect\varepsilon}-v_{\protect\varepsilon} \vert$ in the asymptotic almost
periodic setting, are shown in (a), (b), (c), (d) and (e) respectively for $%
\protect\varepsilon=1/6$ and $h=2\times10^{-3}$.}
\label{FIG033}
\end{figure}

\begin{acknowledgement}
\emph{The work of the second author has been supported by Robert Bosch
Stiftung through the AIMS ARETE chair programme (Grant No 11.5.8040.0033.0)
while the work of the third author has been carried out under the support of
the Alexander von Humboldt Foundation. They gratefully acknowledge the two
Foundations.}
\end{acknowledgement}

\textbf{Data availability}. Data sharing not applicable to this article as
no datasets were generated or analysed during the current study.

\textbf{Conflict of interest}. On behalf of all authors, the corresponding
author states that there is no conflict of interest.


\begin{thebibliography}{99}
\bibitem{B} I. Aavatsmark, 
\newblock{An introduction to multipoint flux
approximations for quadrilateral grids,} \newblock{Computational Geosciences}
\textbf{6} (2002) 405--432.

\bibitem{H} I. Aavastsmark, 
\newblock {Multipoint flux Approximation for
quadrilateral grids,} 
\newblock{9th International Forum on Reservoir
Simulation}, Abu Dhabi, 9--13 December 2007.

\bibitem{Adams} R.A. Adams, J.J.F. Fournier, Sobolev spaces, 2$^{nd}$
Edition, Pure and Applied Mathematics, Academic Press, 2003.

\bibitem{Armstrong} S. Armstrong, A. Gloria, T. Kuusi, Bounded correctors in
almost periodic homogenization, Arch. Rational Mech. Anal. \textbf{222}
(2016) 393--426.

\bibitem{Armstrong1} S. Armstrong, T. Kuusi, J.-C. Mourrat. Quantitative
Stochastic Homogenization and Large-Scale Regularity. Grundlehren der
mathematischen Wissenschaften (A Series of Comprehensive Studies in
Mathematics), Vol. \textbf{352} (2019), xxxviii+518p.

\bibitem{AS2016} S. Armstrong, Z. Shen, Lipschitz estimates in
almost-periodic homogenization, Comm. Pure Appl. Math. \textbf{69} (2016)
1882--1923.

\bibitem{AL87} M. Avellaneda, F.-H. Lin, Compactness methods in the theory
of homogenization, Comm. Pure Appl. Math. \textbf{40} (1987) 803--847.

\bibitem{BL76} J. Bergh, J. L\"{o}fstr\"{o}m, Interpolation spaces. An
introduction, Grundlehren der Mathematischen Wissenschaften, vol. 223,
Springer, Berlin, 1976.

\bibitem{Besicovitch} A.S. Besicovitch, Almost periodic functions,
Cambridge, Dover Publications, 1954.

\bibitem{Lebris} X. Blanc, C. Le Bris, P.-L. Lions, Local profiles for
elliptic problems at different scales: defects in, and interfaces between
periodic structures, Comm. Partial Differential Equations \textbf{40} (2015)
2173--2236.

\bibitem{BBMM05} A. Bondarenko, G. Bouchitte, L. Mascarenhas, R. Mahadevan,
Rate of convergence for correctors in almost periodic homogenization,
Discrete Contin. Dyn. Syst. \textbf{13} (2005) 503--514.

\bibitem{BP2004} A. Bourgeat, A.L. Piatnitski, Approximations of effective
coefficients in stochastic homogenization, Ann. Inst. H. Poincar\'{e}
Probab. Statist. \textbf{40} (2004) 153--165.

\bibitem{Casado} J. Casado Diaz, I. Gayte, The two-scale convergence method
applied to generalized Besicovitch spaces. Proc. R. Soc. Lond. A \textbf{458}
(2002) 2925--2946.

\bibitem{FV} R. Eymard, T. Gallouet, and R.~Herbin, \newblock Finite volume
methods, \newblock Updated preprint(2006) of the work appeared in: P.G.
Ciarlet, J.L. Lions (Eds.), \newblock Handbook of Numerical Analysis Volume
7, North-Holland, Amsterdam, 2000, pp. 713--1020.

\bibitem{Giaquinta} M. Giaquinta, Multiple integrals in the calculus of
variations and nonlinear elliptic systems, Annals of Math. Studies 105,
Princeton Univ. Press, Princeton, NJ, 1983.

\bibitem{GM} M. Giaquinta, L. Martinazzi, An Introduction to the Regularity
Theory for Elliptic Systems, Harmonic Maps and Minimal Graphs, Lecture notes
of the Scuola Normale Superiore Pisa, $2^{nd}$ Edition, 2012.

\bibitem{Gilbarg} D. Gilbarg, N.S. Trudinger, Elliptic Partial Differential
Equations of Second Order, Springer, 1998.

\bibitem{Gloria} A. Gloria, F. Otto, Quantitative results on the corrector
equation in stochastic homogenization, J. Eur. Math. Soc. (JEMS) \textbf{19}
(2017) 3489--3548.

\bibitem{GNO14} A. Gloria, S. Neukamm, F. Otto, An optimal quantitative
two-scale expansion in stochastic homogenization of discrete elliptic
equations, ESAIM\textit{\ }Math. Model. Numer. Anal\textit{.} \textbf{48}
(2014) 325--346.

\bibitem{GNO15} A. Gloria, S. Neukamm, F. Otto, Quantification of ergodicity
in stochastic homogenization: optimal bounds via spectral gap on Glauber
dynamics, Invent. Math. \textbf{199} (2015) 455--515.

\bibitem{Widman} M. Gr\"{u}ter, K.-O. Widman, The Green function for
uniformly elliptic equations, Manuscripta Math. \textbf{37} (1982) 303--342.

\bibitem{Jikov} V.V. Jikov, S.M. Kozlov, O.A. Oleinik, Homogenization of
differential operators and integral functionals, Springer-Verlag, Berlin,
1994.

\bibitem{Koz79} S.M. Kozlov, Averaging of differential operators with almost
periodic rapidly oscillating coefficients, Math. USSR Sbornik \textbf{35}
(1979) 481--498.

\bibitem{Hom1} G.\ Nguetseng, Homogenization structures and applications I,
Z.\ Anal. Anwen. \textbf{22} (2003) 73--107.

\bibitem{CMP} G. Nguetseng, M. Sango, J.L. Woukeng, Reiterated ergodic
algebras and applications, Commun. Math. Phys \textbf{300} (2010) 835--876.

\bibitem{Suslina} M. A. Pakhnin, A. Suslina, Operator error estimates for
homogenization of the elliptic Dirichlet problem in a bounded domain, St.
Petersburg Math. J. \textbf{24} (2013) 949--976.

\bibitem{Jikov2} S.E. Pastukhova, V.V. Zhikov, Homogenization estimates of
operator type for an elliptic equation with quasiperiodic coefficients,
Russ. J. Math. Phys. \textbf{22} (2015) 264--278.

\bibitem{Po-Yu89} A.V. Pozhidaev, V.V. Yurinskii, On the error of averaging
of symmetric elliptic systems, Izv. Akad. Nauk SSSR Ser. Mat. \textbf{53}
(1989) 851-867. In Russian; translated in Math. USSR Izv. \textbf{35} (1990)
183--201.

\bibitem{NA} M. Sango, N. Svanstedt, J.L. Woukeng, Generalized Besicovitch
spaces and application to deterministic homogenization, Nonlin. Anal. TMA 
\textbf{74} (2011) 351--379.

\bibitem{24'''} M. Sango, J.L. Woukeng, Stochastic sigma convergence and
applications, Dynamics of PDE \textbf{8} (2011) 261-310.

\bibitem{Shen} Z. Shen, Convergence rates and H\"{o}lder estimates in
almost-periodic homogenization of elliptic systems, Anal. PDE \textbf{8}
(2015) 1565--1601.

\bibitem{Shen1} Z. Shen, J. Zhuge, Approximate correctors and convergence
rates in almost-periodic homogenization, J. Math. Pures Appl. \textbf{110}
(2018) 187--238.

\bibitem{Antonio3} A. Tambue, An exponential integrator for finite volume
discretization of a reaction-advection-diffusion equation Comput. Math Appl. 
\textbf{71} (2016) 1875--1897.

\bibitem{Tar07} L. Tartar, An introduction to Sobolev spaces and
interpolation spaces, Lecture Notes of the Union of Matematica Italiana,
Springer, Berlin, 2007.

\bibitem{Wou014} J.L. Woukeng, Introverted algebras with mean value and
applications, Nonlinear Anal. TMA \textbf{99} (2014) 190--215.

\bibitem{Yu86} V.V. Yurinski\u{\i}, Averaging of symmetric diffusion in a
random medium (Russian), Sibirsk. Mat. Zh. \textbf{27} (1986) 167--180.

\bibitem{Zaidman} S. Zaidman, Almost-periodic functions in abstract spaces,
Pitman Advanced Pub. Program, Boston, 1985.

\bibitem{Zhikov1} V.V. Zhikov, S.E. Pastukhova, On operator estimates for
some problems in homogenization theory, Russian J. Math. Phys. \textbf{12}
(2005) 515--524.

\bibitem{Zhikov4} V.V. Zhikov, E.V. Krivenko, Homogenization of singularly
perturbed elliptic operators, Matem. Zametki, \textbf{33} (1983), 571-582
(english transl.: Math. Notes, \textbf{33} (1983), 294--300).
\end{thebibliography}
\end{document}